\documentclass[12pt]{article}

\usepackage[latin2,utf8]{inputenc}

\usepackage{amsmath,amssymb,amsthm, amsfonts}
\usepackage{verbatim}
\usepackage{graphicx}
\usepackage{arydshln}
\usepackage{xcolor}

 \numberwithin{equation}{section}
\numberwithin{figure}{section}
\numberwithin{table}{section}
 
\renewcommand{\div}{\mbox{\rm div}\;\!}

\newcommand{\bign}{\Big\|}
\newcommand{\de}{\partial}
\newcommand{\R}{\mathbb{R}}
\newcommand{\BA}{\mathbb{A}}
\newcommand{\BB}{\mathbb{B}}
\newcommand{\BC}{\mathbb{C}}
\newcommand{\BD}{\mathbb{D}}
\newcommand{\BH}{\mathbb{H}}
\newcommand{\BI}{\mathbb{I}}

\newcommand{\N}{\mathbb{N}}
\newcommand{\BR}{\mathbb{R}}
\newcommand{\BT}{\mathbb{T}}

\newcommand{\CA}{\mathcal{A}}

\newcommand{\CD}{\mathcal{D}}

\newcommand{\CF}{\mathcal{F}}
\newcommand{\cF}{\mathcal{F}}

\newcommand{\CL}{\mathcal{L}}
\newcommand{\CP}{\mathcal{P}}

\newcommand{\CR}{\mathcal{R}}
\newcommand{\CS}{\mathcal{S}}

\newcommand{\ba}{\boldsymbol{a}}

\newcommand{\bff}{\boldsymbol{f}}
\newcommand{\bF}{\boldsymbol{F}}
\newcommand{\bg}{\boldsymbol{g}}
\newcommand{\bG}{\boldsymbol{G}}
\newcommand{\bh}{\boldsymbol{h}}
\newcommand{\bH}{\boldsymbol{H}}

\newcommand{\bn}{\boldsymbol{n}}
\newcommand{\bu}{\boldsymbol{u}}

\newcommand{\bU}{\boldsymbol{U}}
\newcommand{\bv}{\boldsymbol{v}}
\newcommand{\bV}{\boldsymbol{V}}

\newcommand{\bW}{\boldsymbol{W}}
\newcommand{\tbW}{\widetilde \bW}
\newcommand{\bX}{\boldsymbol{X}}

\newcommand{\oq}{\bar q}
\newcommand{\oqq}{\bar q_0}

\newcommand{\fp}{\mathfrak{p}}
\newcommand{\fq}{\mathfrak{q}}

\newcommand{\sfQ}{\mathsf{Q}}
\newcommand{\sfP}{\mathsf{P}}

\newcommand{\pd}{\partial}

\newcommand{\HS}{\mathbb{R}^d_+}

\newcommand{\dB}{\dot{B}}
\newcommand{\dH}{\dot{H}}

\newcommand{\dW}{\dot{W}}
\newcommand{\tW}{\widetilde\bW}

\DeclareMathOperator{\dv}{{div}}

\renewcommand{\d}{\,\mathrm{d}}

\DeclareMathOperator{\DV}{{Div}}

\newcommand{\dtau}{\d \tau}
\newcommand{\Gdiv}{G_\mathrm{div}}

\newtheorem{thm}{Theorem}
\newtheorem{lem}{Lemma}
\newtheorem{cor}{Corollary}
\newtheorem{rmk}{Remark}

\begin{document}

\title{Free Boundary Problem for inhomogeneous Navier-Stokes equations}
\author{P.B. Mucha$^1$
, T. Piasecki$^1$, Y. Shibata$^2$}
\date{}
\maketitle

\begin{center}
$^1$ Institute of Applied Mathematics and Mechanics, University of Warsaw\\ 
Banacha 2, 02-097 Warsaw, Poland\\
emails: p.mucha@mimuw.edu.pl, t.piasecki@mimuw.edu.pl 
\vskip5mm
$^2$ Department of Mathematics and Research Institute of Science and Engineering,\\ Waseda University,\\ 
Ohkubo
3-4-1, Shinjuku-ku, Tokyo 169-8555, Japan\\
email: yshibata325@gmail.com
\end{center}


\begin{abstract}

We study free boundary problems for incompressible inhomogeneous flows governed by the Navier--Stokes equations, focusing on the regularity and global-in-time well-posedness of solutions in critical functional frameworks for small initial data.
We introduce a novel analytical framework for free boundary problems formulated as perturbations of the half-space. Our approach relies on the natural Lagrangian change of coordinates and a detailed analysis of the linearized problem (the Stokes system) in the maximal regularity regime, formulated in the Lebesgue spaces $L_p(0,T; L_q)$, including time-weighted variants. The main difficulty lies in the treatment of boundary terms, for which we apply a new technique based on complex interpolation to control nonlinear terms in fractional Sobolev spaces. This strategy also allows us to handle the case of variable density, which is not easily addressed by approaches based on Besov spaces.

Using this framework and real interpolation techniques, we construct also solutions in the Lorentz class $L_{p,1}(0,T; L_q)$ in time. The method further enables a rigorous study of the stability of equilibrium configurations. In particular, we resolve the problem in two spatial dimensions, where the interplay between geometry and regularity is especially subtle. Beyond these specific applications, the proposed approach provides a powerful tool for broader classes of nonlinear PDEs and further developments in maximal regularity theory.

\end{abstract}

\noindent
{\sc Key words:} Navier-Stokes equations, free boundary problems,  maximal regularity, complex and real interpolations, inhomogeneous incompressible flows, Lorentz spaces.

\newpage

\section{Introduction} 


Free boundary problems in fluid mechanics stand among the most formidable challenges in contemporary mathematical analysis. Their complexity lies in the necessity to simultaneously track the evolution of a moving interface and the dynamics of the fluid beneath it — a coupling that forces the analyst to marshal the full strength of modern transport theory, to invoke maximal regularity for parabolic systems, and to draw upon the most refined tools available in nonlinear function space theory.
At the heart of the difficulty is the very possibility of describing or controlling the free interface. Any sharp or high-fidelity identification of the boundary compels one to seek solutions of correspondingly high regularity, pushing existing analytical frameworks to their limits.
Within this vast landscape, our investigation focuses on regular solutions to the incompressible Navier–Stokes equations with free boundaries. Despite decades of intensive study, the field continues to harbour profound open questions: the precise mechanisms governing regularity, the subtle dependence on spatial dimension, and the persistent mysteries surrounding the well-posedness — or, in certain regimes, the inherent ill-posedness — of associated initial value problems.

Our main interest in this paper is the following free boundary problem describing the flow of inhomogenous, incompressible fluid by the Navier-Stokes equations
\begin{equation}
\label{inhomo-eq-original}
\left\{\begin{aligned}
\pd_t \rho + \bV \cdot \nabla \rho & =0 & \quad & \text{in $\Omega (t) \times (0,T)$}, \\
\rho( \pd_t \bV + (\bV \cdot \nabla) \bV) - \dv \BT(\bV,P) & = 0 & \quad & \text{in $\Omega (t) \times (0,T)$}, \\
\dv \bV & = 0 & \quad & \text{in $\Omega (t) \times (0,T)$}, \\
\BT(\bV,P) \bn & = 0 & \quad & \text{on $\pd \Omega (t) \times (0,T)$}, \\
V_{\pd \Omega (t)} & = \bV \cdot \bn & \quad & \text{on $\pd \Omega (t) \times (0,T)$}, \\
\rho|_{t=0}=\rho_0,\qquad \bV \vert_{t = 0} & = \bU_0& \quad & \text{in $\Omega_0$}, \\
\Omega(0) & = \Omega_0,
\end{aligned}\right.		
\end{equation}
Here, the unknowns are 
the density $\rho=\rho(x,t)$, the velocity field $\bV = \bV (x, t)$, 
the pressure $P = P (x, t)$, and the domain $\Omega = \Omega (t)$,
whereas an initial velocity field $\ba = \ba (x)$ is given.
{\color{black} Furthermore}, the stress tensor is given by 
$\BT(\bV,P)=\mu \BD (\bV) - P \BI$ with $\BD (\bV) = \nabla \bV + [\nabla \bV]^\top$, where $\mu > 0$ 
stands for the viscosity coefficient of the fluid, which we assume to be constant for simplicity, and $\BI$ is the $d \times d$ identity matrix. 
The atmospheric pressure $P_0$ is assumed to satisfy $P_0 \equiv 0$ 
without loss of generality (cf. \cite[$(3.2)$]{SS20}).
The outward unit normal vector to 
$\pd \Omega (t)$ is denoted by $\bn$ and the normal velocity of 
$\pd \Omega (t)$ is denoted by $V_{\pd \Omega (t)}$.
Let us briefly explain the meaning of each equation in \eqref{inhomo-eq-original}. 
The first equation is the continuity equation describing conservation of mass, the second is the momentum equation.  
The third equation is called the dynamic boundary condition and, physically, this equation states that the normal
stress is continuous as one passes through $\pd \Omega (t)$. 
The fourth equation in \eqref{eq-original} implies that 
the free surface is advected with the fluid, 
which is called the kinematic boundary condition. 
In other words, this boundary condition ensures that fluid particles 
do not cross the free surface $\pd \Omega (t)$.

The system is planned to be considered either in the half-space
in dimension $d \geq 2$, or in its small a perturbation $\Omega_0$ determined by a function $h:\R^{d-1} \to \R$ such that
\begin{equation}\label{omega:0}
    \Omega_0 = \{ x=(x',x_d) \in \R^d: x_d > h(x') \}.
\end{equation}
Perturbation in that case means that $h$ will be small in some suitable norm. Precise definition is given in Section \ref{sec:results}, where we formulate our main results.  

We are interested in showing global well-posedness of problem \eqref{inhomo-eq-original} in the framework of regular solutions.
The key idea is to reduce the above problem to the case of the homogeneous setting, where the density is just constant. For this purpose the density needs to be a perturbation of $\rho \equiv 1$, for simplicity. The second requirement is to have simple information about the regularity of the density. As the inhomogeneous Navier-Stokes equations \eqref{inhomo-eq-original} are often used to model so-called polluted fluids, the $L^\infty$ bound seems to be the most natural. Then there is a need of finding a suitable functional setting to realize these needs. We shall remember that the optimal language is given by the Besov spaces approach \cite{DM2012,EPS,OS}, but then there is a problem to find a reasonable set of admissible densities. By \cite{DM2009} we know that we need to work in some positive class of regularity of the density in the multiplication spaces of Besov type $\dot B^0_{d,1}$. Then it makes our set of initial data relatively small, since this class is essentially smaller than $L^\infty$. And here we need to return to the simpler language of the $L^p$ framework, relying on techniques based on the Maximal Regularity theory for parabolic-type systems in the $L_{p}(0,T;L_q(\Omega))$ norms. Then by the real interpolation theory one can generalize the methods on the spaces in the class of Lorentz spaces in time, ie. $L_{p,1}(0,T;L_q(\Omega))$, keeping the important property of possibility of multiplication by $L^\infty$ function. 
Interpolation of classical $L_p-L_q$ maximal regularity estimates allows for the initial velocity to be in the Besov space $B^s_{p,1}$, as observed first in \cite{DMT} (see also \cite{MP}).   
From that viewpoint we show that, in our case, simpler means more general.



Our idea is to concentrate first on the constant density case with $\rho \equiv 1$. Then, for $t > 0$, the equations governing motion of homogeneous incompressible fluid read
\begin{equation}
\label{eq-original}
\left\{\begin{aligned}
\pd_t \bV + (\bV \cdot \nabla) \bV - \dv \BT(\bV,P) & = 0 & \quad & \text{in $\Omega (t) \times (0,T)$}, \\
\dv \bV & = 0 & \quad & \text{in $\Omega (t) \times (0,T)$}, \\
\BT(\bV,P) \bn & = 0 & \quad & \text{on $\pd \Omega (t) \times (0,T)$}, \\
V_{\pd \Omega (t)} & = \bV \cdot \bn & \quad & \text{on $\pd \Omega (t) \times (0,T)$}, \\
\bV \vert_{t = 0} & = \bU_0& \quad & \text{in $\Omega(0)$}, \\
\Omega(0) & = \Omega_0,
\end{aligned}\right.		
\end{equation} 

This is a suitable place to present the historical development of the problem. We begin with the fundamental results in $L_2-L_2$ setting due to Beale (\cite{B81}, \cite{B84}) and Solonnikov (\cite{Sol77}, \cite{Sol88}). 
The first global existence results for more complex system describing the flow of compressible fluid (without divergence free condition) has been obtained in \cite{MN} and \cite{VZ}.   

The maximal regularity theory is a powerful tool in the framework of regular solutions. The seminal paper of Weis \cite{Weis} established a relation between so called $\CR$-boundedness property of associated resolvent problem and maximal regularity of time dependent problem.  
If zero is in the resolvent set, we are usually able to obtain even exponential decay for the resolvent problem, as observed in \cite{ES} in context of compressible Navier-Stokes equations in bounded domains (see also \cite{PSZ2} for generalization allowing to investigate a class of reaction-diffusion systems). Another, more direct approach to decay estimates to the inhomogemeous system \eqref{inhomo-eq-original} has been applied recently in \cite{DMP} to prove exponential stability of solutions constructed in \cite{DM2019}. A novel framework for maximal regularity approach to time-periodic problems based on $\CR$-boundedness has been developed recently in \cite{EKS1} and \cite{EKS2}.    

In the halfspace, however, zero is in the spectrum of the resolvent problem, which hinders the possibility of showing exponential decay. 
More direct, but technically demanding methods to show time decay of solutions are then required.  
Ukai \cite{Ukai} was the first to derive explicit solution formula for the Stokes problem in the half space. Later in \cite{CPS} this formula has been extended to case of nonzero external force, and well posedness of solutions with initial velocity in Besov spaces with negative index of regularity was shown. 
Decay estimates for the Stokes problem with surface tension in the half space has been shown in \cite{SaSh2016}, and applied to the nonlinear problem in \cite{SaSh2024}. A consistent framework for mathematical analysis of free boundary problems is described in the monograph \cite{PS_book}.
For results in exterior domains we refer among others to \cite{Iwashita}, \cite{MarSol}, \cite{ShibataExt1}, \cite{ShibataExt2}.

The classical incompressible Stokes and Navier-Stokes problems in half-space are nowadays quite well investigated, which led to many results for the incompressible system \eqref{eq-original}. Kozono \cite{Kozono89} proved the well posedness in class $C^1(L^n)$, where $n$ is the space dimension. An interestig general result showing generation of analytic semigroups on $BUC_{\sigma}(\HS),C_{0,\sigma}(\HS)$ and $L^\infty_\sigma(\HS)$ by the Stokes operator can be found in \cite{DHP}. 
The first $L_p-L_q$ maximal regularity result for the Stokes problem in the half-space was shown in \cite{SS12}, then the local well-posedness for \eqref{eq-original} has been proved by Shibata \cite{SS15}. This result is of particular importance for our work, as we rely on the linear results developed therein.   
In \cite{OS22}, the existence for the free boundary problem without surface tension for the incompressible system has been proved for exterior domain with noncompact boundary in space dimension $d \geq 4$, and in the half-space for $d \geq 3$. The paper focuses on nonlinear estimates, assuming appropriate decay for the linearized problem. For recent developments regarding maximal regularity and decay for two-phase problems with sharp interface we refer to \cite{OS25}, \cite{SSZ} and references therein. 

The maximal regularity theory allowed to develop techniques leading to more advanced results, as seen in the works Pr\"uss, Escher, and Mucha \cite{EPS,EM}.
The focus has not only been on determining for which systems the stability of static solutions can be established, but also on identifying the critical function spaces where such stability holds. This line of inquiry has led to the framework of Besov spaces, which allow for the treatment of solutions in the largest admissible functional setting. However, as noted, this approach introduces certain restrictions, particularly in the presence of density variations. In such cases, techniques developed by Shibata \cite{SS20} and Prüss \cite{DHP} appear to be the most suitable.

The classical maximal regularity theory is restricted to UMD spaces \cite{Weis}, which are in particular reflexive. Therefore, $L_1$ regularity in time is excluded. An important progress has been done independently by several authors in the $L_1$ theory recently. We shall mention here the result by Shibata and Watanabe \cite{SW}, who prove the $L_1(B^s_{q,1})$ regularity, where the Besov space $B^s_{q,1}$ is defined below, and the spatial domain in consideration is the half-space with space dimension $d \geq 2$. The proof is based on resolvent estimates in the above Besov space based on explicit solution formulae. An independent observation has been made in already mentioned paper \cite{DMT}, where the $L_1$ regularity has been obtained via application of Lorentz spaces in time. This new, promising direction is also the subject of the second part of our paper, as outlined above.      

The results of our paper concern the global in time stability of small solutions to systems (\ref{inhomo-eq-original}) and (\ref{eq-original}) with initial domain being close to the half-space with dimension $d\geq 2$. Regularity of initial data is explained in the Besov spaces 
given the main considerations in the spaces of type $L_p(0,T;L_q(\Omega(t))$.
The main result in this setting is Theorem \ref{thm:gwp}. The result can be generalized to the setting of Lorentz spaces 
$L_{p,1}(0,T;L_q(\Omega(t))$, then the part concerning the analysis of nonlinear terms becomes simpler, since we obtain the "sharp"
estimate of the key quantity
\begin{equation}
    \int_0^\infty \|\nabla \bV\|_{\infty} dt.
\end{equation}
Here the main result is stated in Theorem \ref{thm:Lorentz}.  


\smallskip 

Working in a critical functional framework creates a strong need to control nonlinear terms in suitable norms. When boundary regularity is involved, as in free boundary problems, it is not at all obvious that these nonlinear expressions are well posed. Their treatment usually demands careful and subtle estimates, often highly technical and sensitive to the behaviour of the boundary.
In this work, we take a different and more direct approach. We introduce a method for controlling such nonlinear quantities that relies on the complex interpolation techniques of \cite{BL}. This strategy provides a clean and unified way to handle nonlinearities in critical spaces, offering the necessary bounds while avoiding much of the technical complexity that appears in more classical methods.
The usefulness of the approach goes further: it is flexible, broadly applicable, and can be adapted to many related problems in fluid mechanics. Using this interpolation-based method, we are able to resolve the problem in the two-dimensional case, obtaining regularity and well-posedness in a setting where conventional tools often encounter serious obstacles.


\medskip

Let us conclude this introduction by summarizing the main contributions of the present work:

\smallskip 

– We develop a new analytical framework for free boundary problems formulated as perturbations of the half-space. The approach is robust and flexible, and it also covers the particularly delicate case of two spatial dimensions, where the interaction between geometry and regularity becomes especially subtle.

\smallskip 

– Our method allows us to study the stability of equilibrium configurations in the regime of regular and unique solutions. In particular, we provide a setting that rigorously justifies linearization around equilibrium states and captures their dynamical stability in a transparent and mathematically precise way.

\smallskip 

– We construct solutions in the maximal regularity class $L_{p,1}(0,T;L_q)$, a natural space for describing inhomogeneous incompressible flows with low-regularity initial data. A notable feature of our framework is that it accommodates fluid densities that are only small perturbations of a constant reference state in the $L_\infty$-norm, thereby extending the applicability of the analysis to a broader range of physical situations than previously addressed.

\smallskip 

– A central technical contribution of the paper is a new method for estimating nonlinear terms in fractional Sobolev spaces. The technique, based on complex interpolation, gives a systematic way to control fractional-order quantities through carefully arranged estimates on integer-order derivatives. The resulting bounds are sharp and well aligned with the nonlinear structure of the system, offering a useful tool not only for maximal regularity theory in free boundary problems, but also for other classes of nonlinear PDEs.



\subsection{Problem reformulation in Lagrangian coordinates}
It is well-known that the movement of the domain $\Omega (t)$ 
creates numerous mathematical difficulties.
Hence, as usual, we will transform the free boundary problem 
\eqref{eq-original} to a problem with a fixed interface
in order to avoid these difficulties, motivated by the 
pioneering {\color{black}works} due to Beale \cite{B81} and 
Solonnikov \cite{Sol77}.
Since we consider the system \textit{without} taking surface 
tension into account, following the idea due to Solonnikov \cite{Sol88} we may rely on the 
Lagrangian transformation to 
rewrite the problem to the fixed domain $\Omega_0$.  
More precisely, let $x = \bX(y,t) \in \Omega (t)$ be 
the solution to the Cauchy problem
\begin{equation}
\label{eq-Cauchy-Lagrange}
\frac{\d \bX}{\d t} = \bV (x(y,t), t) \quad \text{for $t > 0$}, \qquad x \vert_{t = 0} = y \in \HS.
\end{equation}
This means that $x = \bX(y,t)$ describes the position of 
fluid particle at time $t > 0$ which was located in 
$y \in \Omega_0$ at $t = 0$. Then the change of coordinates 
$(x, t) \rightsquigarrow (y, t)$ is said to be the 
\textit{Lagrangian transformation}. 
We note that for each $t > 0$ the kinematic condition for 
$\pd \Omega (t)$ (i.e., the fourth equation in \eqref{eq-original}) 
is automatically satisfied under the Lagrangian transformation. 
If $\bV (x, t)$ is  Lipschitz continuous with respect to $x$, 
then \eqref{eq-Cauchy-Lagrange} admits a unique solution given by
\begin{equation} \label{Lag:0}
\bX(y,t) = y + \int_0^t \bV (\bX (y, \tau), \tau) \d \tau, \qquad y \in \overline{\Omega_0}, \; t > 0,
\end{equation}
which represents the particle trajectory, due to the classical Picard-Lindel\"of theorem. 
We now set $\bW (y, t) := \bV (\bX(y,t), t)$ and $\sfQ (y, t) = P (\bX(y,t), t)$,
where $\bW (y, t)$ is so-called the \textit{Lagrangian velocity field}. 
Then we can rewrite \eqref{Lag:0} as 
\begin{equation}
\label{Lagrangian-transformation}
\bX(y, t) = y + \int_0^t \bW (y, \tau) \d \tau, \qquad y \in \overline{\Omega_0}, \; t > 0.
\end{equation}
Clearly, we have
$$
\Omega (t)  = \{x \in \BR^d \mid x = \bX_{\bW} (y, t), \enskip y \in \Omega_0 \}, \quad 
\pd \Omega (t)  = \{x \in \BR^d \mid x = \bX_{\bW} (y, t), \enskip y \in \pd \Omega_0 \}.
$$
Since the Jacobian matrix of the transformation $\bX_{\bW} (y, t)$ is given by
\begin{equation}
\nabla_y \bX_{\bW} (y, t) = \BI + \int_0^t \nabla_y \bW (y, \tau) \d \tau,
\end{equation}
the invertibility of $\bX_{\bW} (y, t)$ in \eqref{Lagrangian-transformation} 
is guaranteed for all $t \ge 0$ if $\bW$ satisfies
\begin{equation}
\label{integral-smallness}
\bigg\lVert \int_0^t \nabla \bW (\,\cdot\,, \tau) \d \tau \bigg\rVert_{L_\infty (\Omega_0)} \ll 1,
\end{equation}
which may be achieved by a Neumann-series argument. 
By virtue of \eqref{integral-smallness}, we may write
\begin{equation}
\label{def-Au}
\BA_{\bW} (y, t) := (\nabla_y \bX_{\bW} (y, t))^{- 1} 
= \sum_{l = 0}^{\infty} \bigg(- \int_0^t \nabla_y \bW (y, \tau) \d \tau \bigg)^l.
\end{equation}

\subsubsection{Reformulation of the incompressible system}
With the above notation, for $0 < T \le \infty$ problem \eqref{eq-original} in Lagrangian
coordinates reads
\begin{align}
\label{eq-fixed}
\left\{\begin{aligned}
\pd_t \bW - \dv (\mu \BD(\bW) - \sfP \BI) & = \bF (\bW) & \quad & \text{in $\Omega_0 \times {\color{black} (0, T)}$}, \\
\dv \bW & = \Gdiv (\bW) = \dv \bG (\bW) & \quad & \text{in $\Omega_0 \times {\color{black} (0, T)}$}, \\
(\mu \BD\bW) - \sfP \BI) \bn_0 & = \BH (\bW) \bn_0 & \quad & \text{on $\pd \Omega_0 \times {\color{black} (0, T)}$}, \\
\bW \vert_{t = 0} & = \bU_0 & \quad & \text{in $\Omega_0$},
\end{aligned}\right.		
\end{align}
where we have set $\bn_0 = (0, \ldots, 0, - 1)$. The right-hand side 
$\bF (\bW)$, $G_\mathrm{div} (\bW)$, $\bG (\bW)$, and $\BH (\bW)$
represent nonlinear terms given as follows:
\begin{equation}
\label{def-nonlinear-terms}
\begin{split}
\bF (\bW) & := 
\bigg(\int^t_0\nabla_y \bW\d\tau\bigg)
\Big(\pd_t \bW - \mu \Delta_y \bW  \Big) 
+ \mu \nabla_y \Big((\BA_{\bW}^\top - \BI) \colon \nabla_y \bW \Big)\\
&+ \mu \bigg(\BI+\int^t_0\nabla_y \bW\d\tau \bigg)\dv_y 
\Big( (\BA_{\bW} \BA_{\bW}^\top - \BI) \nabla_y \bW \Big), \\
G_\mathrm{div} (\bW) & := ( \BI - \BA_{\bW}^\top) \colon \nabla_y \bW, \\
\bG (\bW) & := (\BI - \BA_{\bW}) \bW, \\
\BH (\bW) & := 
\mu \bigg[ \bigg\{\nabla_y \bW + \bigg(\BI+\int^t_0\nabla_y \bW\d\tau \bigg)
[\nabla_y \bW]^\top \BA_{\bW} \bigg\} (\BI-\BA_{\bW}^\top)\\
&+ \bigg(\int_0^t \nabla_y \bW \d \tau \bigg) [\nabla_y \bW]^\top \BA_{\bW} 
+ [\nabla_y \bW]^\top (\BI-\BA_{\bW}) \bigg], \\
\BD_y (\bW) & := \nabla_y \bW + [\nabla_y \bW]^\top.	
\end{split}
\end{equation}
Recall that for $d \times d$ matrices 
$\BA = (A_{j, k})$ and $\BB = (B_{j, k})$, 
we write $\BA \colon \BB = \sum_{j, k}^d A_{j, k} B_{j ,k}$. 
The formulas \eqref{def-nonlinear-terms} have essentially been derived in 
\cite[Subsec.~3.3.2]{SS20}. For the sake of completeness we recall its derivation in Appendix \ref{sec-A}. 
It should be emphasized here that all nonlinear terms given in 
\eqref{def-nonlinear-terms} do not contain the pressure term $\sfQ$,
which is different from the formula derived by Solonnikov \cite{Sol88}. 
The advantage of our representation is that it is not necessary to use
the estimate of the pressure to construct solutions to \eqref{eq-fixed}, which means that
the estimate of the boundary trace of the pressure term is not necessary in our analysis.

\subsubsection{Reformulation of the nonhomogeneous system}

The key difference between systems \eqref{inhomo-eq-original} and \eqref{eq-original} is the continuity equation $(\ref{inhomo-eq-original})_1$. 
Applying the Lagrangian coordinates we introduce the new density
\begin{equation}
    \eta(y,t) = \rho (\bX(y,t),t).
\end{equation}
Then the equation $(\ref{inhomo-eq-original})_1$ takes the following simple form
\begin{equation}
\partial_t \eta(y,t)=0, \mbox{ \ thus \ } 
\eta(y,t)=\rho_0(y).
\end{equation}
It implies that in the $(y,t)$ coordinates the new velocity is constant in time, which implies that
the original density is constant along the streamlines.
Thus, the equations (\ref{inhomo-eq-original}), take the form
\begin{align}
\label{inhomo-eq-fixed}
\left\{\begin{aligned}
\pd_t \eta& =0& \quad & \text{in $\Omega_0 \times {\color{black} (0, T)}$}, \\
\rho_0 \pd_t \bW - \DV (\mu \BD(\bW) - \sfP \BI) & = \bF (\bW) & \quad & \text{in $\Omega_0 \times {\color{black} (0, T)}$}, \\
\dv \bW & = G_\mathrm{div} (\bW) = \dv \bG (\bW) & \quad & \text{in $\Omega_0 \times {\color{black} (0, T)}$}, \\
(\mu \BD\bW) - \sfP \BI) \bn_0 & = \BH (\bW) \bn_0 & \quad & \text{on $\pd \Omega_0 \times {\color{black} (0, T)}$}, \\
\bW \vert_{t = 0}, \eta|_{t=0}=\rho_0 & = \bU_0 & \quad & \text{in $\Omega_0$},
\end{aligned}\right.		
\end{align}
Then the system can be treated as a perturbation of (\ref{eq-fixed}) restated as follows
\begin{align}
\label{eq-fixed:2}
\left\{\begin{aligned}
\pd_t \bW - \DV (\mu \BD(\bW) - \sfP \BI) & = 
\bF (\bW) + (1-\rho_0)\pd_t \bW & \quad & \text{in $\Omega_0 \times {\color{black} (0, T)}$}, \\
\dv \bW & = G_\mathrm{div} (\bW) = \dv \bG (\bW) & \quad & \text{in $\Omega_0 \times {\color{black} (0, T)}$}, \\
(\mu \BD\bW) - \sfP \BI) \bn_0 & = \BH (\bW) \bn_0 & \quad & \text{on $\pd \Omega_0 \times {\color{black} (0, T)}$}, \\
\bW \vert_{t = 0} & = \bU_0 & \quad & \text{in $\Omega_0$},
\end{aligned}\right.		
\end{align}
The above form shows that the inhomogeneous version of the classical Navier-Stokes equations is indeed a small perturbation of the homogeneous system. We only need to control the perturbation $(1-\rho_0)\pd_t \bW$.

\subsection{Main results} \label{sec:results}
As outlined in the introduction, we investigate problems with initial domain being a half-space or its small perturbation. We are now in a position to formulate the smallness condition on the domain perturbation in the second case. Namely, we assume for $\epsilon>0$ small
\begin{equation}\label{omega}
    \Omega_0 = \{ x \in \R^d: x_d > h(x') \mbox{ \ with \ } x' \in \R^{d-1}: 
    \|h\|_{(W^{2-1/p}_{q_2} \cap W^{2-1/p}_{q_1})(\R^{d-1})} <\epsilon\}.
\end{equation}
Our first main result gives the global wellposedness in the maximal $L_p-L_q$ regularity setting. The results covers both inhomogeneous and homogeneous systems, and reads as follows. In order to formulate it, let us introduce a brief notation for the regularity class of the solutions. Namely, define 
\begin{equation} \label{sol:space}
\Xi = \left\{ v \in \bigcap_{i=1}^2 L_p(\BR_+; W^2_{q_i}(\Omega_0)^d) \cap W^1_p(\BR_+;
L_{q_i}(\Omega_0)^d): \quad tv \in \bigcap_{i=1}^2 \dot W^{2,1}_{q_i,p}  \right\},
\end{equation}
where the space $\dot W^{2,1}_{q,p}$ is defined in \eqref{def:space1} below.
\begin{thm}\label{thm:gwp}
Let $d \geq 2$ and assume $\epsilon$ is sufficiently small. Assume $\Omega_0=\R^d_+$, or it is given by \eqref{omega}. Assume moreover
\begin{equation}\label{require.1}
0 < \theta < \frac{1}{2}, \quad \frac1p < \frac{\theta}{2}
\end{equation}
and set 
\begin{equation}\label{exp:1}
\frac{1}{q_0} = \frac{1+2\theta}{d}, 
\quad \frac{1}{q_1} = \frac{1+\theta}{d}, \quad  \frac{1}{q_2} = \frac{\theta}{d}.
\end{equation} 
In case of inhomogeneous system (\ref{inhomo-eq-fixed})
we assume additionally that 
\begin{equation} \label{small}
\|\rho_0-1\|_{L_\infty(\Omega_0)} < \epsilon \quad {\rm and} \quad \|\rho_0-1\|_{L_{q_2}(\Omega_0)} < \epsilon. 
\end{equation}
Denote
\begin{equation} \label{def:ETW}
\begin{aligned}
E_T(\bW) = \sum_{i=1}^2\Big(\|(1+t)(\pd_t, \nabla^2)\bW\|_{L_p((0, T); L_{q_i}(\Omega_0))}
&+\|\bW\|_{L_p((0, T), L_{q_i}(\Omega_0))} \\
&+ \|\bW\|_{L_\infty((0, T), L_{q_i}(\Omega_0))}\Big).
\end{aligned}
\end{equation}
Then, there exists a small $\sigma > 0$ such that if initial data $\bU_0$ 
for equations  \eqref{eq-fixed}, respectively (\ref{inhomo-eq-fixed}), satisfies the smallness condition:
\begin{equation}\label{small:1}
\sum_{i=1}^2 \|\bU_0\|_{B^{2(1-1/p)}_{q_i, p}(\Omega_0)} + 
\|\bU_0\|_{L_{q_0}(\Omega_0)} \leq \sigma,
\end{equation}
then problem \eqref{eq-fixed}, respectively (\ref{inhomo-eq-fixed}), admits a unique solution $\bW \in \Xi$
satisfying the estimate: $E_{+\infty}(\bW) \leq c\sigma$ with some constant $c>0$.
\end{thm}
\begin{rmk}
Condition \eqref{exp:1} implies in particular
\begin{equation}\label{exp:2}
\dfrac{1}{q_0} = \dfrac{1}{q_1} + \dfrac{1}{q_2}, \quad 
d(\frac{1}{q_1}-\frac{1}{q_2}) = 1, \quad d(\frac{1}{q_0}-\frac{1}{q_1}) = \theta, \quad 
1 < q_0 < q_1 < d < q_2.
\end{equation}
The first relation allows to apply H\"older inequality, while the second is related to the Sobolev imbedding. We shall make extensive use of these relations.
\end{rmk}

Relying on Theorem \ref{thm:gwp} we proceed with proving our second main result, which gives global in time well posedness, still for the incompressible system, in the setting of Lorentz spaces. The result reads 

\begin{thm}\label{thm:Lorentz}
Let $d \geq 2$ and $\epsilon$ is sufficiently small. Assume $\Omega_0=\R^d_+$, or it is given by \eqref{omega}. Assume moreover that \eqref{require.1} and \eqref{exp:1} hold.
%
In case of inhomogeneous system (\ref{inhomo-eq-fixed})
we assume moreover \eqref{small}. 
Denote
\begin{equation} \label{def:ETW:2}
\begin{aligned}
E_{L,T}(\bW) = \sum_{j=1}^2\|\bW\|_{W^{2,1}_{q_j,(p,1)}(\Omega_0 \times (0,T))}+\|(1+t)(\pd_t, \nabla^2)\bW\|_{L_{p,1}((0, T); L_{q_1}(\Omega_0))},
\end{aligned}
\end{equation} 
where the space $W^{2,1}_{q,(p,1)}$ is defined in \eqref{def:space}.
Then, there exists a small $\sigma > 0$ such that if initial data $\bU_0$ 
for equations  \eqref{eq-fixed}, respectively (\ref{inhomo-eq-fixed}), satisfies the smallness condition: 
\begin{equation}\label{small:2}
\sum_{i=1}^2 \|\bU_0\|_{B^{2(1-1/p)}_{q_i, 1}(\Omega_0)} + 
\|\bU_0\|_{L_{q_0}(\Omega_0)} \leq \sigma,
\end{equation}
then problem \eqref{eq-fixed}, respectively (\ref{inhomo-eq-fixed}), admits a unique solution $\bW$ with
$$\bW \in \bigcap_{i=1}^2  \, W^{2,1}_{q_i,(p,1)}(\Omega_0 \times \BR_+)$$
satisfying the estimate: $E_{+\infty}(\bW) \leq c\sigma$ with some constant $c>0$.
\end{thm}

The final result is related to control of the nonlinear terms in low fractional spaces. It solves the problem with estimation of boundary terms in a specific Tribel-Lizorkin spaces.
Applying the complex interpolation one proves.

\begin{thm} \label{l:interp}
Let $f,g$ be a measurable functions.
Consider the operator 
\begin{equation}
    Sf = f B(\int_0^t g \,d\tau),
\end{equation}
where $B(0)=0$ and $B \in C^1$ in values of $g$. 
Assume that 
\begin{equation}
\frac{1}{r_0}=\frac{1}{r_1}+\frac{1}{r_2}
\end{equation}
and define $\frac{1}{r_1^+}=\frac{1}{2r_1}+\frac{1}{2r_2}$ and 
 \begin{equation} \label{c:20}
I_0(g):=\|g\|_{L_1(\R_+;L_{r_2})} +\|tg\|_{L_p(\HS;L_{r_2})},\quad 
I_1(g):=\|g\|_{L_1(\R_+;L_{r_2})} +\|tg\|_{L_p(\HS;L_{r_1})}.
\end{equation}
Then 
\begin{align} 
    & S:  \dot H^{1/2}_p(\R_+;L_{r_1}) \cap L_{2p}(\R_+,(\frac{dt}{t^{1/2}});L_{r_1}) \mapsto \dot H^{1/2}_p(\R_+;L_{r_0}),\\
    & S: \dot H^{1/2}_p(\R_+;L_{r_1}) \cap 
    L_{p}(\R_+,(\frac{dt}{t^{1/2}});L_{r_1^+}) \mapsto 
    \dot H^{1/2}_p(\R_+;L_{r_0})
\end{align}    
is the linear bounded operator, such that
\begin{align} 
    &\|Sf\|_{\dot H^{1/2}_p(\R_+;L_{r_0})}\leq
    C I_0(g) (\|f\|_{\dot H^{1/2}_p(\R_+;L_{r_1})} +
    \|f\|_{L_{2p}(\R_+,(\frac{dt}{t^{1/2}});L_{r_1})}),\label{half:1}\\[3pt]
    &\|Sf\|_{\dot H^{1/2}_p(\R_+;L_{r_0})}\leq
    C I_1(g) (\|f\|_{\dot H^{1/2}_p(\R_+;L_{r_1})} +
    \|f\|_{L_{p}(\R_+,(\frac{dt}{t^{1/2}});L_{r_1^+})}).  \label{half:2}
\end{align}
\end{thm}

The proof of the above theorem one finds in Section 6. Note that we are interested in the inequalities (\ref{half:1}) and (\ref{half:2}). Next, we introduce notations used in our article.

\subsection{Functional spaces} In this part we introduce the functional spaces which we are using in our considerations. 
We use standard notation $L_p(\Omega)$ and $W^k_p(\Omega)$ for, respectively, Lebesgue and Sobolev spaces. By $\dot W^k_p(\Omega)$ we denote a homogeneous Sobolev space equipped with the seminorm 
\begin{equation}
\|f\|_{\dot W^k_p} = \sum_{|\alpha|=k} \|D^\alpha f\|_{L_p}.
\end{equation}
Moreover, for a spatial domain $\Omega$, time interval $I$, $m,n \in \N$ and $1\leq p,q\leq \infty$ we denote 
\begin{equation} \label{def:wmn}
W^{m,n}_{q,p}(\Omega \times I)=\{u:\; \nabla^k u \in L_p(\R_+(L_q(\Omega)))\; {\rm for}\; k\leq m, \;\; \de^l_t u \in L_p(\R_+(L_q(\Omega))) \;{\rm for}\, l \leq n \},
\end{equation}
where $\nabla^0 u:=u$. We also denote 
\begin{equation} \label{def:sp}
\hat W^1_{q,0}(\Omega) = \{ f \in L_{q, {\rm loc}}(\Omega) \mid
\nabla f \in L_q(\Omega)^d, \quad f|_{x_d=0} = 0\}.
\end{equation}
and $\dot W^{-1}_q(\Omega) = \hat W^1_{q, 0}(\Omega)^*$. 
Furthermore, we introduce the solenoidal space $J^q(\Omega)$, which is defined by
$$J_q(\Omega) = \{\bu \in L_q(\Omega) \mid (\bu, \nabla\varphi) = 0
\quad\text{for every $\varphi \in \hat W^1_{q',0}$}.$$

{\bf Lorentz spaces} are defined on a measure space $(X,\mu)$ by real interpolation as 
\begin{equation*}
L_{p,r}(X,\mu)=(L_\infty(X,\mu),L_1(X,\mu))_{\frac{1}{p},r} \quad {\rm for} \quad 
p\in(1,\infty), \; r \in [1,\infty]
\end{equation*}
In a direct way the spaces are defined as spaces of functions for which the norm 
\begin{equation*}
    \|f\|_{L_{p,r}}=\left\{
    \begin{array}{lr}
   \displaystyle p^{1/r} \left( \int_0^\infty (s|\{|f|>s\}|^{1/p})^r \frac{ds}{s}\right)^{1/r}     &  \mbox{ for \ } r<\infty, \\[14pt]
    \displaystyle \sup_{s>0} s|\{ |f|>s\}|^{1/p}     &  \mbox{ for \ } r=\infty,
    \end{array}
    \right. 
\end{equation*}
where $|\cdot|$ denotes $\mu(\cdot)$, is finite. The
factor $p^{1/r}$ ensures that $\|f\|_{L_{p,p}}=\|f\|_{L_p}$.

Let us recall several properties of Lorentz spaces which can be found in \cite[Chapter 4.4]{BS}, see also \cite{G}.  
For brevity we write $L_p, L_{p,r}$ instead of, respectively, $L_p(X,\mu), L_{p,r}(X,\mu)$. 
\begin{enumerate}
    \item Imbedding: \begin{equation}\label{imbed}
    L_{p,p}=L_p, \;\;\; L_{p,r_1} \subset L_{p,r_2} \quad {\rm for} \quad r_1 \leq r_2.
    \end{equation}
    \item H\"older inequality: for $1<p,p_1,p_2<\infty$ and $1\leq r,r_1,r_2\leq\infty$ such that $\frac{1}{p}=\frac{1}{p_1}+\frac{1}{p_2}$ and $\frac{1}{r}=\frac{1}{r_1}+\frac{1}{r_2}$ we have   
    \begin{equation}\label{holder}
\|fg\|_{L_{p,r}}\leq \|f\|_{L_{p_1,r_1}}\|g\|_{L_{p_2,r_2}}
    \end{equation}
    Inequality holds also for $(p_1,r_1)=(1,1)$. 
    \item Generalized H\"older inequality: for $1<p,p_1,\ldots p_n$,
    $1\leq r,r_1,\ldots,r_n$ such that 
    $$
    \frac{1}{p}=\sum_{k=1}^n \frac{1}{p_k}, \quad 
    \frac{1}{r}=\sum_{k=1}^n \frac{1}{r_k}
    $$
    it holds 
    \begin{equation} \label{holder2}
        \|\Pi_{k=1}^n f_k\|_{L_{p,r}} \leq \Pi_{i=1}^k \|f_k\|_{L_{p_k,r_k}}.     
    \end{equation}
    
    \item If $f \in L_\infty(\Omega)$ and $g \in L_{p,q}(\Omega)$ then $fg \in L_{p,q}(\Omega)$ for all $p,q \in [1,\infty]$ and 
    \begin{equation} \label{0:1}
        \|fg\|_{L_{p,q}}\leq \|f\|_\infty \|g\|_{L_{p,q}}.
    \end{equation}
    \item For $\alpha>0$, $p,q,p\alpha,q\alpha \in [1,\infty]$
    \begin{equation} \label{0:2}
        \|f^\alpha\|_{L_{p,q}} \sim \|f\|_{L_{p\alpha,q\alpha}}^\alpha.
    \end{equation}
\end{enumerate}

{\bf Homogeneous and nonhomogeneous Besov spaces} can be defined by real interpolation of Sobolev spaces as, respectively, 
\begin{equation*}
\dot B^s_{q,p}(\Omega)=(L_q(\Omega);\dot W^1_q(\Omega))_{s,p}, \qquad    
B^s_{q,p}(\Omega)=(L_q(\Omega);W^1_q(\Omega))_{s,p}.
\end{equation*}
We recall two important properties of Besov spaces, the proofs of which can be found in \cite{D-book}.
The first one is that Besov spaces are an interpolation family, i.e.
\begin{equation} \label{0:3a}
    \dot B^s_{q,r}=(\dot B^{s_1}_{q_1,p_1},\dot B^{s_2}_{q_2,p_2})_{\beta,r}, \qquad
     B^s_{q,r}=(B^{s_1}_{q_1,p_1},B^{s_2}_{q_2,p_2})_{\beta,r}, 
\end{equation}
with $\beta \in (0,1)$, $s=(1-\beta)s_1+\beta s_2$ and $s_1,s_2 \in \R$, $p_1,p_2,r,q,q_1,q_2\in [1,\infty]$. This property implies the interpolation inequality 
\begin{equation} \label{int:besov}
    \|f\|_{\dot B^{s}_{q,r}} \leq C\|f\|_{\dot B^{s_1}_{q_1,p_1}}^{1-\beta} \|f\|_{\dot B^{s_2}_{q_2,p_2}}^\beta
\end{equation}
for parameters as above. The second one is the imbedding theorem
\begin{equation} \label{imbed:besov}
    \dot B^s_{q_1,r} \subset \dot B^{s-d(\frac{1}{q_1}-\frac{1}{q_2})}_{q_2,r}, \mbox{ \ \ for \ }  s \in \R,\; q_1,q_2,r \in [1,\infty],
\end{equation}
where $d$ is the space dimension. In particular, 
\begin{equation} \label{imbed:besov2}
\dot B^{d/p}_{q,1} \subset \dot B^0_{\infty,1} \subset L_\infty \mbox{ \ \ for \ } q,r \in [1,\infty]
\end{equation}
We shall denote 
\begin{equation} \label{frac:Sobolev}
W^s_q:=B^s_{q,q}, \quad \dot W^s_q=\dot B^s_{q,q}.
\end{equation}

We also recall the classical imbedding \cite{DHP} which is of crucial importance in the $L_p-L_q$ maximal regularity theory: 
\begin{equation} \label{para}
\|u\|_{L_\infty((0,T);B^{2-{2/p}}_{q,p}(\Omega))} \leq C \left( \|\de_t u,\nabla^2 u\|_{L_p((0,T);L_q(\Omega))}  + \|u(0)\|_{B^{2-{2/p}}_{q,p}(\Omega)} \right).
\end{equation}

We also recall the definition of $X$-valued Bessel potential space where $X$ is a Banach space. Namely, for $s\in (0,1)$ 
we define 
$$
H^s_p(\R,X)=\{f:\R\to X: \; \|f\|_{H^s_p(\R,X)}<\infty \},
$$
where 
$$
\|f\|_{H^s_p(\R,X)}= \bign \cF^{-1}\big[ (1+\tau^2)^{\frac{s}{2}}\cF[f](\tau) \big] \bign_{L_p(\R,X)}.
$$
Then we define the Bessel potential space on $\R_+$ by extension, as a space of functions $f:\R_+ \to X$ for which the norm 
$$
\|f\|_{H^s_p(\R_+,X)}:=\inf \{\|g\|_{H^s_p(\R,X)}: \; g \in H^s_p(\R,X), g|_{\R_+}=f \} <\infty.
$$
Bessel potential spaces can be also defined on a time interval $I$ by complex interpolation: 
\begin{equation} \label{complex:2}
H^s_p(I,X) = (L_p(I;X),W^1_p(I;X))_s.
\end{equation}
Next, by real interpolation we define
\begin{equation} \label{def:Hp1}
H^{1/2}_{p,1}(I,X)=\left( H^{1/2}_{p_1}(I,X),H^{1/2}_{p_2}(I,X) \right)_{\sigma,1}, \quad \frac{1}{p}=\frac{\sigma}{p_1}+\frac{1-\sigma}{p_2}
\end{equation}
In order to describe the regularity of the boundary data we set
\begin{equation} \label{def:lambda}
\Lambda^{1/2}\bH = \CF^{-1}[(ik)^{1/2}\CF[\bH](k)], 
\end{equation}
where $\CF$ and $\CF^{-1}$ denote, respectively, the Fourier transform with respect to time variable and the inverse formula with respect to dual variable $k$. 
We have 
\begin{equation} \label{half}
\|\Lambda^{1/2}\bH\|_{L_p(\BR, L_q(\HS))} \leq C \|\bH\|_{\dot H^{1/2}_p(\BR, L_q(\HS))}.
\end{equation}
With \eqref{frac:Sobolev} and \eqref{complex:2} we can generalize the definition \eqref{def:wmn}; namely, for $r,s \in \R$ and $1 \leq p,q \leq \infty$ we define 
\begin{equation} \label{def:space1}
W^{r,s}_{q,p}(\Omega \times I) = L_p(I;W^r_q(\Omega)) \cap H^s_p(I;L_q(\Omega)).
\end{equation}
A general imbedding theorem \cite{BIN} implies in particular 
\begin{equation} \label{imbed:3}
\|\nabla v\|_{W^{1,1/2}_{q,p}} \leq C \|v\|_{W^{2,1}_{q,p}}, \quad
\|\nabla v\|_{\dot W^{1,1/2}_{q,p}} \leq C \|v\|_{\dot W^{2,1}_{q,p}}
\end{equation}
and 
\begin{equation} \label{imbed:4}
\|\nabla v\|_{W^{1,1/2}_{q,(p,1)}} \leq C \|v\|_{W^{2,1}_{q,(p,1)}}, \quad \|\nabla v\|_{\dot W^{1,1/2}_{q,(p,1)}} \leq C \|v\|_{\dot W^{2,1}_{q,(p,1)}}.
\end{equation}
%
%
Next, for $\Omega \subset \R^d$ we denote 
\begin{equation} \label{def:space}
W^{2,1}_{q,(p,r)}(\Omega \times \R_+):= \{ u \in C_b(\R_+; B^{2-2/p}_{q,r}(\Omega)): \; \de_t u,\nabla^2 u \in L_{p,r}(\R_+;L_q(\Omega) \}
\end{equation}
with the norm
\begin{equation*}
    \|u\|_{W^{2,1}_{q,(p,r)}(\Omega\times \R_+)} :=
    \|u\|_{L_\infty(\R_+;\dot B^{2-2/p}_{q,r}(\R^d_+))} +
    \|\partial_t z, \nabla^2 z\|_{L_{p,r}(\R_+;L_q(\Omega))}.
\end{equation*}
Finally, for $I \subset \R_+, \; \alpha \in \R$ and $f:(0,T) \to X$ where $X$ is a Banach space we denote 
\begin{equation} \label{def:lpdt}
\|f\|_{L_p(I;t^\alpha\,dt,X)} = \|t^\alpha f\|_{L_p(I,X)}.
\end{equation}

By complex interpolation \cite{Tr} we have 
\begin{equation} \label{complex:int}
\begin{aligned}
&\left( L_{p_1}(I,t^\alpha dt;L_{q_1}),L_{p_1}(I,t^\beta dt;L_{q_2}) \right)_{1/2} = L_{p}(I,t^\gamma dt;L_{q}),\\[3pt]
&\frac{1}{p}=\frac{1}{2p_1}+\frac{1}{2p_2}, \quad \frac{1}{q}=\frac{1}{2q_1}+\frac{1}{2q_2}, \quad \gamma=\frac{\alpha}{2}+\frac{\beta}{2}.
\end{aligned}
\end{equation}

{\bf Outline of the paper.} The paper unfolds as follows. In section \ref{sec:lin:lplq} we prove $L_p-L_q$ maximal regularity estimates in the half space for the linear problem related to \eqref{eq-fixed}. We also prove necessary time-weighted $L_p-L_q$ estimates. 
In section \ref{sec:lin:pert} we extend these results to the linearized problem in the perturbed halfspace.
Next, in Section \ref{sec:nonlin:lplq} 
we prove the estimates for the nonlinearities arising on the right hand side of \eqref{eq-fixed}. Here the main difficulty lies in the proof of necessary estimate in fractional Sobolev space for the nonlinearity related to the boundary term. The nonlinear estimates allow to complete the proof of Theorem \ref{thm:gwp}. In section \ref{sec:Lorentz:lin} we prove estimate for the linearized problem in Lorentz spaces. Next, in Section \ref{sec:nonlin:Lorentz} we derive Lorentz spaces estimates for the nonlinearities arising in \eqref{eq-fixed} and complete the proof of Theorem \ref{thm:Lorentz}.

\section{Linear $L_p-L_q$ theory in the half space} \label{sec:lin:lplq}

In this section we consider time dependent Stokes problem corresponding to linearization of \eqref{eq-fixed}:
\begin{equation}\label{eq:1}\left\{\begin{aligned}
\pd_t\bW - \dv\BT(\bW, \sfP) = \bF&&\quad&\text{in $\HS\times\BR_+$}, \\
\dv\bW = \Gdiv = \dv\bG&&\quad&\text{in $\HS\times\BR_+$}, \\
\BT(\bW, \sfP) \bn_0=\bH &&\quad&\text{on $\pd\HS\times\BR_+$}, \\
\bW|_{t=0} = \bU_0&&\quad&\text{in $\HS$},
\end{aligned}\right.
\end{equation}
where $\bn_0 = (0, \ldots, 0, -1)^\top$, where we denote 
$$
\HS =\{x=(x_1, \ldots, x_d) \in \BR^d \mid x_d > 0\},  \quad
\pd\HS = \{x = (x_1, \ldots, x_d) \in \BR^d \mid x_d=0\}. 
$$
Since the term $(1-\rho_0)\de_t \bW$ on the RHS of \eqref{inhomo-eq-fixed} can be regarded as a small perturbation, system \ref{eq:1} can be also used as linearization of \eqref{inhomo-eq-fixed}.
The first main result of this section concerns maximal $L_p-L_q$ regularity 
for \eqref{eq:1}.
\begin{thm} \label{thm:max}
Let  $1 < p, \bar q < \infty$.  Let $\bU_0 \in \dot B^{(1-1/p)}_{\oq,p}(\HS)$ and 
\begin{align*}
\bF &\in L_p(\BR_+, L_{\oq}(\HS)^d), \quad 
\Gdiv \in L_p(\BR_+; \dW^1_{\oq}(\HS)), \quad \bG \in \dW^1_p(\BR_+; L_{\oq}(\HS)^d), \\
 \bH &\in L_p(\BR_+; \dW^1_{\oq}(\HS)^d) \cap \dH^{1/2}_p(\BR_+; L_{\oq}(\HS)^d).
\end{align*}
Then, problem \eqref{eq:1} admits unique solutions $\bW$ and $\sfP$ with
$$\bW \in L_p(\BR_+; \dot W^2_{\oq}(\HS)^d) \cap \dot W^1_p(\BR_+, L_{\oq}(\HS)^d)),
\quad \nabla \sfP \in L_p(\BR_+; L_{\oq}(\HS)^d)$$
satisfying the estimate:
\begin{equation} \label{lin:global}
\begin{aligned} 
&\|(\pd_t\bW, \nabla^2\bW, \nabla \sfP)\|_{L_p(\BR_+; L_{\oq}(\HS))} \\
&\quad \leq C(\|\bU_0\|_{\dot B^{2(1-1/p)}_{\oq,p}(\HS)} 
+ \|(\bF, \nabla \Gdiv, \pd_t\bG, \nabla \bH, \Lambda^{1/2}\bH\|_{L_p(\BR_+; L_{\oq}(\HS))})\\
&\quad \leq C(\|\bU_0\|_{\dot B^{2(1-1/p)}_{\oq,p}(\HS)} 
+ \|(\bF, \nabla \Gdiv, \pd_t\bG, \nabla \bH\|_{L_p(\BR_+; L_{\oq}(\HS))})
+ \|\bH\|_{\dot H^{1/2}_p(\BR_+; L_{\oq}(\HS))} .
\end{aligned}
\end{equation}
Moreover, if 
\begin{equation} \label{ic:zero}
F|_{t=0} =0, \Gdiv|_{t=0}=0,
\bG|_{t=0}=0 \;\;{\rm and}\;\; \bH|_{t=0}=0,
\end{equation}
then for any $T>0$ we have a local in time estimate
\begin{equation}\label{local.2}
\begin{aligned}
&\|(\pd_t, \nabla^2) \bW\|_{L_p((0, T); L_{\oq}(\HS))}
\leq  C(\|\bU_0\|_{B^{2(1-1/p)}_{\oq,p}(\HS)} \\
&\qquad +\|(\bF, \nabla \Gdiv, \pd_t \bG, \nabla \bH)\|_{L_p((0, T); L_{\oq}(\HS))}
+ \|\bH\|_{\dot H^{1/2}_p((0, T); L_{\oq}(\HS))}.
\end{aligned}
\end{equation}
Here, the constant $C$ is independent of $T$.
\end{thm} 
The above theorem belongs nowadays to the classical results. They can be  found in \cite{SS12}. For the reader's convenience we recall the main ideas in the Appendix \ref{A1}.

The second result gives time-weighted maximal regularity for \eqref{eq:1}.
In order to formulate it, for $\Omega \subset \R^d$ we define
\begin{equation}\label{right.1}
\begin{aligned}
\CF^1_q(T,\Omega)&= \|(1+t)(\bF, \nabla \Gdiv, \pd_t \bG, \nabla \bH)\|_{L_p((0, T); L_q(\Omega))} + \|(1+t)\bH\|_{\dot H^{1/2}_p((0, T); L_q(\Omega))}\\
&+ \|\bG\|_{L_p(0, T); L_q(\Omega))},\\[3pt]
\CF^2_q(T,\Omega) & =\|(1+t)(\bF, \nabla \Gdiv, \pd_t \bG, \nabla \bH)\|_{L_p((0, T); L_{q}(\Omega))}
+ \|(1+t)\bH\|_{\dot H^{1/2}_p((0, T); L_q(\Omega))}.
\end{aligned}\end{equation}
Then we have
\begin{thm}\label{thm:weight} Let $1 <p,  \oq < \infty$.
Let $1<T\leq +\infty$ and let $\bW$ and $\sfP$ be solutions 
to equations \eqref{eq:1}.  Then, there holds
\begin{equation} \label{weight.global.1}
\|\bW\|_{L_\infty((0, T), L_{\oq}(\HS))}  \leq C(\|\bU_0\|_{L_{\oq}(\HS)}
+ \CF^2_{\oq}(T,\HS)).
\end{equation}
Moreover, let 
\begin{equation} \label{c:weight}
1 < \oqq <\oq, \quad \bar \theta = d(1/{\oqq}-1/{\oq}), \quad 1/p < \bar \theta/2.  
\end{equation}
Then, there holds
\begin{align} \label{weight.global.2}
&\|(1+t)(\pd_t, \nabla^2)\bW\|_{L_p((0, T); L_{\oq}(\HS))}
\leq C\big( \|\bU_0\|_{B^{2(1-1/p)}_{\oq,p}(\HS)} +\|\bU_0\|_{L_{\oqq}(\HS)} \\
& \hskip9cm
+ \CF^1_{\oq}(T,\HS) + \CF^2_{\oqq}(T,\HS) \big), \nonumber \\[3pt]
&\|\bW\|_{L_p((0, T); L_{\oq}(\HS))} \label{weight.global.3}
\leq C\big(\|\bU_0\|_{L_{\oq}(\HS)} + \|\bU_0\|_{L_{\oqq}(\HS)} + \CF^2_{\oq}(T,\HS) + \CF^2_{\oqq}(T,\HS) \big).
\end{align}
\end{thm}

\subsubsection*{Weighted estimates - proof of Theorem \ref{thm:weight}}

In order to prove Theorem \ref{thm:weight} we consider
$(1+t)\bW$ and $(1+t)\sfP$.  Then , for any $T \geq 1$ we have
\begin{equation}\label{weight:1}\left\{\begin{aligned}
\pd_t(1+t)\bW - \dv\BT((1+t)\bW,  (1+t)\sfP) = (1+t)\bF- \bW
&&\quad&\text{in $\HS\times(0, T)$}, \\
\dv(1+t)\bW = (1+t)\Gdiv = \dv((1+t)\bG)&&\quad&\text{in $\HS\times(0, T)$}, \\
\BT((1+t)\bW, (1+t)\sfP)\bn_0 = (1+t)\bH&&\quad&\text{on 
$\pd\HS\times(0, T)$},\\
\bW|_{t=0} = \bU_0 &&\quad&\text{on 
$\HS$}.
\end{aligned}\right.\end{equation} 
Let 
\begin{align}
F(t) &= (e_T[\bF], \nabla e_T[\Gdiv], \pd_t e_T[\bG], \Lambda^{1/2}
e_T[\bH],\nabla e_T[\bH]), \label{def:F} \\
F_t(t) &=((1+t)e_T[\bF], (1+t)\nabla e_T[\Gdiv], 
\pd_t((1+t)e_T[\bG], \label{def:Ft} \nonumber \\
&\qquad \Lambda^{1/2}((1+t)e_T[\bH]),  (1+t)\nabla e_T[\bH]),
\end{align}
and then by \eqref{local.2} we have
\begin{equation}\label{maxdec.1}\begin{aligned}
&\|(1+t)(\pd_t, \nabla^2)\bW\|_{L_p((0, T), L_{\oq}(\BR_+))}
+\|(1+t)\nabla \sfP\|_{L_p((0, T), L_{\oq}(\HS))}\\
&\quad \leq C(\|\bU_0\|_{B^{2(1-1/p)}_{\oq,p}(\HS)}
+ \|(1+t)(F, \nabla \Gdiv, \pd_t \bG, \nabla \bH, \Lambda^{1/2}\bH)\|_{L_p((0, T); L_{\oq}(\HS))} 
\\
&\hskip11cm + \|\bW\|_{L_p((0, T), L_{\oq}(\HS))})\\
&\quad \leq C(\|\bU_0\|_{B^{2(1-1/p)}_{\oq,p}(\HS)}
+ \|(1+t)(F, \nabla \Gdiv, \pd_t \bG, \nabla \bH)\|_{L_p((0, T); L_{\oq}(\HS))} + \|\bW\|_{L_p((0, T), L_{\oq}(\HS))})\\
&\quad +\|(1+t)\bH\|_{\dot H^{1/2}_p((0, T); L_{\oq}(\HS))}.
\end{aligned}
\end{equation}
Let
$$\bU(t) = \int^t_{-\infty} T(t-s)F(s)\,ds,$$
where $T(\cdot)$ is the analytic semigroup introduced in Appendix \ref{A1}. Then  $$\bW = T(t)(\bU_0-\bU(0)) + \bU.$$
Note that thanks to the limit regularity of kind of $W^{2,1}_p$ we are able to construct such extension  of the RHS such that $\bU(0) \equiv 0$. In other words, $\bF(s) \equiv 0$ for $s \leq 0$, see \cite{MuZaj}.

We shall show the following lemma.
\begin{lem}  Let $1 < \oq < \infty$ and $1 < p < \infty$. Assume moreover 
$F \equiv 0$ for $t<0$.
Then, for all $0<T \leq \infty$ we have
\begin{equation}\label{weight:5}
\|\bU\|_{L_\infty((0, T); L_{\oq}(\HS))} \leq C\|(1+t)F\|_{L_p((0, T); L_{\oq}(\HS))}.
\end{equation}
\par
Moreover, if $p,\oqq,\oq,\bar \theta$ satisfy \eqref{c:weight},
then we have
\begin{equation} \label{weight:3}
\|\bU\|_{L_p((1, T); L_{\oq}(\HS))} \leq C\|(1+t)F\|_{L_p((\frac{1}{2},T); L_{\oqq}(\HS))},
\end{equation}
and   
\begin{equation} \label{weight:3b}
\|\bU\|_{L_p((0,1); L_{\oq}(\HS))} \leq C\|F\|_{L_p((0,1); L_{\oq}(\HS))}.
\end{equation}
\end{lem}
\begin{proof}
When $0 < \bar \theta \leq 1$,  
by the Gagliardo-Nirenberg inequality and \eqref{bound:1}
we have 
\begin{equation}\label{inter:1}
\|\CS(\lambda)F\|_{L_{\oq}(\HS)} \leq C\| \CS(\lambda)F\|_{L_{\oqq}(\HS)}^{1-\theta} 
\|\nabla \CS(\lambda)F\|_{L_{\oqq}(\HS)}^\theta
\leq C|\lambda|^{-(1-\frac{\theta}{2})}\|F\|_{L_{\oqq}(\HS)},
\end{equation}
where $S(\cdot)$ is the solution operator to the generalized resolvent problem introduced in Appendix \ref{A1}.
By \eqref{inter:1}, we have
\begin{equation}\label{weight:2}
\|T(t)F\|_{L_{\oq}(\HS)} \leq Ct^{-\frac\theta2}\|F\|_{L_{\oqq}(\HS)}.
\end{equation}
Since $\|\CS(\lambda)F\|_{L_{\oq}(\HS)} \leq C\|F\|_{L_{\oq}(\HS)}$, we also have 
\begin{equation}\label{spectral:1}
\|T(t)F\|_{L_{\oq}(\HS)} \leq C\|F\|_{L_{\oq}(\HS)}. 
\end{equation}
Thus, as $F=0$ for $t<0$, we have
$$\|\bU(t)\|_{L_{\oq}(\HS))}
\leq C\int^t_{-\infty} \|F(s)\|_{L_{\oq}(\HS)} \,ds
\leq C\|(1+s)F\|_{L_p((0, t); L_{\oq}(\HS))}. $$
This shows \eqref{weight:5}.
To prove \eqref{weight:3}, 
the estimate is divided as follows: 
$$\|\bU\|_{L_p((1, T); L_{\oq}(\HS))}^p 
\leq C(I + II ),$$
where
\begin{align*}
I &= \int_{1}^T \Bigl(\int^{t/2}_{0}\|T(t-s)F(s)\|_{L_{\oq}(\HS)}\,ds\Bigr)^{p}\,dt, \quad
II &= \int_{1}^T\Bigl(\int^{t}_{t/2}\|T(t-s)F(s)\|_{L_{\oq}(\HS)}\,ds\Bigr)^{p}\,dt.
\end{align*}
For $I$, using the fact that $(t-s) \geq t/2$ for $s < t/2$, we have
\begin{align*}
I & \leq C\int_{1}^T\Bigl(\int^{t/2}_{0}(t-s)^{-\frac{\bar\theta}{2}}\|F(s)\|_{L_{\oqq}(\HS)}\,ds\Bigr)^p\,
\d t 
\\
& \leq C\int_{1}^T t^{-\frac{p\bar\theta}{2}}
\Bigl(\int_{0}^{t/2}(1+s)^{-1}(1+s)\|F(s)\|_{L_{\oqq}(\HS)}\,ds
\Bigr)^p\,dt\\
&\leq C_p\|(1+s)F(s)\|_{L_p((0,T); L_{\oqq}(\HS))}^p.
\end{align*}
Here, we have used $1/p < \bar\theta/2$. \par
For $II$, by H\"older's inequality
\begin{align*}
II & \leq C\int_{1}^T\Bigl(\int^t_{t/2}(t-s)^{-\frac{\bar\theta}{2}}
\|F(s)\|_{L_{\oqq}(\HS)}\,ds\Bigr)^p\,
\d t 
\\
& \leq C\int_{1}^T\Bigl(\Bigl(\int^t_{t/2}(t-s)^{-\frac{\bar\theta}{2}}\,ds\Bigr)^{p/p'}
\int^{t}_{t/2}(t-s)^{-\frac{\bar\theta}{2}}\|F(s)\|_{L_{\oqq}}^p\,ds\Bigr)\d t \\
&\leq  C\int_{1}^T \Bigl(\int^{t}_{t/2} t^{\frac{p}{p'}\left(1-\frac{\bar\theta}{2}\right)} 
(t-s)^{-\frac{\bar\theta}{2}}\|F(s)\|_{L_{\oqq}}^p\,\d s\Bigr)\,\d t \\
& \leq  C\int_{1/2}^T \Bigl(\int^{2s}_{s} t^{\frac{p}{p'}\left(1-\frac{\bar\theta}{2}\right)} 
(t-s)^{-\frac{\bar\theta}{2}}\, \d t\Bigr)\, \|F(s)\|_{L_{\oqq}}^p\,\d s\\
& \leq C\int_{1/2}^T s^{1-\frac{\bar\theta}{2}+(p-1)\left(1-\frac{\bar\theta}{2}\right)}\|F(s)\|_{L_{\oqq}}^p\,\d s\\
& \leq C\int_{\frac{1}{2}}^T (1+s)^p\|F(s)\|_{L_{\oqq}}^p\,\d s.
\end{align*} 
Summing up, we have obtained \eqref{weight:3}. 
\par
Next, since $\|T(t-s)F(s)\|_{L_{\oq}(\HS)} \leq C\|F(s)\|_{L_{\oq}(\HS)}$, we have
\begin{align*}
&\|\bU\|^p_{L_p((0,1); L_{\oq}(\HS))} \leq \int^1_{0}\Bigl(\int^t_{0}
\|T(t-s)F(s)\|_{L_{\oq}(\HS)}\,ds\Bigr)^p\,dt \\
&\quad \leq C\int^1_0\Bigl(\int^t_{0}\|F(s)\|_{L_{\oq}(\HS)}\,ds\Bigr)^p\,dt
\leq C\|F\|^p_{L_p((0,1); L_{\oq}(\HS)},
\end{align*}
which proves \eqref{weight:3b}. This completes the proof.
\end{proof}

Now, we are in a position to complete the proof of Theorem \ref{thm:weight}. For this purpose we consider the initial value problem \eqref{eq:4} with replacing $\bg$ with $\bU_0-\bU(0)$,
 and then 
$\bV(t) = T(t)(\bU_0-\bU(0))$ and $Q=K(T(t)(\bU_0-\bU(0))$, 
where the solution operator $K$ is defined in \eqref{def:K}.
From \eqref{bound:1} and \eqref{weight:2}  it follows that 
\begin{align}
\|\bV(t)\|_{L_{\oq}(\HS)} &\leq C\|\bU_0-\bU(0)\|_{L_{\oq}(\HS)}, 
\label{semi:2}\\
\|\bV(t)\|_{L_{\oq}(\HS)} &\leq Ct^{-\frac{\bar \theta}{2}}\|\bU_0-\bU(0)\|_{L_{\oqq}(\HS)},
\label{semi:3}
\end{align}
where $\bar \theta$ is defined in \eqref{c:weight}.
Combining \eqref{semi:2} and  \eqref{weight:5} yields
\begin{equation}\label{weight:6}
\|\bW\|_{L_\infty((0, T); L_{\oq}(\HS))} \leq C(\|\bU_0\|_{L_{\oq}(\HS)}+\|\bU(0)\|_{L_{\oq}(\HS)}
+ \|(1+|t|)F\|_{L_p((0,T); L_{\oq}(\HS))}).
\end{equation}
for all $0<T\leq +\infty$.  Since $p \bar\theta/2 > 1$, from \eqref{semi:3} it follows that 
$$\|\bV\|_{L_p((1, T); L_{\oq}(\HS))} \leq C(\|\bU_0\|_{L_{\oqq}(\HS)}
+ \|\bU(0)\|_{L_{\oqq}(\HS)}),
$$
while from \eqref{semi:2} it follows that 
$$\|\bV\|_{L_p((0, 1); L_{\oq}(\HS))} \leq C(\|\bU_0\|_{L_q(\HS)}+\|\bU(0)\|_{L_{\oq}(\HS)}.$$
Combining these two estimates with \eqref{weight:3} and \eqref{weight:3b} yields
\begin{equation}\label{weight:7}\begin{aligned}
\|\bW\|_{L_p((0, T), L_{\oq}(\HS))}
&\leq C(\|\bU_0\|_{L_{\oq}(\HS)} + \|\bU_0\|_{L_{\oqq}(\HS)}
+ \|\bU(0)\|_{L_{\oq}(\HS)} + \|\bU(0)\|_{L_{\oqq}(\HS)}\\
&\quad + \|(1+t)F\|_{L_p((\frac{1}{2},T); L_{\oqq}(\HS))}+\|F\|_{L_p((0,1); L_{\oq}(\HS))}
\end{aligned}\end{equation}
for $0<T\leq +\infty$. To complete the proof we observe that for $r \in \{\oq, \oqq\}$, we observe that
\begin{align*}
\|\bU(0)\|_{L_r(\HS)}&\leq \int^0_{-\infty}\|T(-s)F(s)\|_{L_r(\HS)}\,ds 
\leq C\int^0_{-\infty}\|F(s)\|_{L_r(\HS)}\,ds  \\
&\leq C\|(1+|s|)F\|_{L_p((-\infty, 0); L_r(\HS))}.
\end{align*}
However, this part can be neglected by the suitable choice of the extension of the RHS. Thus, from \eqref{weight:6} and \eqref{weight:7}, we have
\begin{align}\label{weight:8}
\|\bW\|_{L_\infty(0, T); L_{\oq}(\HS))} &\leq C(\|\bU_0\|_{L_{\oq}(\HS)}
+ \|(1+|t|)F\|_{L_p(\BR; L_{\oq}(\HS))}), \\
\|\bW\|_{L_p((0, T); L_{\oq}(\HS))}
&\leq C(\|\bU_0\|_{L_{\oq}(\HS)} + \|\bU_0\|_{L_{\oqq}(\HS)}\nonumber \\
&\quad + \|(1+t)F\|_{L_p((\frac{1}{2},T); L_{\oqq}(\HS))}+\|F\|_{L_p((0,1); L_{\oq}(\HS))}
\label{weight:9}
\end{align}
Using the definitions \eqref{def:F} and \eqref{def:Ft}, we have the following estimates:
\begin{align*}
&\|F_t\|_{L_p(\BR; L_{\oq}(\HS))} 
\leq C(\|(1+t)(\bF, \nabla \Gdiv, \pd_t \bG, \nabla \bH)\|_{L_p((0, T); L_{\oq}(\HS))}\\
&+ \|\bG\|_{L_p(0, T); L_{\oq}(\HS))} + \|\Lambda^{1/2}(1+t)\bH\|_{L_p((0,T); L_{\oq}(\HS))},
\\[5pt]
&\|(1+t)F\|_{L_p(\BR; L_{\oqq}(\HS))} 
\leq C(\|(1+t)(\bF, \nabla \Gdiv, \pd_t \bG, \nabla \bH)\|_{L_p((0, T), L_{\oqq}(\HS))}\\
&+\|\Lambda^{1/2}(1+t)\bH\|_{L_p((0,T); L_{\oq}(\HS))}.
\end{align*}
Then, combining \eqref{maxdec.1}, \eqref{weight:8}, 
and \eqref{weight:9}, we obtain \eqref{weight.global.1} and \eqref{weight.global.2}. The proof of Theorem \ref{thm:weight} is complete.
\qed

\section{Linear theory in the perturbed half-space} \label{sec:lin:pert}

In this part we want to generalize Theorems \ref{thm:max} and \ref{thm:weight} onto the case the domain is a perturbation of the half space defined in \eqref{omega}.
Therefore, we consider the Stokes problem:
\begin{equation}\label{eq:1:pert}\left\{\begin{aligned}
\pd_t\bW - \dv\BT(\bW, \sfP) = \bF&&\quad&\text{in $\Omega_0 \times\BR_+$}, \\
\dv\bW = G = \dv\bG&&\quad&\text{in $\Omega_0 \times\BR_+$}, \\
\BT(\bW, \sfP) \bn_0=\bH &&\quad&\text{on $\pd\Omega_0 \times\BR_+$}, \\
\bW|_{t=0} = \bU_0&&\quad&\text{in $\Omega_0$},
\end{aligned}\right.
\end{equation}
where $\bn_0$ is the outer normal vector to the boundary $\partial \Omega_0$. Since our analysis is done in unbounded domains, we are required to 
fix the factors of integrability. 
Let $d \geq 2$ and assume \eqref{exp:1}, which implies \eqref{exp:2}.

We start with recalling a general imbedding estimate which can be of independent interest. 
\begin{lem} 
Assume $q_1,q_2,\alpha$ satisfy $q_1<d<q_2$ and 
\begin{equation} \label{c:alpha}
\frac{\alpha}{q_1} + \frac{1-\alpha}{q_2} = \frac{1}{d}, \quad \alpha \in (0, 1).
\end{equation}
Then for every $\bv \in \dW^2_{q_1}\cap \dW^2_{q_2}$ we have 
\begin{equation}\label{sup:1}
\|\nabla \bv\|_{L_\infty} \leq C\|\nabla^2\bv\|_{L_{q_1}}^{1-\alpha}\|\nabla^2 \bv\|_{L_{q_2}}^\alpha
\end{equation}
and 
\begin{equation} \label{sup:L1}
\|\nabla \bv\|_{L_1((0,T);L_\infty)}\leq C\sum_{i=1}^2\| (1+t)\nabla^2 \bv \|_{L_p((0,T);L_{q_i})}.
\end{equation}
\end{lem}
\begin{proof}
In order to show these elementary facts we are using the Besov spaces approach and real interpolation theorems.
Since $d/q_2 < 1$, by \eqref{0:3a} we have 
\begin{equation} \label{imbed:3}
L_\infty \supset \dot B^{d/q_2}_{q_2,1} = (\dot B^1_{q_2, \infty}, 
\dot B^{1-d(1/q_1- 1/q_2)}_{q_2, \infty})_{\alpha, 1}
\end{equation}
where $\alpha \in (0, 1)$ is chosen as 
$$\frac{d}{q_2} = (1-\alpha) + \alpha\biggl(1-d\bigl(\frac{1}{q_1}-\frac{1}{q_2}\bigr)\biggr),
$$
which is equivalent to \eqref{c:alpha}.
Since $q_1 < q_2$, by \eqref{imbed:besov} we have 
$\dot B^{1-d(1/q_1- 1/q_2)}_{q_2, \infty}
\supset \dot B^1_{q_1, \infty},$  
which, combined with \eqref{imbed:3}, gives 
$L_\infty \supset (\dB^1_{q_2, \infty}, \dB^1_{q_1, \infty})_{\alpha, 1}. $
From the observation above, for any 
$\bv \in \dot W^2_{q_1} \cap \dot W^2_{q_2}$
we obtain \eqref{sup:1}
provided that $q_1,q_2$ and $\alpha \in (0, 1)$ satisfy \eqref{c:alpha}. 
In the last step we have used the fact that $\dB^1_{q,\infty} \supset \dot W^1_q$. Now, \eqref{sup:L1} is a direct consequence of \eqref{sup:1} and H\"older inequality:
$$
\int_0^T \|\nabla \bv\|_{\infty} dt\leq \int_0^T [(1+t)^{-1}(1+t) \sum_{i=1}^2\|\nabla^2 \bv\|_{q_i}]\,dt \leq C \sum_{i=1}^2 \|(1+t)\nabla^2 \bv\|_{L_p((0,T);L_{q_i})},
$$
which completes the proof.
\end{proof}
The first main result of this section reads 
\begin{thm} \label{thm:max:pert}
Let  $1 < p < \infty$ and let $q_1,q_2$ satisfy the assumptions of Theorem \ref{thm:gwp}. 
Let $\Omega_0$ be defined by (\ref{omega}) with sufficiently small $\epsilon>0$.
Let $\bU_0 \in \bigcap_{i=1}^2 \dot B^{(1-1/p)}_{q_i,p}(\Omega)$ and 
\begin{align*}
\bF &\in \bigcap_{i=1}^2 L_p(\BR_+; L_{q_i}(\Omega_0)^d), \quad 
\Gdiv \in \bigcap_{i=1}^2 L_p(\BR_+; \dW^1_{q_i}(\Omega_0)), \quad \bG \in \bigcap_{i=1}^2 \dW^1_p(\BR_+; L_{q_i}(\Omega_0)^d), \\
 \bH &\in \bigcap_{i=1}^2 \left( L_p(\BR_+; \dW^1_{q_i}(\Omega_0)^d) \cap \dW^{1/2}_p(\BR_+; L_{q_i}(\Omega_0)^d)\right).
\end{align*}
Then, problem \eqref{eq:1:pert} admits unique solutions $\bW$ and $\sfP$ with
$$\bW \in \bigcap_{i=1}^2 \left( L_p(\BR_+; W^2_{q_i}(\Omega_0)^N) \cap  W^1_p(\BR_+; L_{q_i}(\Omega_0)^N))\right),
\quad \nabla \sfP \in \bigcap_{i=1}^2 L_p(\BR_+; L_{q_i}(\Omega_0)^N)$$
satisfying the estimate:
\begin{equation} \label{lin:global:pert}
\begin{aligned} 
&\sum_{i=1}^2\|(\pd_t\bW, \nabla^2\bW, \nabla \sfP)\|_{L_p(\BR_+; L_{q_i}(\Omega_0))} \\
&\quad \leq C\sum_{i=1}^2\big(\|\bU_0\|_{\dot B^{2(1-1/p)}_{q_i,p}(\Omega_0)} 
+ \|(\bF, \nabla \Gdiv, \pd_t\bG, \nabla \bH, \Lambda^{1/2}\bH)\|_{L_p(\BR; L_{q_i}(\Omega_0))}\big).
\end{aligned}
\end{equation}
Moreover, if \eqref{ic:zero} holds, then
for any $T>0$ we have
\begin{equation}\label{local.2:pert}
\begin{aligned}
&\sum_{i=1}^2\|(\pd_t, \nabla^2) \bW\|_{L_p((0, T); L_{q_i}(\Omega_0))}
\leq  C\sum_{i=1}^2\big(\|\bU_0\|_{B^{2(1-1/p)}_{q_i,p}(\Omega_0)} \\
&\qquad +\|(\bF, \nabla \Gdiv, \pd_t \bG, \nabla \bH)\|_{L_p((0, T); L_{q_i}(\Omega_0))}
+ \|\bH\|_{\dot H^{1/2}_p((0, T); L_{q_i}(\Omega_0))}\big).
\end{aligned}
\end{equation}
Here, the constant $C$ is independent of $T$.
\end{thm} 
\begin{proof}
Let us give a sketch of the proof, focusing on the most restrictive terms. Based of the result from \cite{DM2015} -- Section 2.4, given $\Omega_0$ as assumed, there is a diffeomorphism 
\begin{equation}
    \Phi: \Omega_0 \to \HS
\end{equation}
such that 
\begin{equation} 
    \Phi \in W^2_{q_1}(\Omega_0) \cap W^2_{q_2}(\Omega_0) \quad {\rm and} \; \Phi \;
    \textrm{is a measure preserving map, ie.} \; {\rm det}\{D\Phi\} =1.
\end{equation}
Additionally, due to the smallness of $\epsilon$, there exists a point $x_0\in \partial \Omega_0$ such that 
\begin{equation} \label{prop:phi}
    \Phi(x_0)=0, \quad D\Phi_{x_0}(T_{x_0}(\partial \Omega_0)) = \R^{d-1}, \mbox{ \ 
     and \ } \|\nabla \Phi(x) - Id\|_{(W^{1}_{q_2}\cap W^1_{q_1})(\Omega_0)} < C \epsilon.
\end{equation}
Let us denote 
$$
\tilde \bW(t,y)=\bW(t,\Phi^{-1}(t,y)), \quad \tilde P(t,y)=P(t,\Phi^{-1}(t,y)), \quad \BT_{\Omega}(\tW,\tilde P)=\BT_x (\bW,P)
$$
and 
$$
\dv_\Omega \widetilde {\bf Z} := \dv_x {\bf Z},
$$
where ${\bf Z}$ is either a vector or matrix-valued function and 
$\widetilde {\bf Z}(t,y)={\bf Z}(t,\Phi^{-1}(t,y))$.

Taking the transformation into the halfspace system \eqref{eq:1:pert} takes the form
\begin{equation}\label{Stokes:fixed1}\left\{\begin{aligned}
\pd_t\widetilde\bW - \dv_{\Omega_0}\BT_{\Omega_0}(\widetilde \bW, \widetilde\sfP) = \widetilde \bF&&\quad&\text{in $\HS\times\BR_+$}, \\
\dv_{\Omega_0} \widetilde\bW = \widetilde \Gdiv = \dv_{\Omega_0} \widetilde \bG&&\quad&\text{in $\HS \times\BR_+$}, \\
\BT_{\Omega_0} (\widetilde \bW, \widetilde\sfP) \bn_{0\Omega_0} =\widetilde \bH &&\quad&\text{on $\pd\HS  \times\BR_+$}, \\
\widetilde \bW|_{t=0} = \widetilde \bU_0&&\quad&\text{in $\HS$}.
\end{aligned}\right.
\end{equation}
In order to apply Theorems \ref{thm:max} and \ref{thm:weight}, the system \eqref{Stokes:fixed1} is restated in the following way
\begin{equation}\label{Stokes:fixed2}
\left\{\begin{aligned}
\pd_t\widetilde\bW - \dv \BT(\widetilde \bW, \widetilde\sfP) =
\dv_{\Omega_0}\BT_{\Omega_0}(\widetilde \bW, \widetilde\sfP)-
\dv \BT (\widetilde \bW, \widetilde\sfP)+
\widetilde \bF&&\quad&\text{in $\HS\times\BR_+$}, \\
\dv \widetilde\bW = \dv(Id -A_{\Omega_0}) \widetilde\bW + \dv A_{\Omega_0} \widetilde \bG&&\quad&\text{in $\HS \times\BR_+$}, \\
\BT (\widetilde \bW, \widetilde\sfP) \bn_{0} =
\BT (\widetilde \bW, \widetilde\sfP) \bn_{0}-
\BT_{\Omega_0} (\widetilde \bW, \widetilde\sfP) \bn_{0{\Omega_0}}
\widetilde \bH &&\quad&\text{on $\pd\HS  \times\BR_+$}, \\
\widetilde \bW|_{t=0} = \widetilde \bU_0&&\quad&\text{in $\HS$},
\end{aligned}\right.
\end{equation}
where thanks to measure preserving property of map $\Phi$, for any vector field $F$ we have 
\begin{equation}
    \dv_{\Omega_0} F = \div A_{\Omega_0} F,
\end{equation}
where $A_{\Omega_0}:=D\Phi$, see for instance \cite[Section 1]{DMP}.
In order to close the estimates we need to estimate the rhs of \eqref{Stokes:fixed2}. Since the map $\Phi$ is time independent, the issue is relatively easy, so we just take into account the main term from the  momentum equation. Note that
\begin{equation} \label{div:pert}
\begin{aligned}
    \dv_\Omega\BT_\Omega(\widetilde \bW, \widetilde\sfP)-
\dv \BT (\widetilde \bW, \widetilde\sfP)= &
\dv_\Omega\left( \BD_\Omega(\widetilde \bW)-
 \BD (\widetilde \bW) \right)
+\left(\dv_\Omega -
\dv \right) \BD (\widetilde \bW)\\
& + (Id-A_\Omega)\nabla \widetilde\sfP.
\end{aligned}
\end{equation}
By \eqref{prop:phi} we have obviously 
\begin{equation}
\|\nabla \Phi - Id\|_{(W^{1}_{q_2}\cap W^1_{q_1})(\Omega_0)} < C \epsilon.
\end{equation}
Then it is clear that
\begin{multline} \label{est:div:1}
    \| \dv_{\Omega_0}\left( \BD_{\Omega_0}(\widetilde \bW)-
 \BD (\widetilde \bW) \right) \|_{L_p(0,T;L_q(\HS))} \leq \\
 C\| Id - \nabla \Phi \|_{L_\infty(\HS)} \|\nabla^2 \widetilde \bW\|_{L_p(0,T;L_q(\HS))}+
 C\| |\nabla^2 \Phi | \, |\nabla \widetilde \bW| \|_{L_p(0,T;L_q(\HS))},
\end{multline}
where, with slight abuse of notation, in the second term we denote $\nabla^2 \Phi:=\nabla^2_x\Phi(t,\Phi^{-1}(t,y))$, therefore the norm is taken over $\HS$. 
The last term requires some care and we need to do the estimates for $q_2$ and $q_1$
separately. For $q_2$ we have by \eqref{prop:phi} and \eqref{sup:1} below
\begin{equation}  \label{Phi:q2}
\begin{aligned}
    &\| |\nabla^2 \Phi | \, |\nabla \widetilde \bW| \|_{L_p(0,T;L_{q_2}(\HS))}
    \leq C\|\nabla^2 \Phi\|_{L_{q_2}(\HS)} \|\nabla \widetilde \bW \|_{L_p(0,T;L_{\infty}(\HS))}\\
    &\leq C\epsilon \left(\|\nabla^2 \widetilde \bW\|_{L_p{(0,T;L_{q_1}(\HS))}}+\|\nabla^2 \widetilde \bW\|_{L_p{(0,T;L_{q_2}(\HS))}}\right)
\end{aligned}
\end{equation}
and for $q_1$
\begin{equation}  \label{Phi:q1}
\begin{aligned}
    &\| |\nabla^2 \Phi | \, |\nabla \widetilde \bW |\|_{L_p(0,T;L_{q_1}(\HS))}
    \leq C\|\nabla^2 \Phi\|_{L_{d}(\HS)} \|\nabla \widetilde \bW \|_{L_p(0,T;L_{q_2}(\HS))},\\
    &\leq C (\|\nabla^2 \Phi\|_{L_{q_1}(\HS)} + \|\nabla^2 \Phi\|_{L_{q_2}(\HS)}) \|\nabla^2 \widetilde \bW \|_{L_p(0,T;L_{q_1}(\HS))}
    \leq C\epsilon \|\nabla^2 \widetilde \bW \|_{L_p(0,T;L_{q_1}(\HS))},
\end{aligned}
\end{equation}
remembering that $\frac{1}{q_1}=\frac{1}{q_2}+\frac{1}{d}$. Above the estimate for $\nabla^2 \Phi$ follows from the interpolation inequality as $q_1<d<q_2$.

Combining the estimate on $\nabla^2 \Phi$ with \eqref{est:div:1}-\eqref{Phi:q2} 
we obtain 
\begin{equation} \label{est:div:2}
 \sum_{i=1}^2\| \dv_{\Omega_0}\left( \BD_{\Omega_0}(\widetilde \bW)-
 \BD (\widetilde \bW) \right) \|_{L_p(0,T;L_{q_i}(\HS))} \leq C\epsilon \sum_{i=1}^2\|\nabla^2 \widetilde \bW \|_{L_p(0,T;L_{q_i}(\HS))}.
\end{equation}
Estimating similarily the other terms resulting from the change of coordinates on the RHS of \eqref{Stokes:fixed2}, we can apply
\eqref{lin:global} and \eqref{local.2} to \eqref{Stokes:fixed2}. This way we obtain, respectively, \eqref{lin:global:pert} and \eqref{local.2:pert}, which completes the proof of Theorem \ref{thm:max:pert}. 
\end{proof}

The second result gives time-weighted maximal regularity for \eqref{eq:1}.
\\
\begin{thm}\label{thm:weight:pert} Let $p>1$ and assume $q_0,q_1,q_2$ satisfy the assumptions of Theorem \ref{thm:gwp}.
Let $1<T\leq +\infty$ and let $\bW$ and $\sfP$ be solutions 
to equations \eqref{eq:1}.  Then, there holds
\begin{equation} \label{weight.global.1P}
\sum_{i=1}^2\|\bW\|_{L_\infty((0, T), L_{q_i}(\Omega_0))}  \leq C\sum_{i=0}^2
(\|\bU_0\|_{L_{q_i}(\Omega_0)}
+ \CF^2_{q_i}(T)),
\end{equation}
and
\begin{align} \label{weight.global.2P}
&
\sum_{i=1}^2
\|(1+t)(\pd_t, \nabla^2)\bW\|_{L_p((0, T); L_{q_i}(\Omega_0))}
\leq \\
&
C\big( \|\bU_0\|_{B^{2(1-1/p)}_{q_2,p}(\Omega_0)}+
\|\bU_0\|_{B^{2(1-1/p)}_{q_1,p}(\Omega_0)}+\|\bU_0\|_{L_{q_0}(\Omega_0)} 
+ \CF^1_{q_2}(T,\Omega_0) + \CF^2_{q_1}(T,\Omega_0) 
 + \CF^2_{q_0}(T)\big), \nonumber \\[3pt]
&
\|\bW\|_{L_p((0, T); L_{q_1}(\Omega_0))} \label{weight.global.3P}
\leq C\big(\|\bU_0\|_{L_{q_0}(\Omega_0)} + \|\bU_0\|_{L_{q_1}(\Omega_0)} + \CF^2_{q_1}(T,\Omega_0) + \CF^2_{q_0}(T,\Omega_0) \big),\\[3pt]
&\|\bW\|_{L_p((0, T); L_{q_2}(\Omega_0))} \leq [\textrm{RHS of \eqref{lin:global:pert}}] + [\textrm{RHS of \eqref{weight.global.3P}}]  \label{weight.global.4P},
\end{align}
where $\CF^i_{q}(T,\Omega_0)$ is defined in \eqref{right.1}.
\end{thm}

\smallskip

{\bf Sketch of the proof of Theorem \ref{thm:weight:pert}}.
We consider system \eqref{weight:1} on $\Omega_0 \times (0,T)$ 
and rewrite it to the fixed domain $\HS \times (0,T)$ using the diffeomorphism $\Phi$. This leads to an analog of system \eqref{Stokes:fixed2} with $\left((1+t)\widetilde \bW,(1+t)\widetilde \sfP\right)$ instead of $(\widetilde \bW,\widetilde \sfP)$ and an additional term $- \widetilde \bW$ on the RHS of the momentum equation.
To this system we apply Theorem \ref{thm:weight}, 
for this purpose we need to estimate the terms coming from the transformation. Similarly as in the proof of Theorem \ref{thm:max:pert}, we restrict ourselves to showing the details for the most restrictive representative terms related to 
$$
\|(1+t) \dv_{\Omega_0}\left( \BD_{\Omega_0}(\widetilde \bW)-
 \BD (\widetilde \bW) \right) \|_{L_p(0,T;L_{q_i})}, \quad i=0,1,2. 
$$
Namely, we have 
\begin{equation}  \label{Phi:q110}
\begin{aligned}
    &\|(1+t) |\nabla^2 \Phi | \, |\nabla \widetilde \bW |\|_{L_p(0,T;L_{q_0}(\HS))}
    \leq C\|\nabla^2 \Phi\|_{L_{q_1}(\HS)} \|(1+t)\nabla \widetilde \bW \|_{L_p(0,T;L_{q_2}(\HS))}\\
    &\leq C \|\nabla^2 \Phi\|_{L_{q_1}(\HS)} \|(1+t) \nabla^2 \widetilde \bW \|_{L_p(0,T;L_{q_1}(\HS))}
    \leq C\epsilon \|(1+t)\nabla^2 \widetilde \bW \|_{L_p(0,T;L_{q_1}(\HS))},\\[5pt]
    &\|(1+t) |\nabla^2 \Phi | \, |\nabla \widetilde \bW |\|_{L_p(0,T;L_{q_1}(\HS))}
    \leq C\|\nabla^2 \Phi\|_{L_{d}(\HS)} \|(1+t)\nabla \widetilde \bW \|_{L_p(0,T;L_{q_2}(\HS))}\\
    &\leq C (\|\nabla^2 \Phi\|_{L_{q_1}(\HS)}+\|\nabla^2 \Phi\|_{L_{q_2}(\HS)}) \|(1+t) \nabla^2 \widetilde \bW \|_{L_p(0,T;L_{q_1}(\HS))}\\
    &\leq C\epsilon \|(1+t)\nabla^2 \widetilde \bW \|_{L_p(0,T;L_{q_1}(\HS))}.    
\end{aligned}
\end{equation}
and, by \eqref{sup:1}, 
\begin{equation}  \label{Phi:q111}
\begin{aligned}
    &\|(1+t) |\nabla^2 \Phi | \, |\nabla \widetilde \bW |\|_{L_p(0,T;L_{q_2}(\HS))}
    \leq C\|\nabla^2 \Phi\|_{L_{{q_2}}(\HS)} \|(1+t)\nabla \widetilde \bW \|_{L_p(0,T;L_{\infty}(\HS))}\\
    &\leq C \|\nabla^2 \Phi\|_{L_{q_2}(\HS)}
    \left( \|(1+t) \nabla^2 \widetilde \bW \|_{L_p(0,T;L_{q_1}(\HS))}
    + \|(1+t) \nabla^2 \widetilde \bW \|_{L_p(0,T;L_{q_2}(\HS))} \right)\\
    &\leq C\epsilon \left( \|(1+t) \nabla^2 \widetilde \bW \|_{L_p(0,T;L_{q_1}(\HS))}
    + \|(1+t) \nabla^2 \widetilde \bW \|_{L_p(0,T;L_{q_2}(\HS))} \right) .
\end{aligned}
\end{equation}
The other terms are estimated in a similar way, there is just a need to control the new parts related to $\nabla^2 \Phi^{-1}$, which fortunately are small and fine integrable. Therefore, applying \eqref{weight.global.2} we obtain \eqref{weight.global.2P}, while \eqref{weight.global.3} yields \eqref{weight.global.3P}. Finally, \eqref{weight.global.4P} is obtained 
using the imbedding
$$
\|\bW\|_{L_p(0,T;L_{q_2}(\Omega_0))} \leq C \left( \|\bW\|_{L_p(0,T;L_{q_1}(\Omega_0))} + \|\nabla^2 \bW\|_{L_p(0,T;L_{q_1}(\Omega_0))} \right).
$$

\section{Nonlinear $L_p-L_q$ estimates. Proof of Theorem \ref{thm:gwp}} \label{sec:nonlin:lplq}

The local wellposedness of equations \eqref{eq-fixed} has been shown by Shibata \cite{SS15}.  
Let $T >0$ and assume $\bW$ satisfies equations \eqref{eq-fixed}. 
Then, as it is well-known (cf. for instance \cite{ES}, \cite{OS22} or \cite[Section 5]{PSZ}), 
in order to prove Theorem \ref{thm:gwp}, it suffices to prove that 
$\bW$ satisfies the estimate:
\begin{equation}\label{apriori}
E_T(\bW) \leq C\left(\sum_{i=1}^2 \|\bU_0\|_{B^{2(1-1/p)}_{q_i,p}(\HS)} 
+ \|\bU_0\|_{L_{q_0}(\HS)} + E_T(\bW)^2\right)
\end{equation}
for some constant being independent of $T$, where $E_T(\bW)$ is defined in \eqref{def:ETW}.

When $0<T<2$, in $E_T(\bW)$ 
the weight $(1+t)$ has no meaning.  So, essentially we have to consider the 
case when $T > 2$. 
Recall the definition \eqref{exp:1} of exponents $q_0$, $q_1$ and $q_2$ given in Theorem \ref{thm:gwp}. 
In particular, we have \eqref{exp:2}.

{\bf Half space case.} We want to apply Theorem \ref{thm:weight}, first with $\oq=q_1$, $\oqq=q_0$ and $\bar \theta = \theta_1:=d(1/q_0-1/q_1)$. 
Then by \eqref{exp:2} we have $\theta_1 = \theta$, so 
\eqref{c:weight} implies
\begin{equation}\label{require.1A}
\frac1p < \frac{\theta}{2}.
\end{equation}
Next, we apply Theorem \ref{thm:weight} with $\oq=q_2,\;\oqq=q_1$, which leads to $\bar \theta=d(1/q_1-1/q_2)=1$. 
Therefore the last condition in \eqref{c:weight} reduces to $p > 2$, which is weaker then \eqref{require.1A}.

Moreover, to obtain $q_0 > 1$, we assume that $1+2\theta < d$. Since $d \geq 2$, 
it suffices to assume that $1 + 2\theta <2$, that is $0 <\theta < 1/2$.
Let us denote $\CF^j_{q_i}(\bW)(T,\Omega)$ by taking 
$$\bF=\bF(\bW), \enskip \Gdiv=\Gdiv(\bW), \enskip \bG = \bG(\bW),  \enskip  
\bH = \BH(\bW)
$$
in \eqref{right.1}. 
Applying Theorem \ref{thm:weight} $\oq=q_1,\; \oqq=q_0$ to \eqref{eq-fixed} we get   
\begin{equation} \label{nonlin.10A}
\begin{aligned}
&\|(1+t)(\pd_t, \nabla^2)\bW\|_{L_p((0, T); L_{q_1}(\HS))}
+ \|\bW\|_{L_p((0, T); L_{q_1}(\HS))}\\ 
&\leq C(\|\bU_0\|_{B^{2(1-1/p)}_{q_1,p}(\HS)} + \|\bU_0\|_{L_{q_0}(\HS)} 
+ \CF^1_{q_1}(\bW)(T) + \CF^2_{q_0}(\bW)(T)),\\[5pt]
&\|\bW\|_{L_\infty((0, T); L_{q_1}(\HS))} + \|\bW\|_{L_p((0, T); L_{q_1}(\HS))}\\ 
&\leq C\big(\|\bU_0\|_{L_{q_1}(\HS)} + \|\bU_0\|_{L_{q_0}(\HS)} + \CF^2_{q_1}(\bW)(T,\HS) + \CF^2_{q_0}(\bW)(T,\HS) \big).
\end{aligned}
\end{equation}
while setting $\oq =q_2,\; \oqq=q_1$ in \ref{thm:weight}
\begin{equation} \label{nonlin.10}
\begin{aligned}
&\|(1+t)(\pd_t, \nabla^2)\bW\|_{L_p((0, T); L_{q_2}(\HS))}
+ \|\bW\|_{L_p((0, T); L_{q_2}(\HS))}\\ 
&\leq C(\|\bU_0\|_{B^{2(1-1/p)}_{q_2,p}(\HS)} + \|\bU_0\|_{L_{q_1}(\HS)} 
+ \CF^1_{q_2}(\bW)(T) + \CF^2_{q_1}(\bW)(T)),\\[5pt]
&\|\bW\|_{L_\infty((0, T); L_{q_2}(\HS))} + \|\bW\|_{L_p((0, T); L_{q_2}(\HS))}\\ 
&\leq C\big(\|\bU_0\|_{L_{q_2}(\HS)} + \|\bU_0\|_{L_{q_1}(\HS)} + \CF^2_{q_2}(\bW)(T,\HS) + \CF^2_{q_1}(\bW)(T,\HS) \big).
\end{aligned}
\end{equation}
Combining the above inequalities and recalling that \eqref{right.1} implies in particular $\CF_q^1 \geq \CF_q^2$ we conclude 
\begin{align}
&\sum_{i=1}^2 \Bigl(\|(1+t)(\pd_t\bW, \nabla^2\bW, \nabla\sfP)\|_{L_p((0, T); L_{q_i}(\HS))}
+ \|\bW\|_{L_p((0, T); L_{q_i}(\HS))} +\|\bW\|_{L_\infty(0, T), L_{q_i}(\HS))}\Bigr)
\nonumber \\
&\quad \quad \leq C\Bigl(\sum_{i=1}^2(\|\bU_0\|_{B^{2(1-1/p)}_{q_i}(\HS)} 
+\CF^1_{q_i}(\bW)(T)) + \|\bU_0\|_{L_{q_0}(\HS)} + \CF^2_{q_0}(\bW)(T)\Bigr).\label{nonlinear:1}
\end{align}
In the nonhomogeneous case \eqref{inhomo-eq-fixed} the only difference is we the additional term 
\begin{equation} \label{nonlin:inhomo}
\sum_{i=0}^2 \|(1-\rho_0)(1+t)\de_t \bW\|_{L_p((0,T);L_{q_i}(\HS))}
\end{equation}
on the RHS of \eqref{nonlinear:1}.

The analysis for 
{\bf perturbed half space case} is immediate. This time we apply Theorem \ref{thm:weight:pert} 
to \eqref{eq-fixed}. This leads precisely to \eqref{nonlinear:1} with $\HS$ replaced by $\Omega_0$ - we skip the details for brevity.  
Therefore, in order to prove \eqref{apriori} we need to find appropriate estimates of the nonlinear terms on the RHS of \eqref{nonlinear:1}. For simplicity, we write $L_q:=L_q(\Omega_0)$,
where $\Omega_0=\R^d_+$, or it is given by \eqref{omega}, 
and similarly for other functional spaces defined on the spatial domain.

In what follows, we shall estimate nonlinear terms for $t > 1$. 
For this purpose, observe that we have 
\begin{equation} \label{struct:lag}
\BI - \BA_{v} = \phi(\int_0^t \nabla v d\tau), \quad \phi \in C^\infty(\R^d;\R^d), \quad \phi(0)=0.
\end{equation}
Therefore, taking into account the stucture of nonlinearities \eqref{def-nonlinear-terms}, we need to consider only the following typical nonlinear terms:
\begin{equation} \label{nonlin}
\begin{aligned}
N_1 &= (\pd_t, \nabla^2)v\int^t_0\nabla v\,\d s, \quad
N_2 = \nabla v \int^t_0\nabla^2v\,\d s, \\
N_3 & = \nabla v \int^t_0 \nabla v\,\d s, 
\quad N_4 = v\int^t_0\nabla v\,\d s.
\end{aligned}
\end{equation}
The first two terms are in the equation, the third term is the divergence condition:
$\dv \bu=G$ and boundary terms.  The last one is the divergence condition:
$\dv \bu = \dv \bG$.
Precisely, by \eqref{struct:lag} we have 
\begin{align}
&\bF(v),\nabla G_{\rm div}(v), \nabla \bH(v) \sim N_1+N_2, \label{struct:F}\\
&\bG(v)\sim N_4, \label{struct:G}\\
&\bH(v)\sim N_3, \label{struct:H}
\end{align}

\subsection{Estimates of $\bF(v),\bG(v), \Gdiv(v)$ and $\nabla \bH(v)$}
For the purpose of estimating $\bF(v),\bG(v), \Gdiv(v)$ and $\nabla \bH(v)$, taking into account \eqref{struct:lag}, it is enough to find appropriate bounds on 
$N_1,N_2$ and $N_4$.
Recalling that $1/q_0 = 1/q_1 + 1/q_2$ and $d(1/q_1-1/q_2) = 1$,  we have $\|\nabla v\|_{L_{q_2}}
\leq C\|\nabla^2 v\|_{L_{q_1}}$, and so
\begin{align*}
\|N_1\|_{L_{q_0}} &\leq \|(\pd_t, \nabla^2)v\|_{L_{q_1}}\int^t_0 \|\nabla v\|_{L_{q_2}}\,ds
\leq C\|(\pd_t, \nabla^2)v\|_{L_{q_1}}
\int^t_0\|\nabla^2 v\|_{L_{q_1}}\,\d s.
\end{align*}
For $q=q_1$ and $q_2$, we have 
\begin{align*}
\|N_1\|_{L_{q_i}} &\leq \|(\pd_tv, \nabla^2v)\|_{L_{q_i}}\int^t_0\|\nabla v\|_{L_\infty}\,
\d s \\
& \leq C \|(\pd_tv, \nabla^2v)\|_{L_{q_i}}\int^t_0(\|\nabla^2 v\|_{L_{q_1}}
+ \|\nabla^2v\|_{L_{q_2}})\,\d s.
\end{align*}
Observe that 
\begin{equation} \label{L1:1}
\int_0^t \|\nabla^2 v\|_{L_{q_i}}ds \leq \left( \int_0^t(1+t)^{-p'} \right)^{1/p'}\|(1+t)\nabla^2 v\|_{L_p(0,T;L_{q_i})}=C\|(1+t)\nabla^2 v\|_{L_p(0,T;L_{q_i})}.
\end{equation}
and similarly, as $q_2=q_1^*$, 
\begin{equation} \label{L1:2}
\int_0^t \|\nabla v\|_{L_{q_i}}ds \leq C\|(1+t)\nabla^2 v\|_{L_p(0,T;L_{q_1})}.
\end{equation}
Putting together the above estimates yields  
\begin{equation}\label{close:1}
\sum_{i=0}^2\|(1+t)N_1\|_{L_p((0, T); L_{q_i})}
\leq C(\sum_{\ell=1}^2\|(1+t)(\pd_t, \nabla^2)v\|_{L_p((0, T); L_{q_\ell})})^2.
\end{equation}
Next, recalling that $1/q_0 = 1/q_1 +1/q_2$,  we have 
\begin{align*}
\|N_2\|_{L_{q_0}} \leq C\|\nabla v\|_{L_{q_2}} \int^t_0\|\nabla ^2v\|_{L_{q_1}}\,\d s
\leq C\|\nabla^2 v\|_{L_{q_1}} \int^t_0\|\nabla ^2v\|_{L_{q_1}}\,\d s.
\end{align*}
From this  it follows that 
$$\|(1+t)N_2\|_{L_p((0, T); L_{q_0})}
\leq C\|(1+t)\nabla^2v\|_{L_p((0, T); L_{q_1})}\|(1+t)\nabla^2v\|_{L_p((0, T); L_{q_1})}.
$$
For $i =1,2$, we have 
\begin{align*}
\|N_2\|_{L_{q_i}}\leq C\|\nabla v\|_{L_\infty} \int^t_0\|\nabla ^2v\|_{L_{q_i}}\,\d s
\leq C(\|\nabla^2 v\|_{L_{q_1}} + \|\nabla^2 v\|_{L_{q_2}})
\|(1+t)\nabla^2 v\|_{L_p((0, T); L_{q_i})}.
\end{align*}
Thus, for $i=1,2$, we have 
$$\|(1+t)N_2\|_{L_p((0, T); L_{q_i})}
\leq C(\sum_{\ell=1,2}\|(1+t)\nabla^2 v\|_{L_p((0, T); L_{q_\ell})})
\|(1+t)\nabla^2 v\|_{L_p((0, T); L_{q_i})}.$$
Putting estimates above together yields
\begin{equation}\label{close:2}
\sum_{i=0}^2\|(1+t)N_2\|_{L_p((0, T); L_{q_i})}
\leq C(\sum_{\ell=1}^2\|(1+t)\nabla^2v\|_{L_p((0, T); L_{q_\ell})})^2.
\end{equation}
Combining \eqref{close:1} with \eqref{close:2} and recalling \eqref{def:ETW} we obtain 
\begin{equation} \label{close:1A}
\begin{aligned}
&\sum_{i=0}^2 \|(1+t)\bF(v),(1+t)\nabla G_{div}(v),\nabla \bH(v)\|_{L_p((0, T); L_{q_i})}\\
&\leq C(\sum_{\ell=1}^2\|(1+t)(\pd_t, \nabla^2)v\|_{L_p((0, T); L_{q_\ell})})^2
\leq C E_T(v)^2.
\end{aligned}
\end{equation}
Next, we estimate $\bG(v)$. For this purpose it is enough to consider $N_4$. Since
$$\pd_tN_4= \pd_tv \left( \int^t_0\nabla v\,\d s\right)
+ v\nabla v,
$$
we will estimate $\|\pd_tv \int^t_0\nabla v\,\d s\|_{L_{q_i}}$ and
$\|v\nabla v\|_{L_{q_i}}$
for $i=0,1,2$,   and from $F_t$, we consider $\|N_4\|_{L_p(\BR_+, L_{q_i}(\HS))}$
 for $i=1,2$. And so, we have to estimate 
$\|v \int^t_0\nabla v\,ds\|_{L_{q_i}}$ ($i=1,2$). 
By H\"older inequality and Sobolev imbedding, we have
\begin{align*}
\|\pd_tv \int^t_0\nabla v\,\d s\|_{L_{q_0}}
&\leq \|\pd_tv\|_{L_{q_1}}\int^t_0\|\nabla v\|_{L_{q_2}}\,\d s
\leq C\|\pd_t v\|_{L_{q_1}}\int^t_0\|\nabla^2 v\|_{L_{q_1}}\,\d s
\end{align*}
and
\begin{align*}
\|v\nabla v\|_{L_{q_0}} \leq \|v\|_{L_{q_1}}\|\nabla v\|_{L_{q_2}} \leq \|v\|_{L_{q_1}}
\|\nabla^2v\|_{L_{q_1}}.
\end{align*}
Combining these estimates we get
$$\|(1+t)\pd_t N_4\|_{L_p(\BR_+; L_{q_0})} 
\leq C(\|(1+t)v_t\|_{L_p(\BR_+; L_{q_1})}
+ \|v\|_{L_\infty(\BR_+, L_{q_1})})\|(1+t)\nabla^2v\|_{L_p(\BR_+; L_{q_1})}.
$$
Similarly, 
for $i=1,2$, we have 
\begin{align*}
\|\pd_tv \int^t_0\nabla v\,\d s\|_{L_{q_i}}
&\leq \|\pd_tv\|_{L_{q_i}}\int^t_0\|\nabla v\|_{L_{\infty}}\,\d s\\
&\leq C\|\pd_t v\|_{L_{q_i}}\int^t_0 \|\nabla^2 v\|_{L_{q_1}}^\alpha\|\nabla^2
v\|_{L_{q_2}}^{1-\alpha}\,\d s, 
\end{align*}
and 
\begin{align*}
\|v\nabla v\|_{L_{q_i}} \leq \|v\|_{L_{q_i}}\|\nabla v\|_{L_{\infty}} \leq \|v\|_{L_{q_i}}
\|\nabla^2v\|_{L_{q_i}}^\alpha\|\nabla^2v\|_{L_{q_2}}^{1-\alpha}.
\end{align*}
for $i = 1,2$.  Thus, we have
\begin{align*}
&\|(1+t)\pd_t N_4\|_{L_p((0, T); L_{q_i})} \\
&\quad \leq C(\|(1+t)\pd_tv\|_{L_p((0, T); L_{q_i})} + \|v\|_{L_\infty(\BR_+; L_{q_i})})
\sum_{\ell=1}^2\|(1+t)\nabla^2v\|_{L_p(\BR_+; L_{q_\ell})}.
\end{align*}
for $i=1,2$. 
It remains to estimate $N_4$. For this purpose we observe that 
$$\|N_4\|_{L_{q_i}} = \|v\int^t_0\nabla v\,ds\|_{L_{q_i}}
\leq  \|v\|_{L_{q_i}}\int^t_0\|\nabla v\|_{L_\infty}
\leq C\|v\|_{L_{q_i}}\int^t_0(\|\nabla^2v\|_{L_{q_1}}+\|\nabla^2v\|_{L_{q_2}})\,ds.
$$
Therefore, for $i=1,2$, we have
$$\|N_4\|_{L_p((0, T); L_{q_i})} \leq C\|v\|_{L_p((0, T); L_{q_i})}\sum_{\ell=1}^2
\|(1+t)\nabla^2v\|_{L_p((0, T); L_{q_\ell})}.
$$

Summing up, we have for $1<T\leq\infty$
\begin{equation}\label{close:6}\begin{aligned}
&\sum_{\ell=0}^2\|(1+t)\pd_t N_4\|_{L_p((0,T); L_{q_\ell})} 
+ \sum_{\ell=1}^2\|N_4\|_{L_p((0, T); L_{q_\ell})}\\
&\leq C(\sum_{\ell=1}^2(\|(1+t)\pd_tv\|_{L_p((0,T); L_{q_\ell})}
+\|v\|_{L_\infty((0,T); L_{q_\ell})} \\
&\qquad+\|v\|_{L_p((0, T); L_{q_\ell})}))
(\sum_{\ell=1}^2\|(1+t)\nabla^2v\|_{L_p((0,T); L_{q_\ell})}),
\end{aligned}\end{equation}
which, by \eqref{struct:G} and \eqref{def:ETW} implies 
\begin{equation} \label{close:2A}
\sum_{\ell=0}^2\|(1+t)\pd_t \bG(v)\|_{L_p((0,T); L_{q_\ell})} 
+ \sum_{\ell=1}^2\|\bG(v)\|_{L_p((0, T); L_{q_\ell})} \leq C\, E_T(v)^2.
\end{equation}
In the inhomogeneous case we have to estimate \eqref{nonlin:inhomo}, which is straightforward as we have\
$$
\|(1-\rho_0)(1+t)\de_t v\|_{L_p((0,T);L_{q_i})} \leq \|1-\rho_0\|_{L_\infty}\|(1+t)\de_t v\|_{L_p((0,T);L_{q_i})}, \quad i=1,2
$$
and
$$
\|(1-\rho_0)(1+t)\de_t v\|_{L_p((0,T);L_{q_0})} \leq \|1-\rho_0\|_{L_{q_2}}\|(1+t)\de_t v\|_{L_p((0,T);L_{q_1})}.
$$
Combining these estimates with assumption \eqref{small} we get 
\begin{equation} \label{est:pert}
\sum_{i=0}^2 \|(1-\rho_0)(1+t)\de_t v\|_{L_p((0,T);L_{q_i})} \leq \epsilon E_T(v). 
\end{equation}
\subsection{Interpolation in action - estimate of the boundary term}
\label{sec:interp}

In order to complete the proof of \eqref{apriori}, we need to show
\begin{equation} \label{est:half}
\sum_{i=0}^2\|(1+t)\bH(v)\|_{\dot H^{1/2}((0,T);L_{q_i})} \leq C E_T(v)^2, \quad 1 < T \leq \infty,
\end{equation}
where $E_T(\cdot)$ is defined in \eqref{def:ETW}.
For this purpose we apply complex interpolation theory. In order to be more precise we take into account the exact structure \eqref{struct:lag} in the nonlinear terms, instead of considering simply $N_3$ from \eqref{nonlin}. We prove a general result, which provides estimates of quite general form of nonlinearities in fractional Sobolev space in time. The result is of independent interest, as it creates a framework of estimating nonlinearities arising in transformation in Lagrangian coordinates in terms of $L_p-L_q$ estimates of the functions which appear in the nonlinearity. The result reads as follows.
\begin{lem} \label{l:interp}
Consider the operator 
\begin{equation} \label{def:S}
    Sf = f B(\int_0^t g \,d\tau),
\end{equation}
where 
$B(0)=0$  and $ B \in C^1$  in values of  $g$.
Assume that 
\begin{equation}
\frac{1}{r_0}=\frac{1}{r_1}+\frac{1}{r_2}
\end{equation}
and define 
\begin{equation} \label{c:20}
\begin{aligned}
&I_0(g):=\|g\|_{L_1(\R_+;L_{r_2})} +\|tg\|_{L_p(\HS;L_{r_2})},\\[3pt]
&I_1(g):=\|g\|_{L_1(\R_+;L_{r_2})} +\|tg\|_{L_{\infty}(\HS;L_{r_1})},\\[3pt]
&\frac{1}{r_1^+}=\frac{1}{2r_1}+\frac{1}{2r_2}.
\end{aligned}
\end{equation}
Then 
\begin{align} 
    &\|Sf\|_{\dot H^{1/2}_p(\R_+;L_{r_0})}\leq
    C I_0(g) (\|f\|_{\dot H^{1/2}_p(\R_+;L_{r_1})} +
    \|f\|_{L_{2p}(\R_+,(\frac{dt}{t^{1/2}});L_{r_1})})\label{half:1}\\[3pt]
    &\|Sf\|_{\dot H^{1/2}_p(\R_+;L_{r_0})}\leq
    C I_1(g) (\|f\|_{\dot H^{1/2}_p(\R_+;L_{r_1})} +
    \|f\|_{L_{p}(\R_+,(\frac{dt}{t^{1/2}});L_{r_1^+})}),  \label{half:2}
\end{align}
\end{lem}
\begin{rmk}
As already mentioned, the structure of the nonlinear term $Sf$ fits the needs of application of Lagrangian coordinates. Moreover, as will be seen from the proof, various modifications of estimates \eqref{half:1},\eqref{half:2} can be obtained in a similar way. However, we restrict  ourselves here to these three estimates which are necessary to estimate $\bH(v)$ in the required regularity.
\end{rmk}
\begin{proof}
{\bf Proof of \eqref{half:1}}. We start with observing that under conditions \eqref{c:20} we have
\begin{equation} \label{int:10}
    \|Sf\|_{L_p(\R_+;L_{r_0})} \leq C \|g\|_{L_1(\R_+;L_{r_2})}
    \|f\|_{L_p(\R_+;L_{r_1})}.
\end{equation}
Next, let us note that
\begin{equation} \label{int:10a}
\begin{aligned}
&\| \partial_t (f B(\int_0^t g d\tau) \|_{L_p(\R_+;L_{r_0}) }\leq 
\| (\partial_t f) B(\int_0^t g d\tau) \|_{L_p(\R_+;L_{r_0}) } +
\| f  B'(\int_0^t g) g \|_{L_p(\R_+;L_{r_0}) } \leq \\
&C \|g\|_{L_1(\R_+;L_{r_2})} \|\partial_t f\|_{L_p(\R_+;L_{r_1})} + C
\|t g  \|_{L_p(\R_+;L_{r_2})}\| f \frac{1}{t}\|_{L_\infty(\R_+;L_{r_1})},
\end{aligned}
\end{equation}
therefore
\begin{equation} \label{int:11}
    \|\partial_t (Sf)\|_{L_p(\R_+;L_{r_0})}
    \leq C \,I_0(g) \, (\|\partial_t f\|_{L_p(\R_+;L_{r_1})} +
    \|f\|_{L_\infty(\R_+,(\frac{dt}{t});L_{r_1})}).
\end{equation}
Now we take the complex interpolation between \eqref{int:10} 
and \eqref{int:11}. Applying the following particular case of \eqref{complex:int}:
\begin{equation}
\left( L_p(\R_+,(dt);L_{r_1}),L_{\infty}(\R_+,(t^{-1}dt);L_{r_1})) \right)_{1/2} = L_{2p}(\R_+,(t^{-1/2}dt);L_{r_1})),
\end{equation}
we get \eqref{half:1}. 

{\bf Proof of \eqref{half:2}.} We use again \eqref{int:10}, but we slightly modify the estimate of the time derivative, namely, we replace \eqref{int:10a} with  
\begin{equation} \label{int:12}
\begin{aligned}
&\| \partial_t (f B(\int_0^t g d\tau) \|_{L_p(\R_+;L_{r_0}) }\leq 
\| (\partial_t f) B(\int_0^t g d\tau) \|_{L_p(\R_+;L_{r_0}) } +
\| f  B'(\int_0^t g) g \|_{L_p(\R_+;L_{r_0}) } \leq \\
&C \|g\|_{L_1(\R_+;L_{r_2})} \|\partial_t f\|_{L_p(\R_+;L_{r_1})} + C
\|t g  \|_{L_{\infty}(\R_+;L_{r_1})}\| f \frac{1}{t}\|_{L_p(\R_+;L_{r_2})}, 
\end{aligned}
\end{equation}
which yields
\begin{equation} \label{int:13}
    \|\partial_t (Sf)\|_{L_p(\R_+;L_{r_0})}
    \leq C \,I_1(g) \, (\|\partial_t f\|_{L_p(\R_+;L_{r_1})} +
    \|f\|_{L_p(\R_+,(\frac{dt}{t});L_{r_2})}).
\end{equation}
Now we interpolate \eqref{int:10} with \eqref{int:13}. Applying
\begin{equation}
\left( L_p(\R_+,(dt);L_{r_1}),L_p(\R_+,(t^{-1}dt);L_{r_2})) \right)_{1/2} = L_{p}(\R_+,(t^{-1/2}dt);L_{r_1^+})),
\end{equation}
where $r_1^+$ is defined in \eqref{c:20},  we get \eqref{half:2}. 
This completes the proof of Lemma \ref{l:interp}.
\end{proof}
\begin{rmk}
As can be seen from the proof, we can replace $t$ by $(1+t)$ in Lemma \ref{l:interp}.
\end{rmk}
We  now apply Lemma \ref{l:interp} (with $t$ replaced by $(1+t)$) to prove \eqref{est:half}. For this purpose we set 
\begin{equation} \label{def:fgB}
f=(1+t)\nabla v, \;\; g=\nabla v, \;\; B(Z)=A(Id+Z)-Id.
\end{equation}
\vskip5mm
\noindent
{\bf Estimate of $\| (1+t) \bH(v)\|_{
    \dot H^{1/2}_p(\R_+;L_{q_0})}$}.  
Taking \eqref{def:fgB} and $r_i=q_i, \;\; i=0,1,2$ in \eqref{half:1} 
\begin{equation} \label{int:32}
    \| (1+t) \bH(v)\|_{
    \dot H^{1/2}_p(\R_+;L_{q_0})
    }\leq C \,I_0(\nabla v)\, \left (
    \| (1+t) \nabla v \|_{
    \dot H^{1/2}_p(\R_+;L_{q_1})
    } + \| (1+t)^{1/2} \nabla v\|_{L_{2p}(\R_+;L_{q_1})}\right).
\end{equation}
Note that
\begin{equation}
   \| (1+t)^{1/2} \nabla v\|_{L_{2p}(\R_+;L_{q_1})} \leq C \|(1+t)\nabla^2 v\|_{L_p(\R_+,L_{q_1})}^{1/2}
   \|v\|_{L_\infty(\R_+;L_{q_1})}^{1/2}.
\end{equation}
Therefore, recalling the definition of $I_0(\nabla v)$ and the estimate 
$$\|(1+t)\nabla v\|_{L_p(L_{q_2})}\leq C \|(1+t)\nabla^2 v\|_{L_p(L_{q_1})}$$
we can rewrite \eqref{int:12} as 
\begin{equation} \label{Hq0:1}
\begin{aligned}
\| (1+t) \bH(v)\|_{
    \dot H^{1/2}_p(\R_+;L_{q_0})
    }\leq & \,C \, \left(\|\nabla v\|_{L_1(\R_+;L_{q_2})}+\|(1+t)\nabla^2 v\|_{L_p(\R_+;L_{q_1})}\right) \\ 
    & \hskip-1cm \times \left (
    \| (1+t) \nabla v \|_{
    \dot H^{1/2}_p(\R_+;L_{q_1})
    } + \|(1+t)\nabla^2 v\|_{L_p(\R_+,L_{q_1})}^{1/2}
   \|v\|_{L_\infty(\R_+;L_{q_1})}^{1/2} \right).
\end{aligned}
\end{equation}
In order to estimate $\|\nabla v\|_{L_1(\R_+;L_{q_2})}$,
by the relation between $q_1$ and $q_2$ we get immediately  that $\dot B^{2/p}_{q_1,1} \subset  \dot B^{2/p-1}_{q_2,1}$,  keeping in mind that $2/p-1 <0$. So we conclude that
for some $0<\sigma<1$ we have
\begin{equation} \label{int:14}
\begin{aligned}
    \int_0^\infty \|\nabla v\|_{L_{q_2}} dt \leq 
    \int_0^\infty \| v\|_{\dot B^1_{q_2,1}} dt &\leq 
    C(\int_0^\infty \|v\|_{\dot B^{2/p-1}_{q_2,1}} dt)^{1-\sigma} (\int_0^\infty \|v\|_{\dot B^{d/q_2}_{q_2,1}} dt)^{\sigma}\\
    &\leq C(\int_0^\infty \|v\|_{\dot B^{2/p}_{q_1,1}} dt)^{1-\sigma} (\int_0^\infty \|v\|_{\dot B^{d/q_2}_{q_2,1}} dt)^{\sigma}.
\end{aligned}
\end{equation}
As $\frac{d}{q_2},\frac{2}{p}<1$, we have 
\begin{equation}\label{int:14a}
   \int_0^\infty \left(\|v\|_{\dot B^{d/q_2}_{q_i,1}}+\|v\|_{\dot B^{2/p}_{q_i,1}}\right) dt \leq C\|(1+t)v\|_{L_p(\R_+;W^2_{q_i})} \leq C E_T(v),
\end{equation}
which, combined with \eqref{Hq0:1} and \eqref{int:14} gives
\begin{equation} \label{int:15}
\| (1+t) \bH(v)\|_{
    \dot H^{1/2}_p(\R_+;L_{q_0})
    }\leq \,C \, E_T(v) 
    \left(
    \| (1+t) \nabla v \|_{
    \dot H^{1/2}_p(\R_+;L_{q_1})
    } + [E_T(v)]^{1/2}
   \|v\|_{L_\infty(\R_+;L_{q_1})}^{1/2} \right).
\end{equation}
Now, it is enough to observe that
$$
\|v\|_{L_\infty(\R_+;L_{q_1})}\leq C \|v\|_{L_\infty(\R_+;W^{2-2/p}_{q_1})}\leq C \, E_{+\infty}(v)  
$$
and, due to \eqref{imbed:3},
\begin{equation} \label{int:16}
\|(1+t) \nabla v\|_{\dot H^{1/2}_p(\R_+;L_{q_1})} \leq \|(1+t)v\|_{W^{2,1}_{q_1,p}}  \leq E_{+\infty}(v). 
\end{equation}
Plugging these estimates to \eqref{int:15} we finally obtain
\begin{equation*} 
\| (1+t) \bH(v)\|_{
    \dot H^{1/2}_p(\R_+;L_{q_0})
    }\leq C E_T(v)^2.
\end{equation*}
However, we can replace $\R_+$ with $(0,T)$ in all the estimates, which yields
\begin{equation} \label{close:10}
\| (1+t) \bH(v)\|_{
    \dot H^{1/2}_p((0,T);L_{q_0})
    }\leq C E_T(v)^2, \quad 1<T\leq\infty.
\end{equation}

\vskip5mm
\noindent
{\bf Estimate of $\|(1+t)\bH(v)\|_{\dot H^{1/2}_p(\R_+;L_{q_1})}$.} 
We apply \eqref{half:2} with $r_0=r_1=q_1$ and $r_2=\infty$, which implies 
$r_1^+=2q_1$.
Taking \eqref{def:fgB} in \eqref{half:2} we obtain 
\begin{equation}
\begin{aligned}
&\|(1+t)\bH(v)\|_{\dot H^{1/2}_p(\R_+;L_{q_1})}\\
&\leq C \left( \|\nabla v\|_{L_1(\R_+;L_\infty)} + \|(1+t) \nabla v\|_{L_{\infty}(\R_+;L_{q_1})} \right) 
\\
& \hskip2cm
\left(\|(1+t) \nabla v\|_{\dot H^{1/2}_p(\R_+;L_{q_1})} +
    \|(1+t)^{1/2}\nabla v\|_{L_{p}(\R_+,L_{2q_1})}\right).
\end{aligned}
\end{equation}
By \eqref{sup:L1} we have 
\begin{equation} \label{int:30}
\|\nabla v\|_{L_1(\R_+;L_\infty)} \leq C E_{+\infty}(v),
\end{equation}
and, by \eqref{para},
\begin{equation} \label{int:31}
\|(1+t) \nabla v\|_{L_{\infty}(\R_+;L_{q_1})} \leq C \|(1+t)v\|_{L_\infty(\R_+;W^{2-2/p}_{q_1})} \leq C E_{+\infty}(v).
\end{equation}
It remains to show
\begin{equation} \label{int:33}
\|(1+t)^{1/2}\nabla v\|_{L_{p}(\R_+,L_{2q_1})} \leq C E_T(v).
\end{equation}
For this purpose notice that,  we have
$$
\|(1+t)^{1/2}\nabla v\|_{2q_1} \leq \|(1+t)^{1/2}\nabla v\|_{q_1}^{1/2}\|\nabla v\|_{\infty}^{1/2}, 
$$
therefore, by \eqref{sup:1}, 
\begin{equation*}
\begin{aligned}
&\|(1+t)^{1/2}\nabla v\|_{L_p(\R_+;L_{2q_1})} \leq \|(1+t)^{1/2}\nabla v\|_{L_\infty(\R_+;L_{q_1})}^{1/2}\|\nabla v\|_{L_p(\R_+;L_\infty)}^{1/2}\\
&\leq \|(1+t)^{1/2}\nabla v\|_{L_\infty(\R_+;L_{q_1})}^{1/2}\|\nabla^2 v\|_{L_p(\R_+;L_{q_1})}^{(1-\alpha)/2}\|\nabla^2 v\|_{L_p(\R_+;L_{q_2})}^{\alpha/2} \leq E_{+\infty}(v),
\end{aligned}
\end{equation*}
which gives \eqref{int:33}.

Combining \eqref{int:16},\eqref{int:30},\eqref{int:31} and \eqref{int:33}  we obtain 
\begin{equation} \label{close:12}
\|(1+t) \bH(v)\|_{\dot H^{1/2}_p((0,T),L_{q_1})} \leq CE_T(v)^2,
\end{equation}
and, replacing $\R_+$ with $(0,T)$ in all the estimates,
\begin{equation} \label{close:11}
\|(1+t) \bH(v)\|_{\dot H^{1/2}_p(\R_+,L_{q_1})} \leq CE_{+\infty}(v)^2, \quad 0<T \leq \infty. 
\end{equation}
\vskip5mm
\noindent
{\bf Estimate of $\|(1+t)\bH(v)\|_{\dot H^{1/2}_p(\R_+,L_{q_2})}$} 
We apply \eqref{half:2} with $r_0=r_1=q_2$ and $r_2=\infty$, which implies 
$r_1^+=2q_2$.
Taking \eqref{def:fgB} in \eqref{half:2} we obtain 
\begin{equation} \label{est:q2}
\begin{aligned}
&\|(1+t)\bH(v)\|_{\dot H^{1/2}_p(\R_+;L_{q_2})}\\
&\leq C \left( \|\nabla v\|_{L_1(\R_+;L_\infty)} + \|(1+t) \nabla v\|_{L_{\infty}(\R_+;L_{q_2})} \right) 
\\
& \hskip2cm \left(\|(1+t) \nabla v\|_{\dot H^{1/2}_p(\R_+;L_{q_2})} +
    \|(1+t)^{1/2}\nabla v\|_{L_{p}(\R_+,L_{2q_2})}\right).
\end{aligned}
\end{equation}
Repeating the estimates \eqref{int:31} and \eqref{int:33} with $q_2$ instead of $q_1$ we get 
\begin{equation} \label{int:40}
\|(1+t) \nabla v\|_{L_{\infty}(\R_+;L_{q_2})} + \|(1+t)^{1/2}\nabla v\|_{L_{p}(\R_+,L_{2q_2})} \leq E_{+\infty}(v). 
\end{equation}
Next, similarly to \eqref{int:16} we get 
\begin{equation} \label{int:41}
\|(1+t) \nabla v\|_{\dot H^{1/2}_p(\R_+;L_{q_2})} \leq E_{+\infty}(v). 
\end{equation}
Plugging \eqref{int:30},\eqref{int:40} and \eqref{int:41} into \eqref{est:q2} we obtain 
\begin{equation} \label{close:13}
\|(1+t) \bH(v)\|_{\dot H^{1/2}_p((0,T),L_{q_2})} \leq CE_T(v)^2.
\end{equation}
\noindent
{\bf Proof of \eqref{est:half}.} Combining \eqref{close:10}, \eqref{close:11} and \eqref{close:12}
we arrive at \eqref{est:half}.

\subsection{Proof of Theorem \ref{thm:gwp}}

Combining \eqref{close:1A}, \eqref{close:2A} and \eqref{est:half}, we arrive at the final nonlinear estimate
$$\sum_{i=1}^2 \CF^1_{q_i}(\bW) + \CF^2_{q_0}(\bW) 
\leq CE_T(\bW)^2,
$$ 
which, taking additionally into account \eqref{est:pert} in the inhomogeneous case, completes the proof of \eqref{apriori}. As explained at the beginning of this section, this allows to conclude Theorem \ref{thm:gwp}. 

\section{Linear theory in Lorentz spaces} \label{sec:Lorentz:lin}

In this section we prove Lorentz estimates estimates for the Stokes system. We start with the half space case, where, relying on the techniques developed in \cite{DMT}, we combine the estimates derived in Theorems \ref{thm:max} and \ref{thm:weight} with interpolation properties of Lorentz spaces. Next we prove analogous result for the perturbed half space, applying the transformation introduced in proofs of Theorems \ref{thm:max:pert} and \ref{thm:weight:pert}.

\subsection{Estimates in the half space}
In this section we prove the following estimate in Lorentz spaces for the Stokes system \eqref{eq:1}.
Similarly as in the case of $L_p-L_q$ estimates, the result in the half space holds for general exponents $1<\oqq<\oq$ , but in case of perturbed half-space we need additional contraints.
\begin{lem} \label{l:Stokes:Lorentz}   
Let $\bW$ solve \eqref{eq:1}. Assume $1 < \oqq < \oq, \; 1 < p < \infty$, let
\begin{equation} \label{theta:p}
\bar \theta = d\left( \frac{1}{\oqq}-\frac{1}{\oq} \right), \quad
\frac{1}{p}<\frac{\bar \theta}{2}.
\end{equation}
Then, we have
\begin{equation} \label{Stokes:Lorentz:1}
\begin{aligned} 
&\|(\pd_t\bW, \nabla^2\bW, \nabla \sfP)\|_{L_{(p,1)}(\BR_+, L_{\oq}(\HS))} \\
&\quad \leq C\left(\|\bU_0\|_{\dot B^{2(1-1/p)}_{\oq,1}(\HS)} 
+ \|(\bF, \nabla \Gdiv, \pd_t\bG, \nabla \bH, \Lambda^{1/2}\bH)\|_{L_{(p,1)}(\BR_+, L_{\oq}(\HS))}\right),
\end{aligned}
\end{equation}
and
\begin{equation} \label{Stokes:Lorentz:2}
\begin{aligned} 
\|\bW\|_{L_{(p,1)}(\R_+;L_{\oq}(\HS))} & \leq C \Big( \|\bU_0\|_{L_{\oqq}(\HS))} + \|\bU_0\|_{L_{\oq}(\HS))}\\ 
& + \|(\bF, \nabla \Gdiv, \pd_t\bG, \nabla \bH, \Lambda^{1/2}\bH)\|_{L_{(p,1)}((0,1);L_{\oq}(\HS))}\\
& + \|(1+t)(\bF, \nabla \Gdiv, \pd_t\bG, \nabla \bH, \Lambda^{1/2}\bH)\|_{L_{(p,1)}(\R_+;L_{\oqq}(\HS))}  
\Big),
\end{aligned}
\end{equation}
\end{lem}
\begin{proof} 
We use $L_p-L_q$ estimates proved in Theorems \ref{thm:max} and \ref{thm:weight} combined with
interpolation properties of Lorentz spaces. Namely, for any Banach space $A$ we have \cite[Theorem 2:1.18.6]{Tr}
\begin{equation}
    (L_{p_0,r_0}(0,T;A);L_{p_1,r_1}(0,T;A))_{\sigma,r}=L_{p,r}(0,T;A)  
\end{equation}
with 
\begin{equation} \label{c:interp}
\frac{1}{p}=\frac{\sigma}{p_0}+\frac{1-\sigma}{p_1}, \quad \sigma \in (0,1).
\end{equation}
In particular, taking $r_0=p_0$ and $r_1=p_1$ we obtain 
\begin{equation} \label{interp:lorentz}
    (L_{p_0}(0,T;A);L_{p_1}(0,T;A))_{\sigma,r}=L_{p,r}(0,T;A) 
\end{equation}
for $p,p_0,p_1,\sigma$ satisfying \eqref{c:interp}.
By Theorem \ref{thm:max} we have 
\begin{align} \label{lsq:1}
&\|(\pd_t\bW, \nabla^2\bW, \nabla \sfP)\|_{L_s(\BR_+, L_{\oq}(\HS))} \nonumber\\
&\quad \leq C(\|\bU_0\|_{\dot B^{2(1-1/s)}_{\oq,s}(\HS)} 
+ \|(\bF, \nabla \Gdiv, \pd_t\bG, \nabla \bH, \Lambda^{1/2}\bH)\|_{L_s(\BR_+, L_{\oq}(\HS))})
\end{align}
for $1<s,\oq<\infty$.
Moreover, by \eqref{weight:3}, \eqref{weight:3b} and \eqref{def:F}
\begin{align} 
    &\|\bU\|_{L_s((0,1); L_{\oq}(\HS))} \leq C  \|(\bF, \nabla \Gdiv, \pd_t\bG, \nabla \bH, \Lambda^{1/2}\bH)\|_{L_s((0,1);L_{\oq}(\HS))} \label{lsq:2}\\
    &\|\bU\|_{L_s((1,+\infty); L_{\oq}(\HS))} \leq C  \|(1+t)(\bF, \nabla \Gdiv, \pd_t\bG, \nabla \bH, \Lambda^{1/2}\bH)\|_{L_s(\R_+;L_{\oqq}(\HS))}, \label{lsq:3} 
\end{align}
Both \eqref{lsq:2} and \eqref{lsq:3} hold provided $\frac{1}{s}<\frac{\bar \theta}{2}$.

Finally, for the initial conditions, we can assume $\bU(0)=0$
in \eqref{semi:2} and \eqref{semi:3}. Therefore, by \eqref{semi:2} we obtain
\begin{equation} \label{lsq:4}
    \|\bV\|_{L_{p}((0,1), L_{\oq}(\HS))} \leq C\|\bV\|_{L_{\infty}((0,1), L_{\oq}(\HS))} \leq C\|\bU_0\|_{L_{\oq}(\HS))}.
\end{equation}  
Next, by \eqref{semi:3}, we have in particular   
$$
\|\bV(t)\|_{L_{\oq}(\HS)} \leq Ct^{-\frac{\theta}{2}}\|\bU_0\|_{L_{\oqq}(\HS)},
$$
which, due to \eqref{theta:p}, implies
\begin{equation}  \label{lsq:5}
    \|\bV\|_{L_p((1,+\infty); L_{\oq}(\HS))} \leq C\|\bU_0\|_{L_{\oqq}(\HS))}.
\end{equation}
Now, if \eqref{theta:p} holds, then we can choose $1<p_1<p<p_2<\infty$ such that 
$$
\frac{1}{p_k}<\frac{\bar \theta}{2} \;\; {\rm for} \;\; k=1,2
\quad {\rm and} \quad \frac{1}{p}=\frac{1}{2 p_1}+\frac{1}{2 p_2}.
$$
Applying the interpolation inequality \eqref{interp:lorentz} with 
$\sigma=1/2, \, r=1$ and $A=L_{q_1}(\HS)$ to \eqref{lsq:1} with $s \in \{p_1,p_2\}$
we obtain \eqref{Stokes:Lorentz:1}
since, by \eqref{0:3a}, we have 
$$
\left( \dot B^{2-2/p_1}_{\oq,p_1},\dot B^{2-2/p_2}_{\oq,p_2} \right)_{\frac{1}{2},1} = \dot B^{2-2/p}_{\oq,1}.
$$
Similarly, applying \eqref{interp:lorentz} with 
$\sigma=1/2,\,r=1$ and $A=L_{q_1}(\HS)$ to \eqref{lsq:2}  and $A=L_{q_0}(\HS)$ to \eqref{lsq:3} with $s \in \{p_1,p_2\}$ we get 
\begin{align} 
    &\|\bU\|_{L_{(p,1)}((0,1); L_{\oq}(\HS))} \leq C  \|(\bF, \nabla \Gdiv, \pd_t\bG, \nabla \bH, \Lambda^{1/2}\bH)\|_{L_{(p,1)}((0,1);L_{\oq}(\HS))} \label{lorentz:2},\\
    &\|\bU\|_{L_{(p,1)}((1,+\infty); L_{\oq}(\HS))} \leq C  \|(1+t)(\bF, \nabla \Gdiv, \pd_t\bG, \nabla \bH, \Lambda^{1/2}\bH)\|_{L_{(p,1)}(\R_+;L_{\oqq}(\HS))}. \label{lorentz:3} 
\end{align}
Finally, \eqref{lsq:4} and \eqref{lsq:5} with $p$ replaced by $s\in\{p_1,p_2\}$ as before imply
\begin{equation} \label{V:lorentz}
\|\bV\|_{L_{(p,1)}(\R_+;L_{\oq}(\HS))} \leq C \left( \|\bU_0\|_{L_{\oqq}(\HS))} + \|\bU_0\|_{L_{oq}(\HS))} \right).
\end{equation}
Combining \eqref{V:lorentz},\eqref{lorentz:2} and \eqref{lorentz:3}
we get \eqref{Stokes:Lorentz:2}.
\end{proof}
From the above lemma we deduce
\begin{cor} Under the assumptions of Lemma \ref{l:Stokes:Lorentz} we have
\begin{align} \label{Stokes:Lorentz}
\|\bW\|_{W^{2,1}_{\oq,(p,1)}(\HS \times \R_+)} & \leq C(\|(1+|t|)(\bF,\nabla \Gdiv, \pd_t \bG, \nabla \bH, \Lambda^{1/2}\bH)\|_{L_{p,1}(\BR_+, L_{\oqq}(\HS))} \nonumber\\
&+ \|(\bF,\nabla \Gdiv, \pd_t \bG, \nabla \bH, \Lambda^{1/2}\bH)\|_{L_{p,1}(\BR_+, L_{\oq}(\HS))} \\ 
&+\|\bU_0\|_{L_{q_0}(\HS)}+\|\bU_0\|_{L_{\oq}(\HS)}
+ \|\bU_0\|_{B^{2-2/p}_{q,1}(\HS))}, \nonumber
\end{align}
where the space $W^{2,1}_{\oq,(p,1)}$ is defined in \eqref{def:space}. 
\end{cor}

Next, in a similar way we derive weighted estimates for the Stokes system in Lorentz spaces. 
\begin{lem} \label{l:weight:lorentz}
Assume $1<p,q<1$. Let $\bW$ solve \eqref{eq:1}. Then 
\begin{align} \label{weight:lorentz}
&\|(1+t)(\pd_t\bW, \nabla^2\bW, \nabla \sfP)\|_{L_{(p,1)}(\BR_+, L_{q}(\HS))} \nonumber\\
&\quad \leq C\left(\|\bU_0\|_{\dot B^{2(1-1/p)}_{q,1}(\HS)} 
+ \|\bW, (1+t)(\bF, \nabla \Gdiv, \pd_t\bG, \nabla \bH, \Lambda^{1/2}\bH)\|_{L_{(p,1)}(\BR_+, L_{q}(\HS))}\right).
\end{align}
\end{lem}
\begin{proof}
By \eqref{maxdec.1} we have 
\begin{equation}\label{maxdec.2}\begin{aligned}
&\|(1+t)(\pd_t, \nabla^2)\bW\|_{L_s((0, T), L_{q}(\BR_+))}
+\|(1+t)\nabla \sfP\|_{L_s((0, T), L_{q}(\HS))}\\
&\quad
\leq C(\|\bU_0\|_{B^{2(1-1/s)}_{q,s}(\HS)}
+ \|\bW,\,(1+t)(\bF, \nabla \Gdiv, \pd_t \bG, \nabla \bH)\|_{L_s((0, T); L_{q}(\HS))}\\
&\quad +\|(1+t)\bH\|_{\dot H^{1/2}_s((0, T); L_{q}(\HS))}).
\end{aligned}\end{equation}
for $1<q,s<\infty$. Applying the interpolation inequality \eqref{interp:lorentz} with $\sigma=1/2,\,r=1$, $A=L_{q_1}(\HS)$ and $s \in \{p_1,p_2\}$ with $\frac{1}{p}=\frac{1}{2p_1}+\frac{1}{2p_2}$
we obtain \eqref{weight:lorentz}.
\end{proof}

\subsection{Estimates in the perturbed half space}
Here we prove Lorentz spaces estimates for the Stokes system \eqref{eq:1:pert}. For this purpose we apply the results from the previous section and the transformation introduced in Section \eqref{sec:lin:pert}.  
\begin{lem} \label{l:Lorentz:pert}
Let $\bW$ solve \eqref{eq:1:pert}, where $\Omega_0$ is defined in \eqref{omega}. Let $p,q_0,q_1,q_2$ satisfy the assumptions of Theorem \ref{thm:gwp}. Then for $i=1,2$ we have
\begin{align} \label{Stokes:Lorentz:1:pert}
&\|(\pd_t\bW, \nabla^2\bW, \nabla \sfP)\|_{L_{p,1}(\BR_+, L_{q_i}(\Omega_0))} \nonumber\\
&\quad \leq C\left(\|\bU_0\|_{\dot B^{2(1-1/p)}_{q_i,1}(\HS)} 
+ \|(\bF, \nabla \Gdiv, \pd_t\bG, \nabla \bH, \Lambda^{1/2}\bH)\|_{L_{p,1}(\BR_+, L_{q_i}(\Omega_0))}\right)\\
&\quad  +\epsilon \left( \|\nabla^2 \bW\|_{L_{p,1}(\BR_+, L_{q_1}(\Omega_0))} + \|\nabla^2 \bW\|_{L_{p,1}(\BR_+, L_{q_2}(\Omega_0))} \right)\nonumber
\end{align}
and
\begin{align} \label{Stokes:Lorentz:2:pert}
\|\bW\|_{L_{p,1}(\R_+;L_{q_i}(\Omega_0))} \leq & C \left( \|\bU_0\|_{L_{q_{i-1}}(\Omega_0)} + \|\bU_0\|_{L_{q_i}(\Omega_0)} \right. \nonumber\\ 
& + \|(\bF, \nabla \Gdiv, \pd_t\bG, \nabla \bH, \Lambda^{1/2}\bH)\|_{L_{p,1}((0,1);L_{q_i}(\HS))}\\ 
&+ \|(1+t)(\bF, \nabla \Gdiv, \pd_t\bG, \nabla \bH, \Lambda^{1/2}\bH)\|_{L_{p,1}(\R_+;L_{q_{i-1}}(\Omega_0))}  
\Big),\\\nonumber
& +\epsilon \|(1+t)\nabla^2 \bW\|_{L_{p,1}(\BR_+, L_{q_1}(\Omega_0))}.
\end{align}
\end{lem}
\begin{proof}
We follow the idea introduced in the proof of Theorem \ref{thm:max:pert}. Therefore, we want to apply the estimates \eqref{Stokes:Lorentz:1} and \eqref{Stokes:Lorentz:2}  to the system \eqref{Stokes:fixed1}. For this purpose we need to estimate the RHS of the latter. Again, we show the details for the most restrictive term in \eqref{div:pert}. We have 
\begin{multline} \label{est:div:10}
    \| \dv_{\Omega_0}\left( \BD_{\Omega_0}(\widetilde \bW)-
 \BD (\widetilde \bW) \right) \|_{L_{p,1}(0,T;L_{q_i}(\HS))} \leq \\
 C\| Id - \nabla \Phi \|_\infty \|\nabla^2 \widetilde \bW\|_{L_{p,1}(0,T;L_{q_i}(\HS))}+
 C\| |\nabla^2 \Phi | \, |\nabla \widetilde \bW| \|_{L_{p,1}(0,T;L_{q_i}(\HS))},
\end{multline}
where we denote $\nabla^2 \Phi:=\nabla^2_x\Phi(t,\Phi^{-1}(t,y))$ similarly as in \eqref{est:div:1}.
We have 
\begin{equation}
\begin{aligned}
&\left\| \||\nabla^2 \Phi|\,|\nabla \tbW| \|_{L_{q_2}(\HS)} \right\|_{L_{p,1}(\R_+)} \leq \left\| \|\nabla^2\Phi\|_{L_{q_2}(\HS)}\|\nabla \tbW\|_{L_\infty(\HS)}  \right\|_{L_{p,1}(\R_+)}\\
&\leq \|\nabla^2 \Phi\|_{L_{q_2}(\HS)} \|\nabla \tbW\|_{L_{p,1}(\R_+,L_{\infty}(\HS))} \leq \epsilon \|\nabla \tbW\|_{L_{p,1}(\R_+,L_\infty(\HS))}\\
& \leq \epsilon \left( \|\nabla^2 \tbW\|_{L_{p,1}(\R_+,L_{q_1}(\HS))} + \|\nabla^2 \tbW\|_{L_{p,1}(\R_+,L_{q_2}(\HS))} 
\right)
\end{aligned}
\end{equation}
and
\begin{equation}
\begin{aligned}
&\left\| \||\nabla^2 \Phi|\,|\nabla \tbW| \|_{L_{q_1}(\HS)} \right\|_{L_{p,1}(\R_+)} \leq \left\| \|\nabla^2\Phi\|_{L_{d}(\HS)}\|\nabla \tbW\|_{L_{q_2}(\HS)}  \right\|_{L_{p,1}(\R_+)}\\
&\leq C \left( \|\nabla^2 \Phi\|_{L_{q_1}(\HS)}+ \|\nabla^2 \Phi\|_{L_{q_2}(\HS)}\right) \|\nabla \tbW\|_{L_{p,1}(\R_+,L_{q_2}(\HS))}\\ 
& \leq \epsilon \|\nabla^2 \tbW\|_{L_{p,1}(\R_+,L_{q_1}(\HS))}.  
\end{aligned}
\end{equation}
Estimating similarly the other terms on the RHS of \eqref{Stokes:fixed1}, we can apply \eqref{Stokes:Lorentz:1} to obtain \eqref{Stokes:Lorentz:1:pert}.
In order to prove \eqref{Stokes:Lorentz:2:pert}
we also need the weighted estimate. By Sobolev imbedding we have $\frac{1}{q_1}=\frac{1}{d}+\frac{1}{q_1^*}$, therefore 
\begin{equation} \label{pert:1}
\begin{aligned}
&\left\| \|(1+t)|\nabla^2 \Phi|\,|\nabla \tbW| \|_{L_{q_1}(\HS)} \right\|_{L_{p,1}(\R_+)}\leq \left\| \|\nabla^2\Phi\|_{L_d(\HS)}\|(1+t)\nabla \tbW\|_{L_{q_2}}  \right\|_{L_{p,1}(\R_+)}\\
&\leq \|\nabla^2 \Phi\|_{L_\infty(\HS)} \|(1+t)\nabla^2 \tbW\|_{L_{p,1}(\R_+,L_{q_1}(\HS))} \leq \epsilon \|(1+t)\nabla^2 \tbW\|_{L_{p,1}(\R_+,L_{q_1}(\HS))}.
\end{aligned}
\end{equation}
Next, by H\"older inequality, 
\begin{equation} \label{pert:2}
\begin{aligned}
&\left\| \|(1+t)|\nabla^2 \Phi|\,|\nabla \tbW| \|_{L_{q_0}(\HS)} \right\|_{L_{p,1}(\R_+)}\leq \left\| \|\nabla^2\Phi\|_{L_{q_1}(\HS)}\|(1+t)\nabla \tbW\|_{L_{q_2}(\HS)}  \right\|_{L_{p,1}\R_+}\\
&\leq \|\nabla^2 \Phi\|_{L_{q_1}(\HS)} \|(1+t)\nabla^2 \tbW\|_{L_{p,1}(\R_+,L_{q_1}(\HS))} \leq \epsilon \|(1+t)\nabla^2 \tbW\|_{L_{p,1}(\R_+,L_{q_1}(\HS))}.
\end{aligned}
\end{equation}
Now we can apply \eqref{Stokes:Lorentz:2} to obtain \eqref{Stokes:Lorentz:2:pert}.
 
\end{proof}

\begin{cor} Under the assumptions of Lemma \ref{l:Lorentz:pert} we have
\begin{align} \label{Stokes:Lorentz:pert}
\|\bW\|_{W^{2,1}_{q_i,(p,1)}(\Omega_0 \times \R_+)} & \leq C\Big(\|(1+t)(\bF,\nabla \Gdiv, \pd_t \bG, \nabla \bH, \Lambda^{1/2}\bH)\|_{L_{p,1}(\BR_+, L_{q_{i-1}}(\Omega_0))} \nonumber\\
&+ \|(\bF,\nabla \Gdiv, \pd_t \bG, \nabla \bH, \Lambda^{1/2}\bH)\|_{L_{p,1}(\BR_+, L_{q_i}(\Omega_0))} \\ 
&+\|\bU_0\|_{L_{q_{i-1}}(\Omega_0)}
+ \|\bU_0\|_{B^{2-2/p}_{q_i,1}(\HS))}\Big) \nonumber\\
& +\epsilon \Big( \|(1+t)\nabla^2 \bW\|_{L_{p,1}(\BR_+, L_{q_1}(\Omega_0))} + \|\nabla^2 \bW\|_{L_{p,1}(\BR_+, L_{q_2}(\Omega_0))} \Big). \nonumber
\end{align}
\end{cor}
\vskip5mm
Next, we extend Lemma \ref{l:weight:lorentz} to the perturbed half space.
\begin{lem}
Let $\bW$ solve \eqref{eq:1:pert}, where $\Omega_0$ is defined in \eqref{omega}. Let $p,q_0,q_1$ satisfy the assumptions of Theorem \ref{thm:gwp}. Then for $i=0,1$ we have
\begin{align} \label{weight:lorentz:pert}
&\|(1+t)(\pd_t\bW, \nabla^2\bW, \nabla \sfP)\|_{L_{(p,1)}(\BR_+, L_{q}(\Omega_0))} \nonumber\\
&\quad \leq C\left(\|\bU_0\|_{\dot B^{2(1-1/p)}_{q,1}(\Omega_0)} 
+ \|\bW, (1+t)(\bF, \nabla \Gdiv, \pd_t\bG, \nabla \bH, \Lambda^{1/2}\bH)\|_{L_{(p,1)}(\BR_+, L_{q}(\Omega_0))}\right)\\
&\quad +\epsilon \|(1+t)\nabla^2 \bW\|_{L_{p,1}(\BR_+, L_{q_1}(\Omega_0))}.\nonumber
\end{align}
\end{lem}
\begin{proof}
Similarly as in the proof of Lemma \ref{l:Lorentz:pert}, the most restrictive terms are 
$$
\left\|(1+t)|\nabla^2 \Phi|\,|\nabla \tbW| \|_{L_{q_i}} \right\|_{L_{p,1}(\R_+)} \quad \text{for} \;\; i=0,1, 
$$
coming from the estimate of
$$
\|(1+t)\dv_{\Omega_0}\BT_{\Omega_0}(\widetilde \bW, \widetilde\sfP)-
(1+t)\dv \BT (\widetilde \bW, \widetilde\sfP)\|_{L_{p,1}(\R_+,L_{q_i})}.
$$
These however have been already derived in \eqref{pert:1} and \eqref{pert:2}.

\end{proof}

\section{Nonlinear estimates in Lorentz spaces} \label{sec:nonlin:Lorentz}

In what follows we assume that our solutions belongs to the following class of the regularity
\begin{equation}
    v \in W^{2,1}_{q_2,(p,1)}(\Omega_0 \times \BR_+)\cap W^{2,1}_{q_1,(p,1)}(\Omega_0 \times \BR_+), \quad \mbox{ and \  \ \ }
    tv \in \dot W^{2,1}_{q_1,(p,1)}(\Omega \times \BR_+) 
\end{equation}
with $q_0,q_1,q_2$ and $p$ satisfying the assumptions of Theorem \ref{thm:Lorentz}, where the space $W^{2,1}_{q,(p,1)}$ is defined in \eqref{def:space}. 
%
%

In this section we estimate the RHS of \eqref{eq-fixed} in the regularity required in the estimate \eqref{Stokes:Lorentz}. 
Therefore, we need to find the bounds on 
\begin{align}
&\|\bF(\bW),\nabla \dv \bG(\bW), \nabla \bH(\bW)\|_{L_{p,1}(\R_+;L_{q_2}(\Omega_0))}, \label{lhs:1}\\
&\|(1+t)\big(\bF(\bW),\nabla \dv \bG(\bW), \nabla \bH(\bW)\big)\|_{L_{p,1}(\R_+;L_{q_1}(\R^d_+)) \cap L_{p,1}(\R_+;L_{q_0}(\Omega_0))  },\label{lhs:2} \\
&\|\de_t \bG(\bW)\|_{L_{p,1}(\R_+;L_{q_2}(\Omega_0))},\label{lhs:3}\\
&\|(1+t) \de_t\bG(\bW)\|_{L_{p,1}(\R_+;L_{q_1}(\R^d_+)) \cap L_{p,1}(\R_+;L_{q_0}(\Omega_0))}\label{lhs:4}\\
&\|\Lambda^{1/2} \bH(\bW)\|_{L_{p,1}(\R_+;L_{q_2}(\Omega_0))}, \label{lhs:5}\\
&\|(1+t) \Lambda^{1/2} \bH(\bW)\|_{L_{p,1}(\R_+;L_{q_1}(\Omega_0)) \cap L_{p,1}(\R_+;L_{q_0}(\Omega_0))}.\label{lhs:6}
\end{align}
For this purpose we need to estimate the nonlinearities $N_1-N_4$
defined in \eqref{nonlin}. Similarly as for the $L_p-L_q$ nonlinear estimates, there is no difference between the half space and perturbed half space cases. Therefore we omit the spatial domain $\Omega_0$ in the notation.  
\subsection{Estimates without time weight}
\begin{lem} \label{l:nonlin1}
Assume that    
\begin{align}
&\frac{1}{d}=\frac{1}{p'q_1}+\frac{1}{pq_2}, \label{c10}\\
&\frac{2d}{3}<q_1<d<q_2=q_1^*. \label{c11}
\end{align}
Then 
\begin{equation} \label{est:10}
\begin{aligned}
\|\bF(v), \nabla \div \bG(v),\nabla \bH(\bW)\|_{L_{p,1}(\R_+;L_{q_2})} 
\leq C \left( \|v\|_{\Xi}^2 + \|\nabla^2 v\|_{L_{p,1}(\R_+;L_{q_2})}^2 + \|t\nabla^2 v\|_{L_{p,1}(\R_+;L_{q_1})}^2  \right),   
\end{aligned}
\end{equation}
where the space $\Xi$ is defined in \eqref{sol:space}.
\end{lem}
Before we proceed with the proof, let us notice that \eqref{c10} implies
\begin{equation}
p=\frac{q_1}{d-q_1}=\frac{q_2}{d}, 
\end{equation}
therefore $p>2$ implies the lower bound $q_1>\frac{2d}{3}$.

{\bf Proof of Lemma \ref{l:nonlin1}.} 
Taking into account \eqref{struct:F} we need to estimate the terms $N_1,N_2$ from \eqref{nonlin} in $L_{p,1}(\R_+;L_{q_2})$. 
Let us start with $N_2$. We have 
$$
\|\nabla v\int_0^t|\nabla^2 v|\dtau\|_{L_{q_2}}\leq 
\|\nabla v\|_{L_\infty} t^{1/p'} \| \nabla^2 v \|_{L_p(\R_+;L_{q_2})},
$$
therefore 
\begin{equation} \label{e1}
\bign \nabla_y v \int_0^t \nabla^2_y v(\tau,y)\dtau \bign_{L_{p,1}(\R_+;L_{q_2}(\R^d_+))}
\leq \| \nabla^2 v \|_{L_p(\R_+;L_{q_2})} \bign t^{1/p'}\|\nabla v\|_\infty\bign_{L_{p,1}(\R_+)}. 
\end{equation} 
In order to estimate $\|\nabla v\|_{\infty}$ we apply \eqref{sup:1}
which yields
\begin{equation} \label{int2}
\|\nabla v\|_{L_\infty} \lesssim \|\nabla^2 v\|_{L_{q_2}}^{1-\alpha}\|\nabla^2 v\|_{L_{q_1}}^{\alpha}.  
\end{equation}
Choosing $\alpha=1-\frac{1}{p}$ in \eqref{int2} 
\begin{equation} \label{e2}
t^{1/p'}\|\nabla v\|_{L_\infty} \lesssim
\|\nabla^2 v\|_{L_{q_2}}^{1/p}\|t\nabla^2 v\|_{L_{q_1}}^{1/p'}.
\end{equation}
Applying \eqref{holder} with $p_1=p^2,p_2=pp',r_1=p,r_2=p'$ and then \eqref{0:2} we get for the first term
\begin{equation*}
\begin{aligned}
&\bign \|\nabla^2 v\|_{q_2}^{1/p}\|t \nabla^2 v\|_{q_1}^{1/p'} \bign_{L_{p,1}(\R_+)} \; 
\leq \bign \|\nabla^2 v\|_{q_2}^{1/p} \bign_{L_{p^2,p}(\R_+)}\bign \|t\nabla^2 v\|_{q_1}^{1/p'} \bign_{L_{pp',p'}(\R_+)}\\[3pt] 
& \leq \|\nabla^2 v\|_{L_{p,1}(\R_+;L_{q_2})}^{1/p}\|t\nabla^2 v\|_{L_{p,1}(\R_+;L_{q_1})}^{1/p'}.
\end{aligned}
\end{equation*}
The second term on the RHS of \eqref{e2} can be treated in the same way, hence from \eqref{e1} we obtain 
\begin{equation} \label{e:1}
\begin{aligned}
&\bign \nabla v \int_0^t \nabla^2 v(\tau,y)\dtau \bign_{L_{p,1}(\R_+;L_{q_2}(\R^d_+))}\\[3pt] 
&\lesssim \| \nabla^2 v \|_{L_p(\R_+;L_{q_2})}
\|\nabla^2 v\|_{L_{p,1}(\R_+;L_{q_2})}^{1/p}\|t\nabla^2 v\|_{L_{p,1}(\R_+;L_{q_1})}^{1/p'}\\[3pt]
&\lesssim \| \nabla^2 v \|_{L_p(\R_+;L_{q_2})}^2 + \|\nabla^2 v\|_{L_{p,1}(\R_+;L_{q_2})}^2 + \|t\nabla^2 v\|_{L_{p,1}(\R_+;L_{q_1})}^2. 
\end{aligned}
\end{equation}
To estimate the term $N_1$ in \eqref{nonlin}, observe that by \eqref{int2} we and \eqref{holder2}  
\begin{equation} \label{e5}
\begin{aligned}
&\int_0^t\|\nabla v(\tau)\|_{L_{\infty}} \dtau \leq 
\int_0^t \tau^{-\alpha}\|\tau\nabla^2 v(\tau)\|_{L_{q_1}}^\alpha\|\nabla^2 v\|_{L_{q_2}}^{1-\alpha}\\[3pt]
&\leq \|t^{-\alpha}\|_{L_{1/\alpha,\infty}(\R_+)}\bign\| t\nabla^2 v \|_{L_{q_1}}^\alpha \bign_{L_{a,1/\alpha}(\R_+)}
\bign\| \nabla^2 v \|_{L_{q_2}}^{1-\alpha} \bign_{L_{b,1/(1-\alpha)}(\R_+)}\\[3pt]
&\lesssim \|t\nabla^2 v\|_{L_{p,1}(\R_+;L_{q_1})}^\alpha\|\nabla^2 v\|_{L_{p,1}(\R_+;L_{q_2})}^{1-\alpha}
\end{aligned}
\end{equation} 
provided 
\begin{equation} \label{c2}
\alpha+\frac{1}{a}+\frac{1}{b}=1, \quad \alpha a=(1-\alpha)b=p.
\end{equation}
Conditions \eqref{c2} imply 
\begin{equation} \label{c3}
\alpha=1-\frac{1}{p},
\end{equation}
which is in accordance with relation necessary to obtain \eqref{e2}. 
Therefore \eqref{e5} reads 
\begin{equation} \label{e4b}
\begin{aligned}
&\int_0^t\|\nabla v(\tau)\|_{L_\infty} \dtau \leq C\|t\nabla^2 v\|_{L_{p,1}(\R_+;L_{q_1})}^{1/p'}\|\nabla^2 v\|_{L_{p,1}(\R_+;L_{q_2})}^{1/p}
\end{aligned}
\end{equation}
This allows us to deduce 
\begin{equation} \label{e:2}
\begin{aligned}
&\left\| \|\nabla^2 v \int_0^t \nabla v(\tau) d\tau\|_{L_{q_2}} \right\|_{L_{p,1}(\R_+)} \leq \left\| \|\nabla^2 v\|_{L_{q_2}} \int_0^t\|\nabla v(\tau) d\tau\|_{L_\infty} \right\|_{L_{p,1}(\R_+)}\\ 
&\leq \|\nabla^2 v\|_{L_{p,1}(\R_+,L_{q_2})}^{1+1/p} \|t\nabla^2 v\|_{L_{p,1}(\R_+,L_{q_1})}^{1/p'}
\leq C \left( \|t\nabla^2 v\|_{L_{p,1}(\R_+;L_{q_1})}^2 + \|\nabla^2 v\|_{L_{p,1}(\R_+;L_{q_2})}^2  \right). 
\end{aligned}
\end{equation}
Estimates \eqref{e:1} and \eqref{e:2} imply \eqref{est:10}.

\begin{flushright}
$\square$  
\end{flushright}

Finally, we estimate the time derivative of the RHS of the divergence equation.
\begin{lem} Assume $p,q_1,q_2$ satisfy \eqref{c10}-\eqref{c11}. Then 
\begin{equation} \label{est:N4}
\begin{aligned}
&\|\de_t \bG(v)\|_{L_{p,1}(\R_+;L_{q_2})} \leq \\ 
& \leq C \left( \|v\|_{W^{2,1}_{q_2,{p,1}}}^2 + \|t\nabla^2 v\|_{L_{p,1}(\R_+,L_{q_1})}^2
+ \|v\|^2_{L_{(p,1)}(\R_+,L_{q_1}) }\right).
\end{aligned}
\end{equation}
\end{lem}
\begin{proof}
Taking into account \eqref{struct:G}, we have to estimate 
\begin{equation} \label{N4:1}
\de_t N_4 \sim v_t \int_0^t\nabla v + v\nabla v.
\end{equation}
Since $v_t$ has the same regularity as $\nabla^2 v$, the first term on the LHS of \eqref{N4:1} can be estimated similarly to $N_2$
using \eqref{e4b}:
\begin{equation} \label{N4:2}
\begin{aligned}
&\|v_t \int_0^t \nabla v(\tau) d\tau\|_{L_{p,1}(\R_+,L_{q_2})} 
\leq \left\| \|v_t\|_{q_2} \int_0^t\|\nabla v(\tau) d\tau\|_{\infty} \right\|_{L_{p,1}(\R_+)}\\ 
&\leq \|v_t\|_{L_{p,1}(\R_+,L_{q_2})} \|\nabla^2 v\|_{L_{p,1}(\R_+,L_{q_2})}^{1/p} \|t\nabla^2 v\|_{L_{p,1}(\R_+,L_{q_1})}^{1/p'}\\
& \leq C \left( \|v_t\|_{L_{p,1}(\R_+,L_{q_2})}^2 + \|\nabla^2 v\|_{L_{p,1}(\R_+,L_{q_2})}^2 + \|t\nabla^2 v\|_{L_{p,1}(\R_+,L_{q_1})}^2\right).
\end{aligned}
\end{equation}
In order to estimate the second term on the RHS of \eqref{N4:1} 
we first observe that 
\begin{equation} \label{N4:3}
\|v \nabla v\|_{L_{q_0}} \leq \|v\|_{q_1}\|\nabla v\|_{L_{q_2}}. 
\end{equation}
On the other hand, by the trace theorem we have 
\begin{equation} \label{N4:3b}
\|\nabla v\|_{L_\infty(\R_+,W^{1-2/p}_{q_2})} \leq \|v\|_{L_\infty(\R_+,W^{2-2/p}_{q_2})} \leq C \|v\|_{W^{2,1}_{q_2,(p,1)}},
\end{equation}
while by the imbedding theorem  
$$
\|\nabla v\|_{L_\infty(\R_+,L_{q_2+\epsilon})} \leq C \|\nabla v\|_{L_\infty(\R_+,W^{1-2/p}_{q_2})} 
$$
for some $\epsilon>0$. This allows us to apply the estimate 
\begin{equation} \label{N4:4}
\|v \nabla v\|_{L_{q_2+\epsilon}} \leq \|v\|_{L_\infty}\|\nabla v\|_{L_{q_2+\epsilon}}. 
\end{equation}
Combining \eqref{N4:3} and \eqref{N4:4} we obtain for certain $\kappa>0$ 
\begin{equation}  \label{N4:5}
\begin{aligned}
\|v\nabla v\|_{L_{q_2}} & \leq \|v\nabla v\|_{L_r}^\kappa \, \|v\nabla v\|_{L_{q_2+\epsilon}}^{1-\kappa}
\leq \|v\|_{L_{q_1}}^\kappa\|v\|_{L_\infty}^{1-\kappa} \|\nabla v\|_{L_{q_2}}^\kappa\|\nabla v\|_{L_{q_2+\epsilon}}^{1-\kappa}\\[3pt]
&\leq C \|v\|_{\dot W^2_{q_2} \cap L_{q_1}} \|\nabla v\|_{W^{1-2/p}_{q_2}},
\end{aligned}
\end{equation}
which, by \eqref{N4:3b} implies 
\begin{equation} \label{N4:6}
\begin{aligned}
\big\| \|v\nabla v\|_{q_2}\big\|_{L_{(p,1)}(\R_+)} 
& \leq C \big\| \|v\|_{\dot W^2_{q_2} \cap L_{q_1}} \|\nabla v\|_{W^{1-2/p}_{q_2}} \big\|_{L_{(p,1)}(\R_+)}\\
& \leq C \|v\|_{\dot W^{2,1}_{q_2,(p,1)} \cap L_{(p,1)}(\R_+,L_{q_1}) }
\|\nabla v\|_{L_\infty(\R_+,W^{1-2/p}_{q_2}(\HS) ) }\\
& \leq C \|v\|_{\dot W^{2,1}_{q_2,(p,1)} \cap L_{(p,1)}(\R_+,L_{q_1}) }
\|v\|_{W^{2,1}_{q_2,(p,1)} }.
\end{aligned}
\end{equation}
Combining \eqref{N4:2} and \eqref{N4:6} we conclude \eqref{est:N4}.

\end{proof}

\subsection{Time weighed nonlinear estimates} 
\begin{lem} \label{l:nonlin2}
Assume $p,q_1,q_2$ satisfy \eqref{c10}-\eqref{c11}. Then
\begin{equation} \label{est:12}
\begin{aligned}
&\|(1+t)\big(\bF(v), \,\nabla \div \bG(v),\,\nabla\bH(\bW) \big)\|_{L_{p,1}(\R_+;L_{q_1})}\\
&\leq C \left( \|\nabla^2 v\|_{L_{p,1}(\R_+;L_{q_2})}^2 + \|(1+t)\nabla^2 v\|^2_{L_{p,1}(\R_+;L_{q_1}} \right),
\end{aligned}
\end{equation}
\begin{equation}\label{est:13}
\begin{aligned}
&\|(1+t)\big(\bF(v), \,\nabla \div \bG(v),\,\nabla\bH(\bW) \big)\|_{L_{p,1}(\R_+;L_{q_0})}\\ 
&\leq C \left( \|(1+t)v\|_{\dot W^{2,1}_{q_1,(p,1)}}^2 + \|v\|_{\Xi}^2 \right). 
\end{aligned}
\end{equation}
\end{lem}
{\bf Proof.}
We start with the bounds on \eqref{lhs:2}.
Similarly as before, the most restrictive term is 
$t N_2$.
By the Sobolev imbedding we have 
\begin{equation} \label{e4}
\begin{aligned}
\bign\|(1+t)\nabla v \int_0^t|\nabla^2 v|d\tau \|_{q_1}\bign_{L_{p,1}(\R_+)}\leq \bign\|(1+t)\nabla v\|_{q_1^*} \|\int_0^t|\nabla^2 v|d\tau \|_{d}\bign_{L_{p,1}(\R_+)}\leq\\[3pt]
\bign\|(1+t)\nabla^2 v\|_{q_1}\bign_{L_{p,1}(\R_+)}\|\int_0^t\nabla^2 v\|_{L_\infty(\R_+;L_d)}.
\end{aligned}
\end{equation}
The condition \eqref{c10} and the interpolation inequality for $L_p$ norms imply
\begin{equation*} 
\int_0^t\|\nabla^2 v(\tau)\|_{L_d} d \tau \leq 
\int_0^t (1+t)^{-1/p'}\|t\nabla^2 v(\tau)\|_{L_{q_1}}^{1/p'}\|\nabla^2 v\|_{L_{q_2}}^{1/p},
\end{equation*} 
so repeating \eqref{e5} we obtain 
\begin{equation}
\int_0^t\|\nabla^2 v(\tau)\|_{L_d} d\tau \lesssim \|(1+t)\nabla^2 v\|_{L_{p,1}(\R_+;L_{q_1})}^{1/p'}\|\nabla^2 v\|_{L_{p,1}(\R_+;L_{q_2})}^{1/p}.
\end{equation}
Combining \eqref{e4} and \eqref{e5} we obtain 
\begin{equation} \label{est:w1}
\begin{aligned}
&\bign\|(1+t)\nabla v \int_0^t|\nabla^2 v|d\tau \|_{L_{q_1}}\bign_{L_{p,1}(\R_+)}\lesssim \|(1+t)\nabla^2 v\|^{1+1/p'}_{L_{p,1}(\R_+;L_{q_1})} \|\nabla^2 v\|_{L_{p,1}(\R_+;L_{q_2})}^{1-1/p'}\\
&\lesssim \|(1+t)\nabla^2 v\|^2_{L_{p,1}(\R_+;L_{q_1})}+\|\nabla^2 v\|_{L_{p,1}(\R_+;L_{q_2})}^2.
\end{aligned}
\end{equation}
for $p,q_1,q_2$ satisfying \eqref{c10}-\eqref{c11}.
The term $(1+t)N_1$ can be estimated using \eqref{e4b}: 
\begin{equation} \label{est:w2}
\begin{aligned}
&\left\| \|(1+t) \nabla^2 v \int_0^t \nabla v(\tau) d\tau\|_{L_{q_1}} \right\|_{L_{p,1}(\R_+)} \leq \left\| \|(1+t) \nabla^2 v\|_{L_{q_1}} \int_0^t\|\nabla v(\tau) d\tau\|_{L_\infty} \right\|_{L_{p,1}(\R_+)}\\[5pt] 
&\leq \|\nabla^2 v\|_{L_{p,1}(\R_+,L_{q_2})}^{1/p} \|(1+t)\nabla^2 v\|_{L_{p,1}(\R_+,L_{q_1})}^{1+1/p'}
\lesssim \|\nabla^2 v\|_{L_{p,1}(\R_+,L_{q_2})}^2 + \|(1+t)\nabla^2 v\|_{L_{p,1}(\R_+,L_{q_1})}^2. 
\end{aligned}
\end{equation}
Estimates \eqref{est:w1} and \eqref{est:w2} imply \eqref{est:12}. 
Let us proceed with the estimates for $q_0$. For $tN_2$ we have 
\begin{equation} \label{est:w1b}
\begin{aligned}
&\left\| \| (1+t)\nabla v\int_0^t |\nabla^2 v|d\tau \|_{q_0}  \right\|_{L_{(p,1)}(\R_+)}
\leq \left\| \| (1+t)\nabla v\|_{q_2}  \|\int_0^t |\nabla^2 v|d\tau \|_{q_1}  \right\|_{L_{p,1}(\R_+)}\\[3pt]
&\leq \left\| \| (1+t)\nabla v\|_{q_2} \right\|_{L_{(p,1)}(\R_+)} 
\|\int_0^t |\nabla^2 v|d\tau\|_{L_\infty(L_{q_1})}\\[3pt]
&\lesssim  \|(1+t)\nabla^2 v\|_{L_{p,1}(\R_+;L_{q_1})}\|(1+t)\nabla^2 v\|_{L_p(\R_+;L_{q_1})}\\[3pt]
&\lesssim \|(1+t)\nabla^2 v\|_{L_{p,1}(\R_+;L_{q_1})}^2 + \|(1+t)\nabla^2 v\|_{L_p(\R_+;L_{q_1})}^2
\leq \|(1+t)\nabla^2 v\|_{L_{p,1}(\R_+;L_{q_1})}^2 + \|v\|_{\Xi}^2,
\end{aligned}
\end{equation}
where in the last passage we have used \eqref{L1:1}. Similarly for $tN_1$: 
\begin{equation} \label{est:w2b}
\begin{aligned}
&\left\| \| (1+t)\nabla^2 v\int_0^t |\nabla v|d\tau \|_{q_0}  \right\|_{L_{(p,1)}(\R_+)}
\leq \left\| \| (1+t)\nabla^2 v\|_{q_1}  \|\int_0^t |\nabla v|d\tau \|_{q_2}  \right\|_{L_{p,1}(\R_+)}\\[3pt]
&\leq \left\| \| (1+t)\nabla^2 v\|_{q_1} \right\|_{L_{(p,1)}(\R_+)} 
\|\int_0^t |\nabla^2 v|d\tau\|_{L_\infty(L_{q_1})}\\[3pt]
&\lesssim \|(1+t)\nabla^2 v\|_{L_{p,1}(\R_+;L_{q_1})}\|(1+t)\nabla^2 v\|_{L_p(\R_+;L_{q_1})}\\[3pt]
&\lesssim \|(1+t)\nabla^2 v\|_{L_{p,1}(\R_+;L_{q_1})}^2 + \|(1+t)\nabla^2 v\|_{L_p(\R_+;L_{q_1})}^2
\leq \|(1+t)\nabla^2 v\|_{L_{p,1}(\R_+;L_{q_1})}^2 + \|v\|_{\Xi}^2.
\end{aligned}
\end{equation}
Combining \eqref{est:w1b} and \eqref{est:w2b} we obtain \eqref{est:13}.
\begin{flushright}
$\square$
\end{flushright}

It remains to prove 
\begin{lem} \label{l:nonlin3} 
Assume $p,q_1,q_2$ satisfy \eqref{c10}-\eqref{c11}. Then
\begin{align}
&\|(1+t) \de_t \bG(v)\|_{L_{p,1}(\R_+;L_{q_1})} \leq 
C \left( \|v\|_{W^{2,1}_{q_1,(p,1)}}^2 + \|(1+t)v\|_{\dot W^{2,1}_{q_1,(p,1)}}^2 + \|v\|_{\Xi}^2 \right),\label{est:nonlin3}\\[5pt]
& \|(1+t)\, \de_t \bG(v)\|_{L_{p,1}(\R_+;L_{q_0})} \leq C \left( \|(1+t)v\|_{\dot W^{2,1}_{q_1,(p,1)}}^2 + \|v\|_{\Xi}^2 \right). \label{est:nonlin4}
\end{align}
\end{lem}
\begin{proof}
We have to estimate 
\begin{equation} \label{tN4:1}
(1+t)\,\de_t N_4 \sim (1+t)\, v_t \int_0^t\nabla v + (1+t)\,v\nabla v.
\end{equation}
The first term can be treated similarly to \eqref{est:w2} applying \eqref{e4b}:
\begin{equation} \label{est:w3}
\begin{aligned}
&\left\| \|(1+t) v_t \int_0^t \nabla v(\tau) d\tau\|_{L_{q_1}} \right\|_{L_{p,1}(\R_+)} \leq \left\| \|(1+t)\,v_t\|_{L_{q_1}} \int_0^t\|\nabla v(\tau) d\tau\|_{\infty} \right\|_{L_{p,1}(\R_+)}\\[5pt] 
&\leq \|\nabla^2 v\|_{L_{p,1}(\R_+,L_{q_2})}^{1/p} \|(1+t)\nabla^2 v\|_{L_{p,1}(\R_+,L_{q_1})}^{1/p'} \|(1+t)\,v_t\|_{L_{p,1}(\R_+,L_{q_1})}\\
&\leq C \left( \|v\|_{\dot W^{2,1}_{q_2,(p,1)}}^2 + \|(1+t)v\|_{\dot W^{2,1}_{q_1,(p,1)}}^2 \right). 
\end{aligned}
\end{equation}
For the second term we have, similarly to \eqref{N4:5}, 
\begin{equation}  \label{tN4:2}
\begin{aligned}
\|(1+t)v\nabla v\|_{q_1} & \leq \|(1+t)v\nabla v\|_{q_0}^\kappa \, \|(1+t)v\nabla v\|_{q_1+\epsilon}^{1-\kappa}\\[3pt]
&\leq \|v\|_{q_1}^\kappa\|v\|_{\infty}^{1-\kappa} \|(1+t)\nabla v\|_{q_2}^\kappa\|(1+t)\nabla v\|_{q_1+\epsilon}^{1-\kappa}\\[3pt]
&\leq C \|v\|_{W^2_{q_1}} \|(1+t)\nabla v\|_{W^{1-2/p}_{q_1} \cap W^{1-2/p}_{q_2} },
\end{aligned}
\end{equation}
which implies 
\begin{equation} \label{est:w4}
\begin{aligned}
\big\| \|(1+t)v\nabla v\|_{q_1} \big\|_{L_{p,1}(\R_+)} & \leq C \big\| \|v\|_{W^2_{q_1}} \|(1+t)\nabla v\|_{W^{1-2/p}_{q_1} \cap W^{1-2/p}_{q_2} }\big\|_{L_{p,1}(\R_+)}\\
& \leq C \|v\|_{W^{2,1}_{q_1,(p,1)}}\,\|(1+t)\nabla v\|_{L_\infty(\R_+,W^{1-2/p}_{q_1} \cap W^{1-2/p}_{q_2}  ) }\\
& \leq C \|v\|_{W^{2,1}_{q_1,(p,1)}}\, \left( \|v\|_{ W^{2,1}_{p,q_1}\cap W^{2,1}_{p,q_2}} +\|(1+t)v\|_{\dot W^{2,1}_{p,q_1}\cap \dot W^{2,1}_{p,q_2}} \right)\\
& \leq C \left( \|v\|_{W^{2,1}_{q_1,(p,1)}}^2 + \|v\|_{\Xi}^2 \right).
\end{aligned}
\end{equation}
The RHS of \eqref{est:w3} is estimated by the RHS of \eqref{est:w4}, therefore we get \eqref{est:nonlin3}. Next, similarly to \eqref{est:w2b} we have 
\begin{equation} \label{est:w2b:2}
\begin{aligned}
&\left\| \| (1+t)\, v_t\int_0^t |\nabla v|d\tau \|_{q_0}  \right\|_{L_{(p,1)}(\R_+)}
\leq C \|(1+t)\, v_t\|_{L_{(p,1)}(\R_+;L_{q_1})}\|(1+t)\nabla^2 v\|_{L_p(\R_+;L_{q_1})}\\
& \lesssim \left( \|(1+t)v\|_{\dot W^{2,1}_{q_1,(p,1)}}^2  + \|(1+t)\nabla^2 v\|_{L_p(\R_+;L_{q_1})}^2\right)
\lesssim \left( \|(1+t)v\|_{\dot W^{2,1}_{q_1,(p,1)}}^2  + \|(1+t)\nabla^2 v\|_{\Xi}^2\right)
\end{aligned}
\end{equation}
Finally,
\begin{equation}
\begin{aligned}
&\left\| \| (1+t)v\nabla v \|_{q_0} \right\|_{L_{p,1}(\R_+)}
\leq \left\| \|v\|_{q_1} \|(1+t)\nabla v\|_{q_2}  \right\|_{L_{p,1}(\R_+)}\\
&\leq \left\| \|v\|_{q_1} \|(1+t)\nabla^2 v\|_{q_1}  \right\|_{L_{p,1}(\R_+)}
\leq \|(1+t)\nabla^2 v\|_{L_{p,1}(\R_+;L_{q_1})} \|v\|_{L_\infty(\R_+;L_{q_1})}\\
&\lesssim \left( \|(1+t)v\|_{\dot W^{2,1}_{q_1,(p,1)}}^2 + \|v\|_{W^{2,1}_{q_1,p}}^2 \right) 
\lesssim \left( \|(1+t)v\|_{\dot W^{2,1}_{q_1,(p,1)}}^2 + \|v\|_{\Xi}^2 \right).
\end{aligned}
\end{equation}

\end{proof}

\subsection{Lorentz spaces estimates of the boundary terms}

In this section we prove the estimates on \eqref{lhs:5} and \eqref{lhs:6}.  
In view of \eqref{def:Hp1} and \eqref{half}, we shall estimate 
\begin{equation*}
    \|(1+t) \nabla v B(\int_0^t \nabla  v dt')\|_{ \dot H^{1/2}_{(p,1)}(\R_+;L_{q_0}) \cap \dot H^{1/2}_{(p,1)}(\R_+;L_{q_1}) } + \|\nabla v B(\int_0^t \nabla  v dt')\|_{ \dot H^{1/2}_{(p,1)}(\R_+;L_{q_2})}.   
\end{equation*}
For this purpose we need to modify of the proofs from Section 
\ref{sec:interp}.
\begin{lem} \label{l:bdry:lorentz}
Assume $S$ be given by \eqref{def:S}, where $B$ satisfies the assumptions of Lemma \ref{l:interp}. Assume moreover 
$$
f \sim t \nabla v \quad \text{and} \quad g \sim \nabla v
$$
for $v \in \Xi$, where the space $\Xi$ is defined in \eqref{sol:space}.
Then 
\begin{equation} \label{bdry:lorentz}
\begin{aligned}
    &\|Sf\|_{\dot H^{1/2}_{(p,1)}(\R_+;L_{q_0})} + \|Sf\|_{\dot H^{1/2}_{(p,1)}(\R_+;L_{q_1})} + \|Sg\|_{\dot H^{1/2}_{(p,1)}(\R_+;L_{q_2})} \\[3pt]  
    &\lesssim  \|v\|_{\Xi} \left(\|t v\|_{\dot W^{2,1}_{q_1,(p,1)}}+\|v\|_{W^{2,1}_{q_1,(p,1)}}+\|v\|_{W^{2,1}_{q_2,(p,1)} }\right)
    + \|v\|_{\Xi}^2.  
\end{aligned}
\end{equation}
\end{lem}
\begin{proof} We prove the estimate for $\|Sf\|_{\dot H^{1/2}_{(p,1)}(\R_+;L_{q_0})}$ in details, the other estimates are derived in a similar way. The proof will follow from the following interpolation rule that for suitable $p^-<p<p^+$ as
\begin{equation}
    T: \dot H^{1/2}_{p^+}(\R_+;L_{r})\cap \Xi \to \dot H^{1/2}_{p^+}(\R_+;L_{s})
\end{equation}
and 
\begin{equation}
    T: \dot H^{1/2}_{p^-}(\R_+;L_{r})\cap \Xi \to \dot H^{1/2}_{p^-}(\R_+;L_{s})
\end{equation}
 then by the real interpolation
 \begin{equation}
   T: \dot H^{1/2}_{(p,1)}(\R_+;L_{r})\cap \Xi \to \dot H^{1/2}_{(p,1)}(\R_+;L_{s})   
 \end{equation}
We need to repeat the considerations from the proof of Lemma \ref{l:interp} for $p^+$ and $p^-$. They are slightly different, let us start with $p^+$. As $\frac{1}{q_0}=\frac{1}{q_1}+\frac{1}{q_2}$, we have
\begin{equation} \label{pp:0}
    \|Sf\|_{L_{p^+}(\R_+;L_{q_0})} \leq 
    C\|g\|_{L_1(\R_+;L_{q_2})}\|f\|_{L_{p^+}(\R_+;L_{q_1})},
\end{equation}
Recalling $g \sim \nabla v$, by \eqref{L1:2} we have
$\|g\|_{L_1(\R_+;L_{q_2})} \leq \|v\|_{\Xi},$
therefore we can rewrite \eqref{pp:0} as 
\begin{equation} \label{pp:1}
    \|Sf\|_{L_{p^+}(\R_+;L_{q_0})} \lesssim 
    C\|v\|_{\Xi}\|f\|_{L_{p^+}(\R_+;L_{q_1})}.
\end{equation}
Next, 
$$
\de_t(Sf) = B(\int_0^t g) \de_t f + fgB'(\int_0^t g),
$$
therefore 
\begin{equation} \label{pp:2}
    \|\partial_t(Sf)\|_{L_{p^+}(\R_+;L_{q_0})} \leq 
    \|g\|_{L_1(\R_+;L_{q_2})}\|\partial_t f\|_{L_{p^+}(\R_+;L_{q_1})} +
    \|f g \|_{L_{p^+}(\R_+;L_{q_0})}.
\end{equation}
Now since $t v \subset \dot W^{2,1}_{q_1,p}$ for $v \in \Xi$, one estimate the last term on the RHS of \eqref{pp:2} as follows
\begin{equation} \label{pp:3}
   \|f g \|_{L_{p^+}(\R_+;L_{q_0})} \leq
   \|t g \|_{L_{p^+}(\R;L_{q_2^-})}
   \|\frac1t f\|_{L_{\infty}(\R_+;L_{q_1^+})}
\end{equation}
for some suitable $q_2^-< q_2$ and $q_1^+ >q_1$ such that $\frac{1}{q_0}=\frac{1}{q_1^+}+\frac{1}{q_2^-}$. Note that since 
\begin{equation}\label{pm:33}
    tv \in \dot W^{2,1}_{q_1,p} \subset
    \dot W^{2-2\sigma,1-\sigma}_{q_1,p^+}, 
\end{equation}
for some $\sigma>0$, we have 
\begin{equation} \label{pp:4}
\|t\nabla v\|_{L_{p^+}(\R_+;L_{q_2^-})} \leq C \|tv\|_{\dot W^{2-2\sigma,1-\sigma}_{q_1,p^+}} 
\leq C \|tv\|_{\dot W^{2,1}_{q_1,p}} \leq C \|v\|_{\Xi}  
\end{equation}
Plugging \eqref{pp:3} and \eqref{pp:4} into \eqref{pp:2} we get 
\begin{equation} \label{pp:5}
    \|\partial_t(Sf)\|_{L_{p^+}(\R_+;L_{q_0})} \leq 
    \|v\|_{\Xi}\left(\|\partial_t f\|_{L_{p^+}(\R_+;L_{q_1})} + \|\frac1t f\|_{L_{\infty}(\R_+;L_{q_1^+})}
    \right).
\end{equation}
Interpolating \eqref{pp:1} and \eqref{pp:5} we get  
\begin{equation} \label{pp:final}
    \|Sf\|_{\dot H^{1/2}_{p^+}(\R_+;L_{q_0})}\leq c
    \|g\|_{L_1(\R_+;L_{q_2})} \|f\|_{\dot H^{1/2}_{p^+}(\R_+;L_{q_1})} +
    \|v\|_{\Xi} \|f\|_{ \left(  \big(L_{p^+};L_{q_1}\big);
    \big(L_\infty(\frac{dt}{t});L_{q_1^+}\big) \right)_{1/2}}.
\end{equation}
By \eqref{complex:int} we have 
\begin{equation} \label{imbed:10}
   \left(  \big(L_{p^+};L_{q_1}\big);
    \big(L_\infty(\frac{dt}{t});L_{q_1^+}\big) \right)_{1/2} =
    L_{2p^+}(\R_+,(\frac{dt}{t^{1/2}});L_{\tilde q_1^+}),
\end{equation}
where $\frac{1}{\tilde q_1^+}=\frac{1}{2 q_1}+\frac{1}{2 q_1^+}$.
Now we need to proceed for $p^-$. Instead of \eqref{pp:1} and \eqref{pp:2} we have, respectively, 
\begin{equation} \label{pm:1}
    \|Sf\|_{L_{p^-}(\R_+;L_{q_0})} \leq 
    C\|v\|_{\Xi}\|f\|_{L_{p^-}(\R_+;L_{q_1})},
\end{equation}
and
\begin{equation} \label{pm:2}
    \|\partial_t(Sf)\|_{L_{p^-}(\R_+;L_{q_0})} \lesssim 
    \|g\|_{L_1(\R_+;L_{q_2})}\|\partial_t f\|_{L_{p^-}(\R_+;L_{q_1})} +
    \|f g \|_{L_{p^-}(\R_+;L_{q_0})}.
\end{equation}
Again, the key term is $\|fg\|_{L_{p^-}(\R_+;L_{q_0})}$. 
Let us note that as
$tv \in W^{2,1}_{q_1,p}$ for $v \in \Xi$, then
for small $\sigma >0$ we have
\begin{equation} \label{pm:3}
    t^{1-\sigma} v \in \dot W^{2,1}_{q_1,p^-}.
\end{equation}
Therefore we can write 
\begin{equation} \label{pm:4}
   \|f g \|_{L_{p^-}(\R_+;L_{q_0})} \leq
   \|t^{1-\sigma} g\|_{L_{p^-}(\R;L_{q_2})}
   \|\frac{1}{t^{1-\sigma}} f\|_{L_{\infty}(\R_+;L_{q_1})}
\end{equation}
Recalling $g \sim \nabla v$ and using \eqref{pm:3} we can rewrite 
\eqref{pm:4} as 
\begin{equation} \label{pm:5}
   \|f g \|_{L_{p^-}(\R_+;L_{q_0})} \leq
   \|v\|_{\Xi}
   \|\frac{1}{t^{1-\sigma}} f\|_{L_{\infty}(\R_+;L_{q_1})}.
\end{equation}

%

\vskip5mm

Therefore we can rewrite \eqref{pm:2} as
\begin{equation} \label{pm:6}
    \|\partial_t(Sf)\|_{L_{p^-}(\R_+;L_{q_0})} \leq 
    \|v\|_{\Xi}\|\partial_t f\|_{L_{p^-}(\R_+;L_{q_1})} +
    \|v\|_{\Xi}
   \|\frac{1}{t^{1-\sigma}} f\|_{L_{\infty}(\R_+;L_{q_1})}.
\end{equation}
Interpolating \eqref{pm:1} and \eqref{pm:6} we arrive at the bound
\begin{equation} \label{pm:final}
    \|Sf\|_{\dot H^{1/2}_{p^-}(\R_+;L_{q_0})}\leq c
    \|g\|_{L_1(R_+;L_{q_2})} \|f\|_{\dot H^{1/2}_{p^-}(\R_+;L_{q_1})} +
    \|v\|_{\Xi} \|f\|_{ \left(L_{p^-};L_{q_1});
    L_\infty(\frac{dt}{t^{1-\sigma}});L_{q_1})\right)_{1/2}}.
\end{equation}
Applying again \eqref{complex:int} we get, similarly to \eqref{imbed:10}, 
\begin{equation} \label{imbed:11}
   \left(  \big(L_{p^-};L_{q_1}\big);
    \big(L_\infty(\frac{dt}{t^{1-\sigma}});L_{q_1}\big) \right)_{1/2} =
    L_{2p^-}(\R_+,(\frac{dt}{t^{(1-\sigma)/2}});L_{q_1}).
\end{equation}

Now, in order to obtain our goal
it is enough to 
take the real interpolation for \eqref{pp:final} and \eqref{pm:final}. Taking into account \eqref{imbed:10} and \eqref{imbed:11} we arrive at
\begin{equation} \label{bdry:final}
\begin{aligned}
    \|Sf\|_{\dot H^{1/2}_{p,1}(\R_+;L_{q_0})} & \lesssim 
    \|g\|_{L_1(\R_+;L_{q_2})} \|f\|_{\dot H^{1/2}_{p,1}(\R_+;L_{q_1})}\\ 
    &+ \|v\|_{\Xi} \|f\|_{ L_{2p^+}(\R_+,(\frac{dt}{t^{1/2}});L_{\tilde q_1^+}) \cap L_{2p^-}(\R_+,(\frac{dt}{t^{(1-\sigma)/2}});L_{q_1}) }\\
    & \lesssim \|v\|_{\Xi} \|f\|_{\dot H^{1/2}_{p,1}(\R_+;L_{q_1})}\\ 
    &+ \|v\|_{\Xi} \|t\nabla v\|_{ L_{2p^+}(\R_+,(\frac{dt}{t^{1/2}});L_{\tilde q_1^+}) \cap L_{2p^-}(\R_+,(\frac{dt}{t^{(1-\sigma)/2}});L_{q_1}) }.  
\end{aligned}
\end{equation}
Now, it is enough to observe that (\ref{pm:33}) and (\ref{pm:3})
\begin{equation}
\|t\nabla v\|_{ L_{2p^+}(\R_+,(\frac{dt}{t^{1/2}});L_{\tilde q_1^+}) \cap L_{2p^-}(\R_+,(\frac{dt}{t^{(1-\sigma)/2}});L_{q_1}) } \lesssim \|v\|_{\Xi},
\end{equation}
which combined with \eqref{bdry:final} gives 
$$
\|Sf\|_{\dot H^{1/2}_{(p,1)}(\R_+;L_{q_0})} \lesssim  \|v\|_{\Xi} \|t\nabla v\|_{
    \dot H^{1/2}_{(p,1)}(\R_+;L_{q_1})}+ \|v\|_{\Xi}^2.
$$
By \eqref{imbed:4} we get 
$$
\|Sf\|_{\dot H^{1/2}_{(p,1)}(\R_+;L_{q_0})} \lesssim  \|v\|_{\Xi} \|t v\|_{\dot W^{2,1}_{q_1,(p,1)}} + \|v\|_{\Xi}^2.
$$
Estimating $\|Sf\|_{\dot H^{1/2}_{(p,1)}(\R_+;L_{q_1})}$ and $\|Sg\|_{\dot H^{1/2}_{(p,1)}(\R_+;L_{q_2})}$ in a similar way we obtain \eqref{bdry:lorentz}.
\end{proof}

\smallskip

\section{Proof of Theorem \ref{thm:Lorentz}}
In order to prove Theorem \ref{thm:Lorentz} it is enough to prove the estimate 
\begin{equation} \label{est:lorentz:final}
E_{L,+\infty}(\bW) \lesssim (E_{L,+\infty}(\bW))^2 + \|\bW\|_{\Xi}^2 + \left( \|\bU_0\|_{L_{q_{i-1}}} + \|\bU_0\|_{B^{2-2/p}_{q_i,1} } \right),
\end{equation}
where $E_{L,T}(\bW)$ is defined in \eqref{def:ETW:2}. 

Let us start with summarizing the nonlinear estimates obtained in the previous section. 
By \eqref{est:10} and \eqref{est:N4} we have 
\begin{equation} \label{Fq2:final}
\| \bF(v),\nabla \div\bG(v), \de_t\bG(v) \|_{L_{p,1}(\R_+;L_{q_2})} \lesssim \|v\|_{\Xi}^2 + \|v\|^2_{\dot W^{2,1}_{q_2,(p,1)}\cap L_{p,1}(\R_+;L_{q_1})} +\|tv\|^2_{\dot W^{2,1}_{q_1,(p,1)}},
\end{equation}
by \eqref{est:12} and \eqref{est:nonlin3} 
\begin{equation} \label{tFq1:final}
\| (1+t) \big( \bF(v), \nabla \div\bG(v),\nabla \bH(v), \de_t\bG(v)\big) \|_{L_{p,1}(\R_+;L_{q_1})} \lesssim \|v\|^2_{W^{2,1}_{q_2,(p,1)} \cap W^{2,1}_{q_1,(p,1)}}+\|tv\|^2_{\dot W^{2,1}_{q_1,(p,1)}}, 
\end{equation}
and, by \eqref{est:13} and \eqref{est:nonlin4}
\begin{equation} \label{tFq0:final}
\| (1+t) \big( \bF(v), \nabla \div\bG(v),\nabla \bH(v), \de_t\bG(v)\big) \|_{L_{p,1}(\R_+;L_{q_0})} \lesssim \|tv\|^2_{\dot W^{2,1}_{q_1,(p,1)}} + \|v\|_{\Xi}^2. 
\end{equation}
Next, we can replace the weight $t$ by $(1+t)$ in Lemma \ref{l:bdry:lorentz}. Then from \eqref{bdry:lorentz} we get
\begin{equation} \label{LH:final}
\begin{aligned}
&\sum_{i=0}^1 \|(1+t) \Lambda^{1/2}\bH(v)\|_{L_{p,1}(\R_+;L_{q_i})} 
+\|\Lambda^{1/2}\bH(v)\|_{L_{p,1}(\R_+;L_{q_2})}\\[3pt]
&\lesssim \|v\|_{\Xi} \left(\|t v\|_{\dot W^{2,1}_{q_1,(p,1)}}+\|v\|_{W^{2,1}_{q_1,(p,1)}}+\|v\|_{W^{2,1}_{q_2,(p,1)} }\right) + \|v\|_{\Xi}^2,
\end{aligned}
\end{equation}
We can now combine the above estimates with the Lorentz spaces estimates for the Stokes system to derive the final estimate which is crucial in proving Theorem \ref{thm:Lorentz}. We immediately focus on the more general perturbed half-space case. 
Summing up \eqref{Stokes:Lorentz:pert} for $q_1$ and $q_2$ we get 
\begin{equation}
\begin{aligned}
&\sum_{j=1}^2 \|\bW\|_{W^{2,1}_{q_j,(p,1)}}\\ 
&\lesssim 
\|\bF(\bW), \nabla \div\bG(\bW),\nabla \bH(\bW), \de_t\bG(\bW), \Lambda^{1/2}\bH\|_{L_{p,1}(\R_+;L_{q_2})}\\[3pt] 
&+\sum_{j=1}^2 \|(1+t)\big(\bF(\bW), \nabla \div\bG(\bW),\nabla \bH(\bW), \de_t\bG(\bW), \Lambda^{1/2}\bH \big)\|_{L_{p,1}(\R_+;L_{q_{j-1}})}\\[3pt]
&+ \sum_{j=1}^2 \left( \|\bU_0\|_{L_{q_{i-1}}} + \|\bU_0\|_{B^{2-2/p}_{q_i,1} } \right)\\[3pt]
&+\epsilon \Big( \|(1+t)\nabla^2 \bW\|_{L_{p,1}(\BR_+, L_{q_1}(\Omega_0))} + \|\nabla^2 \bW\|_{L_{p,1}(\BR_+, L_{q_2}(\Omega_0))} \Big).
\end{aligned}
\end{equation}
Applying \eqref{Fq2:final}-\eqref{LH:final} to the RHS we obtain  
\begin{equation} \label{nonlin:final}
\begin{aligned}
\sum_{j=1}^2 \|\bW\|_{W^{2,1}_{q_j,(p,1)}} \lesssim & 
\sum_{j=1}^2 \left( \|\bU_0\|_{L_{q_{i-1}}} + \|\bU_0\|_{B^{2-2/p}_{q_i,1} } \right)\\[3pt]
&+\sum_{j=1}^2 \|\bW\|^2_{W^{2,1}_{q_j,(p,1)}} + \|(1+t)\bW\|^2_{\dot W^{2,1}_{q_1,(p,1)}} + \|\bW\|_{\Xi}^2\\[3pt]
&+\|\bW\|_{\Xi} \left(\|t \bW\|_{\dot W^{2,1}_{q_1,(p,1)}}+\|\bW\|_{W^{2,1}_{q_1,(p,1)}}+\|\bW\|_{W^{2,1}_{q_2,(p,1)} }\right). 
\end{aligned}
\end{equation}
Next, by \eqref{weight:lorentz:pert} we have 
\begin{equation} 
\begin{aligned}
&\|(1+t)(\de_t \bW,\nabla^2 \bW)\|_{L_{p,1}(\R_+;L_{q_1})}\\[3pt]
&\lesssim \|(1+t)\big(\bF(\bW), \nabla \div\bG(\bW),\nabla \bH(\bW), \de_t\bG(\bW),\Lambda^{1/2}\bH(\bW)\big)\|_{L_{p,1}(\R_+;L_{q_1})}\\[3pt]
& +\|\bU_0\|_{\dot B^{2(1-1/p)}_{q_1,1}} +\|\bW\|_{L_{p_1}(\R_+;L_{q_1})} + \epsilon \|(1+t)(\de_t \bW,\nabla^2 \bW)\|_{L_{p,1}(\R_+;L_{q_1})}
\end{aligned}
\end{equation}
Applying \eqref{tFq1:final} to the LHS we get 
\begin{equation} \label{weight:nonlin:final}
\begin{aligned}
&\|(1+t)(\de_t \bW,\nabla^2 \bW)\|_{L_{p,1}(\R_+;L_{q_1})} \\[3pt] &\lesssim 
\|\bW\|_{L_{p_1}(\R_+;L_{q_1})} + \sum_{j=1}^2 \|\bW\|^2_{W^{2,1}_{q_j,(p,1)}} + \|(1+t)\bW\|^2_{\dot W^{2,1}_{q_1,(p,1)}}\\[3pt]
& +\|\bU_0\|_{\dot B^{2(1-1/p)}_{q_1,1}}+\|\bW\|_{\Xi} \left(\|t \bW\|_{\dot W^{2,1}_{q_1,(p,1)}}+\|\bW\|_{W^{2,1}_{q_1,(p,1)}}+\|\bW\|_{W^{2,1}_{q_2,(p,1)} }\right)
+\|\bW\|_{\Xi}^2. 
\end{aligned}
\end{equation}
Combining \eqref{nonlin:final} and \eqref{weight:nonlin:final}
we get \eqref{est:lorentz:final}.


\subsubsection*{Acknowledgement}
The first (PBM) and second authors (TP) has been partly supported by the Narodowe Centrum Nauki (NCN) Grant No. 2022/45/B/ST1/03432 (OPUS).

\appendix

\section{Recasting the system in Lagrangian coordinates}
\label{sec-A}
For the reader's convenience, we provide here how to derive
\eqref{def-nonlinear-terms}. To this end, we use the following 
well-known formulas:
\begin{alignat}2
\nabla_x & = \BA_{\bu}^\top \nabla_y, & \qquad 
\dv_x (\,\cdot\,) & = \BA_{\bu}^\top \colon \nabla_y (\,\cdot\,) 
= \dv_y (\BA_{\bu} \,\cdot \, ), \\
\bn & = \frac{\BA_{\bu}^\top \bn_0}{\lvert \BA_{\bu}^\top \bn_0 \rvert}, 
& \qquad
\nabla_x \dv_x (\, \cdot \,) & = \BA_{\bu}^\top \nabla_y \dv_y (\, \cdot \,)
+ \BA_{\bu}^\top \nabla_y ((\BA_{\bu}^\top - \BI) \colon \nabla_y \, \cdot \, ),
\end{alignat}
{\color{black}see, e.g., \cite[p. 383]{Sol88}.}
{\color{black} In fact,} as it was proved in \cite[{\color{black} p. 382}]{Sol88}, 
there holds $\det \BA_{\bu} = 1$ as follows from the divergence-free condition, 
which yields the first formula $\nabla_x = \BA_{\bu}^\top \nabla_y$.
By using these formulas, it is easy to verify the representations of
$G_\mathrm{div} (\bu)$ and $\bG (\bu)$. Hence, it suffices to derive
the representations of $\bF (\bu)$ and $\BH (\bu)$.  \par
By a direct calculation, we observe
\begin{equation}
\DV_x(\mu\BD(\bv) - P\BI) = \mu \Delta_x \bv 
+ \mu \nabla_x \dv_x \bv - \nabla_x P.
\end{equation}
We see that
\begin{align}
&\pd_t \bv + (\bv \cdot \nabla_x) \bv = \pd_t \bu, \\
&\Delta_x \bv = \dv_x\nabla_x \bv 
= \dv_y(\BA_{\bu} \BA_{\bu}^\top\nabla_y \bu)
= \dv_y((\BA_{\bu} \BA_{\bu}^\top-\BI)\nabla_y \bu) + \Delta_y \bu, \\
&\nabla_x\dv_x\bv = \BA_{\bu}^\top\nabla_y (\BA_{\bu}^\top:\nabla_y\bu)
= \BA_{\bu}^\top\nabla_y((\BA_{\bu}^\top-\BI):\nabla_y\bu)+
\BA_{\bu}^\top\nabla_y\dv_y\bu, \\
&\nabla_x P = \BA_{\bu}^\top\nabla_y \sfQ.
\end{align}
Since $\BA_{\bu}^\top$ is invertible and $(\BA_{\bu}^\top)^{-1}
= \BI + \int^t_0\nabla \bu\d\tau$, 
the equation
\eqref{eq-original}$_1$ is transformed into
\begin{align}
& \pd_t \bu - \mu \Delta_y \bu - \mu \nabla_y \dv_y \bu
+ \nabla_y \sfQ \\
& = \bigg(\int^t_0\nabla\bu\d\tau\bigg)
\Big(\pd_t \bu - \mu \Delta_y \bu\Big) 
+ \mu \bigg(\BI + \int^t_0\nabla\bu\d\tau\bigg)
\dv_y \Big( (\BA_{\bu} \BA_{\bu}^\top - \BI) \nabla_y \bu \Big)\\
&+ \mu \nabla_y \Big((\BA_{\bu}^\top - \BI) \colon \nabla_y \bu \Big).	
\end{align}
Combined with 
\begin{equation}
\DV_y (\mu \BD(\bu) - \sfQ \BI) 
= \mu \Delta_y \bu + \mu \nabla_y \dv_y \bu - \nabla_y \sfQ,
\end{equation}
we have the representation of $\bF (\bu)$. Note that $\bF (\bu)$
does not contain the pressure $\sfQ$. \par
It remains to deal with $\BH (\bu)$. It is easy to find that
\begin{equation}
\mu\BD(\bv) -  P\BI = \mu \Big(\BA_{\bu}^\top \nabla_y \bu 
+ [\nabla \bu]^\top \BA_{\bu} \Big) - \sfQ \BI.
\end{equation}
Together with the boundary condition \eqref{eq-original}$_3$, it follows that
\begin{equation}
\mu \Big(\BA_{\bu}^\top \nabla_y \bu 
+ [\nabla \bu]^\top \BA_{\bu} \Big) 
\frac{\BA_{\bu}^\top \bn_0}{\lvert \BA_{\bu}^\top \bn_0 \rvert}
- \sfQ \frac{\BA_{\bu}^\top \bn_0}{\lvert \BA_{\bu}^\top \bn_0 \rvert} = 0
\qquad \text{on $\pd \HS$}.
\end{equation}
Multiplying this equation by $\lvert \BA_{\bu}^\top \bn_0 
\rvert (\BA_{\bu}^\top)^{- 1}$ yields
\begin{equation}
\mu \Big(\nabla_y \bu + (\BA_{\bu}^\top)^{- 1} 
[\nabla \bu]^\top \BA_{\bu} \Big) 
\BA_{\bu}^\top \bn_0 - \sfQ \bn_0 = 0
\qquad \text{on $\pd \HS$}.
\end{equation}
Namely, we have
\begin{align*}
&(\mu \BD(\bu) - \sfQ \BI) \bn_0\\
&= \mu \Big[ \Big(\nabla_y \bu + (\BA_{\bu}^\top)^{- 1}
[\nabla_y \bu]^\top \BA_{\bu} \Big) (\BI-\BA_{\bu}^\top)
+ (\BI - (\BA_{\bu}^\top)^{- 1}) [\nabla_y \bu]^\top \BA_{\bu}
+ [\nabla_y \bu]^\top (\BI-\BA_{\bu}) \Big]\bn_0.
\end{align*}
Since there holds
$(\BA_{\bu}^\top)^{- 1} = \BI + (\int_0^t \nabla_y \bu \d \tau )^\top$,
we obtain the representation of $\BH (\bu)$.


\section{Spectral theory - proof of Theorem \ref{thm:max}} \label{A1}

We consider the generalized resolvent problem:
\begin{equation}\label{eq:2}\left\{\begin{aligned}
\lambda \bu - \dv\BT(\bu, \fp) = \bff&&\quad&\text{in $\HS$}, \\
\dv\bu = g = \dv \bg&&\quad&\text{in $\HS$}, \\
\BT(\bu, \fp)\bn_0 = \bh&&\quad&\text{on $\pd\HS$}.
\end{aligned}\right.\end{equation}
Let 
$$\Sigma_\epsilon = \{\lambda \in \BC\setminus\{0\} \mid |\arg\lambda|\leq \pi-\epsilon\}
\quad\epsilon \in (0, \pi/2).$$
In what follows,$\epsilon \in (0, \pi/2)$ is fixed. 
Then, from \cite{SS12} we know that there exist $\CR$-bounded solution 
operators $\CS(\lambda)$ and $\CP(\lambda)$ such that 
$$\CS(\lambda) \in {\rm Hol}\,(\Sigma_\epsilon, \CL(L_q(\HS)^m, W^2_q(\HS)^d)),
\quad
\CP(\lambda) \in {\rm Hol}\,(\Sigma_\epsilon, \CL(L_q(\HS)^m, 
W^1_q(\HS) + \hat W^1_{q,0}(\HS))
$$
satisfying the following two properties:
\begin{enumerate}
\item $\bu = \CS(\lambda)F$ and $\fp = \CP(\lambda)F$ are unique solutions 
of equations \eqref{eq:2}, where 
$$F=(\bff, \nabla g, \lambda \bg, \nabla \bh, \lambda^{1/2}\bh).$$
\item there exists an $r_b > 0$ such that 
\begin{gather*}\CR_{\CL(L_q(\HS)^m, L_q(\HS)^n)}(\{(\tau\pd_\tau)^\ell(\lambda\CS(\lambda),
 \lambda^{1/2}\nabla\CS(\lambda), \nabla^2\CS(\lambda)) \mid
\lambda \in \Sigma_\epsilon\}) \leq r_b, \\
\CR_{\CL(L_q(\HS)^m, L_q(\HS)^d)}
(\{(\tau\pd_\tau)^\ell \nabla \CP(\lambda) \mid 
\lambda \in \Sigma_\epsilon\} \leq r_b.
\end{gather*}
for $\ell=0,1$.  Here, $m$ and $n$ are suitable size of vectors. 
\end{enumerate}
In particular, we have the estimates:
\begin{equation}\label{bound:1}
\|(\lambda, \lambda^{1/2}\nabla, \nabla^2)\CS(\lambda)F\|_{L_q(\HS)} \leq C\|F\|_{L_q(\HS)}.
\end{equation}
The definition of $\CR$ boundedness will be given in Appendix below. 
\par 
To obtain solution formulas, first we consider the boundary value problem: 
\begin{equation}\label{eq:3}\left\{\begin{aligned}
\pd_t\bU - \dv\BT(\bU, P) = \bF&&\quad&\text{in $\HS\times\BR$}, \\
\dv\bU = G = \dv\bG&&\quad&\text{in $\HS\times\BR$}, \\
\BT(\bU, P) \bn_0=\bH &&\quad&\text{on $\pd\HS\times\BR$}.
\end{aligned}\right.\end{equation}
Let $1 < p, q < \infty$. 
 If $\bF$, $G$, $\bG$, $\bH$ are defined on the whole interval
$\BR$ and satisfy the following regularity conditions:
\begin{align*}
\bF &\in L_p(\BR, L_q(\HS)^d), \quad 
G \in L_p(\BR, \dW^1_q(\HS)), \quad \bG \in \dH^1_p(\BR, L_q(\HS)^d), \\
 \bH &\in L_p(\BR, \dW^1_q(\HS)^d) \cap \dH^{1/2}_p(\BR, L_q(\HS)^d)
\end{align*}
then problem \eqref{eq:3} admit unique solutions $\bU$ and $P$ such that 
\begin{gather*}
(\pd_t, \nabla^2)\bU \in L_p(\BR, L_q(\HS)^{d+d^2}), 
\quad \nabla P \in L_p(\BR, L_q(\HS)^d)
\end{gather*}
possessing the estimate:
\begin{equation}\label{fourier.multiplier}
\|(\pd_t\bU, \nabla^2\bU, \nabla P)\|_{L_p(\BR, L_q(\HS)}
\leq Cr_b\|(\bF, \nabla G, \pd_t\bG, \nabla \bH, \Lambda^{1/2}\bH)\|_{L_p(\BR, L_q(\HS))}
\end{equation}
for some constant $C$ depending solely on $p$ and $q$, which follows from 
Weis's operator valued Fourier multiplier theorem \cite{Weis}. 
Here, $\Lambda^{1/2}\bH$ is defined in \eqref{def:lambda}. 
%

We next consider the initial valued problem:
\begin{equation}\label{eq:4}\left\{\begin{aligned}
\pd_t\bV - \dv\BT(\bV, Q) = 0&&\quad&\text{in $\HS\times\BR_+$}, \\
\dv\bV = 0&&\quad&\text{in $\HS\times\BR_+$}, \\
\BT(\bV, Q) \bn_0=0 &&\quad&\text{on $\pd\HS\times\BR_+$}, \\
\bV|_{t=0} = \bg&&\quad&\text{in $\HS$}.
\end{aligned}\right.\end{equation}
To formulate \eqref{eq:4} in the analytic semigroup setting, we consider
the corresponding resolvent problem:
\begin{equation}\label{eq:5}
\left\{\begin{aligned}
\lambda\bV - \dv\BT(\bV, \fq) = \bg&&\quad&\text{in $\HS\times\BR_+$}, \\
\dv\bV = 0&&\quad&\text{in $\HS\times\BR_+$}, \\
\BT(\bV, \fq) \bn_0=0 &&\quad&\text{on $\pd\HS\times\BR_+$}. 
\end{aligned}\right.\end{equation}
To eliminate the pressure term $\fq$,  we introduce
the functional $K$.  For any $\bV \in \dW^2_q(\HS)^N$, let $u = K(\bu)
\in  W^1_q(\HS) + \hat W^1_{q,0}(\HS)$ be a unique solution of the 
weak Dirichlet problem:
\begin{equation} \label{def:K}
(\nabla u, \nabla\varphi) = (\mu\dv \BD(\bu), \nabla\varphi)
\quad\text{for every $\varphi \in \hat W^1_{q', 0}(\HS)$}, \\
\end{equation}
subject to $u|_{x_d=0} = \BD(\bV)\bn_0|_{x_d=0}$. 
And also, we introduce the solenoidal space $J^q(\HS)$, which is defined by
$$J_q(\HS) = \{\bu \in L_q(\HS)^d \mid (\bu, \nabla\varphi) = 0
\quad\text{for every $\varphi \in \hat W^1_{q',0}$}\}.$$
In the half space case, we know the unique existence of solutions to the weak
Dirichlet problem:
\begin{equation}\label{wd:1}(\nabla u, \nabla\varphi) = (\bff, \nabla\varphi)
\quad\text{for every $\varphi \in \hat W^1_{q',0}(\HS)$}.
\end{equation} 
Namely, for every $\bff \in L_q(\HS)^d$, problem \eqref{wd:1}
admits  a unique solution 
$u \in \hat W^1_{q,0}(\HS)$ satisfying the estimate:
\begin{equation}\label{wd:2}
\|\nabla u\|_{L_q(\HS)} \leq C\|\bff\|_{L_q(\HS)}.
\end{equation}
By using this fact, for any $\bff \in L_q(\HS)^d$, there exists a unique $\bg
\in J_q(\HS)$ and $h \in \hat W^1_{q,0}(\HS)$ such that 
$\bff = \bg + \nabla h$ and 
$$\|\bg\|_{L_q(\HS)} + \|\nabla h\|_{L_q(\HS)} \leq C\|\bff\|_{L_q(\HS)}.$$
This is called the second Helmholtz decomposition. 
\par
Let us consider equations:
\begin{equation}\label{eq:6}
\left\{\begin{aligned}
\lambda\bv - \dv\BT(\bv, K(\bv)) = \bg&&\quad&\text{in $\HS$}, \\
\dv\bv = 0&&\quad&\text{in $\HS$}, \\
\BT(\bv, K(\bv)) \bn_0=0 &&\quad&\text{on $\pd\HS$}. 
\end{aligned}\right.\end{equation}
From the definition of $K(\bv)$, the requirement of boundary conditions are:
$$\pd_d v_j + \pd_j v_d = 0 \quad \text{on $\pd\HS$, \quad $j=1, \ldots, d-1$}.
$$
We see easily that if $\bff \in J_q(\HS)$, then $\fq = K(\bv)$, and so
instead of equations \eqref{eq:5}, we consider \eqref{eq:6}.
And then, 
\begin{equation}\label{solformula:1}
\bv = \CS(\lambda)(\bg, 0, \cdots, 0), \quad 
K(\bv) = \CP(\lambda)(\bg, 0, \cdots, 0).
\end{equation}  \par 
Moreover,  the $\CR$ boundedness implies that
\begin{equation}\label{resol:est}|\lambda|\|\bv\|_{L_q(\HS)} + \|\nabla^2\bv\|_{L_q(\HS)} 
\leq C\|\bg\|_{L_q(\HS)}.
\end{equation}
Let us define operator $\CA_q$ and its domain $\CD(\CA_q)$ by
\begin{equation}\label{domain:1}\begin{aligned}
\CD(\CA_q) &= \{\bu \in J_q(\HS) \cap \dW^2_q(\HS) \mid 
(\pd_d u_j + \pd_ju_d)|_{x_d=0}= 0 \enskip(j=1, \ldots, d-1)\}\\
\CA_q \bv & = \dv \BT(\bv, K(\bv))\quad\text{for $\bv \in \CD(\CA_q)$}.
\end{aligned}\end{equation} 
Then, by \eqref{solformula:1}, $(\lambda-\CA_q)^{-1}\bg
= \CS(\lambda)(\bg, 0, \cdots, 0)$ for $\lambda \in \Sigma_\epsilon$
and resolvent estimate \eqref{resol:est} holds for any
$\lambda \in \Sigma_\epsilon$.
Thus, by the theory of $C_0$ analytic semigroup, we know that 
$\CA_q$ generates a $C_0$ analytic semigroup $\{T(t)\}_{t\geq0}$
and there hold
\begin{equation}\label{initial:est}\begin{aligned}
\|(\pd_tT(t)\bg, \nabla^2T(t)\bg)\|_{L_q(\HS))} &\leq Ct^{-1}\|\bg\|_{L_q(\HS)}
&\quad&\text{for $\bg \in J_q(\HS)$},\\
\|(\pd_tT(t)\bg, \nabla^2T(t)\bg)\|_{L_q(\HS))} &\leq C\|\bg\|_{\dot W^2_q(\HS)}
&\quad&\text{for $\bg \in \CD(\CA_q)$}.
\end{aligned}\end{equation}

To prove these facts, following theory of semigroup we represent $T(t)\bg$.  Let 
$\Gamma = \Gamma_+ \cup \Gamma_-$ be a contour in $\BC$ defined by
$$\Gamma_\pm = \{ \lambda = re^{i\pm(\pi-\epsilon)} \mid r \in (0, \infty)\}.
$$
For more rigorous argument, we should take 
$C_\sigma =\{\lambda =\sigma e^{i\theta} \mid \theta \in (-(\pi-\epsilon), \pi-\epsilon)\}$
and we consider $\sigma\to 0$, but we omit this procedure. According to
 theory of $C_0$ analytic semigroup, from \eqref{solformula:1} we see that 
$$T(t)\bg = \begin{cases} 0 & \quad \text{for $t < 0$}, \\[0.5pc]
\dfrac{1}{2\pi}\displaystyle{\int_{\Gamma} e^{\lambda t}
\CS(\lambda)(\bg, 0,\ldots, 0)\,\d\lambda} &\quad
\text{for $t>0$}.
\end{cases}
$$
Then, by $\CS(\lambda)(\bg, 0, \cdots, 0) = (\lambda- \CA_q)^{-1}\bg$ 
and the known argument in  theory of $C^0$ analytic semigroup, we have
\eqref{initial:est}. 
Thus, by real interpolation method with the help of \eqref{initial:est}, we have 
\begin{equation}\label{initial:est*}
\|(\pd_t, \nabla^2)T(t)\bg\|_{L_p(\BR_+, L_q(\HS))} \leq C\|\bg\|_{\dot B^{2(1-1/p)}_{q,p}(\HS)}.
\end{equation}
for $1 < p < \infty$.

For $1 < p, q < \infty$, we know that $BUC(\BR_+, \dot B^{2(1-1/p)}_{q,p}(\HS))
\supset \dot W^1_p(\BR_+, L_q(\HS)) \cap L_p(\BR_+, \dot W^2_q(\HS))$, where the inclusion
is continuous. Thus, we have proved \eqref{lin:global}.

To obtain the local estimate \eqref{local.2}, let $T >0$. 
Recalling the assumption \eqref{ic:zero}, for $f$ defined on $(0, \infty)$ with $f(0)=0$, let $e_T[f]$ be an extension of $f$ defined by
$$e_T[f]= \left \{ \begin{matrix} 0 &\quad t<0, \\
f(t) &\quad 0 < t < T, \\
f(2T-t)&\quad T < t < 2T, \\
0&\quad t>2T.
\end{matrix} \right. $$
In particular, we have
$$\pd_te_T[f]= \left \{ \begin{matrix} 0 &\quad t<0, \\
\pd_tf(t) &\quad 0 < t < T, \\
-(\pd_tf)(2T-t)&\quad T < t < 2T, \\
0&\quad t>2T.
\end{matrix} \right. $$
Thus, we have
\begin{align*}
&\|(e_T[F], \nabla e_T[G], \pd_t e_T[\bG], \Lambda^{1/2}
e_T[\bH],\nabla e_T[\bH])\|_{L_p(\BR, L_q(\HS))} \\
&\quad \leq C(\|(F, \nabla G, \pd_t \bG, \nabla \bH)\|_{L_p((0, T), L_q(\HS))}
+ \|\bH\|_{\dot H^{1/2}_p((0, T), L_q(\HS))}.
\end{align*}
Since our problem generates a Stokes semigroup, the uniqueness holds locally in time.
Thus, we have
\begin{equation}\label{local.1}\begin{aligned}
&\|(\pd_t, \nabla^2) \bW\|_{L_p((0, T), L_q(\HS))}
\leq  C(\|\bU_0-\bU(0)\|_{B^{2(1-1/p)}_{q,p}(\HS)} \\
&\qquad +\|(F, \nabla G, \pd_t \bG, \nabla \bH)\|_{L_p((0, T), L_q(\HS))}
+ \|\bH\|_{\dot H^{1/2}_p((0, T), L_q(\HS))},
\end{aligned}\end{equation}
where $\bU$ solves \eqref{eq:3}.
By the trace theorem, \eqref{fourier.multiplier} and \eqref{half}
we have
\begin{align*}
\|\bU(0)\|_{B^{2(1-1/p)}_{q,p}(\HS)} 
&\leq C(\|\bU\|_{L_p(\BR_+, W^2_q(\HS)} + \|\pd_t\bU\|_{L_p(\BR_+, L_q(\HS))}) \\
&\leq C\|(\bF, \nabla G, \pd_t \bG, \nabla \bH, \Lambda^{1/2}\bH)\|_{L_p(\BR, L_q(\HS))}\\
&\leq C\|(\bF, \nabla G, \pd_t \bG, \nabla \bH\|_{L_p(\BR, L_q(\HS))} + \|\bH\|_{\dot H^{1/2}_p((0, T), L_q(\HS))}.
\end{align*}
Thus, from \eqref{local.1} we can localize the estimate in time and obtain \eqref{local.2}. The proof of Theorem \ref{thm:max} is complete.

\end{document}